\documentclass{ifcolog}
\usepackage{amsmath,amssymb,amsthm,units,relsize,stmaryrd,
upgreek,leftidx,turnstile,multicol, enumerate,authblk,lipsum,graphics}
\usepackage[inline]{enumitem}
\usepackage{yfonts,bm}
\usepackage{units,stmaryrd}%not jsl
\usepackage{color}
\usepackage{graphicx}
\usepackage[all]{xy}
\newtheorem{theorem}{Theorem}[section]
\newtheorem{conj}[theorem]{Conjecture}
\newtheorem{definition}[theorem]{Definition}
\newtheorem{lemma}[theorem]{Lemma}

\newtheorem{corollary}[theorem]{Corollary}
\newtheorem{proposition}[theorem]{Proposition}

\newcommand{\aX}{{\tt X}}
\newcommand{\aY}{{\tt Y}}
\newcommand{\aZ}{{\tt Z}}
\newcommand{\au}{{\tt u}}
\newcommand{\av}{{\tt v}}
\newcommand{\aw}{{\tt w}}
\newcommand{\ocl}[1]{\iter o{#1}}
\newcommand{\ecl}[1]{\iter{{\rm HE}}{#1}}
\newcommand{\iter}[2]{{#1}[{#2}]}
\newcommand{\close}[2]{\overline{#2}^{#1}}
\newcommand{\nrmone}[1]{\left\| #1 \right\|_1}
\newcommand{\lgt}[1]{\# #1}
\newcommand{\cindex}[1]{\left\lfloor{#1}\right\rfloor}
\newcommand{\lan}[1]{{\mathcal L}_{#1}}
\newcommand{\KP}{{\rm KP}}
\newcommand{\SpAut}{{\sf S}}
\newcommand{\WAut}{ {\sf W}}
\newcommand{\Robinson}{\ensuremath{{\mathrm{Q}^+}}\xspace}

\newcommand{\sharpen}[1]{\sharp{#1}}
\newcommand{\card}[1]{\left|{#1}\right|}

\newcommand\pica{{ \Pi}^1_1\mbox{-}{\rm CA}_0}

\newcommand\InfPro[1]{[\infty]_{#1}}

\newcommand\spc[2]{{\tt SPC}_{#1}({#2})}
\newcommand\PsAlt\blacklozenge
\newcommand\NcAlt\blacksquare

\newcommand\yespr[1]{}

\newcommand\spiderfun[1]{ {#1}  ^\cond_O}

\def\eca{\ensuremath{{{\rm ECA}_0}}\xspace}
\def\aca{{{\rm ACA}_0}}
\def\rca{{{\rm RCA}_0}}

\def\atr{{{\rm ATR}_0}}

\def\compax{{\tt CA}}
\def\provfor{{\tt Proof}}
\def\Provfor{{\tt IPC}^{\Lambda}}

\newcommand{\spidersub}{\sqsubset_\Upomega}
\newcommand{\Spiders}{{\mathbb S}}
\newcommand{\flatten}[1]{\flat{#1}}
\newcommand{\proves}{\vdash}
\newcommand{\rcto}{\Rightarrow}
\newcommand{\Bmap}[1]{{#1}^o}

\newcommand{\Cset}{C }
\newcommand{\newpsi}{\uppsi }
\newcommand{\fu}{{\mathfrak u}}
\newcommand{\fv}{{\mathfrak v}}
\newcommand{\fw}{{\mathfrak w}}

\newcommand{\fino}{\acute o}
\newcommand{\trano}{\hat o}
\newcommand{\cnorm}[1]{\|{#1}\|}
\newcommand{\scleq}{\leq}
\newcommand{\scle}{<}
\newcommand{\wge}[1]{\mathrel\rhd_{#1}}
\newcommand{\wgeq}[1]{\mathrel\unrhd_{#1}}
\newcommand{\wle}[1]{\mathrel\lhd_{#1}}
\newcommand{\wleq}[1]{\mathrel\unlhd_{#1}}

\newcommand{\promote}{\uparrow}

\newcommand{\glp}{{\ensuremath{\mathsf{GLP}}}\xspace}
\newcommand{\RC}{{\ensuremath{\mathsf{RC}}}\xspace}

\usepackage{xspace}

\newcommand{\pa}{\ensuremath{{\mathrm{PA}}}\xspace}

\newcommand{\ea}{\ensuremath{{\rm{EA}}}\xspace}

\newcommand{\ord}{{\sf Ord}}

\newcommand\Om[1]{\Upomega({#1})}

\def\compax{{\tt CA}}

\def\provfor{{\tt Proof}}
\def\Provfor{{\tt IPC}}

\def\lb{\left\llbracket}
\def\rb{\right\rrbracket}

\def\<{\langle}
\def\>{\rangle}

\def\sx{{\mathfrak X}}
\def\sy{{\mathfrak Y}}
\def\sz{{\mathfrak Z}}
\def\su{{\mathfrak U}}
\def\sv{{\mathfrak V}}
\def\sw{{\mathfrak W}}

\newcommand\cond\uppi

\newcommand{\provx}[2]{[{#1}]_{#2}}

\usepackage{xspace}

\newcommand{\Worms}{{\mathbb W}}

\newcommand{\iprovx}[3]{\textstyle{#1\brack #2}_{#3}}
\newcommand{\iconsx}[3]{\textstyle{#1\bangle #2}_{#3}}

\newcommand{\term}[2]{{#1\termop #2}}

\def\lb{\left\llbracket}
\def\rb{\right\rrbracket}
\def \bangle{ \atopwithdelims \langle \rangle}
\def \termop{ \atopwithdelims \lgroup \rgroup}
\def\nc{{\Box}}
\def\ps{{\Diamond}}

\def\alang{{\mathcal L}_\in}

\def\peq{\preccurlyeq}

\title{Worms and Spiders:\\
Reflection Calculi and Ordinal Notation Systems
      }

\titlerunning{Worms and Spiders}

\addauthor[david.fernandez@irit.fr]{David Fern\'andez-Duque}{Institute de Recherche en Informatique de Toulouse, Toulouse University, France.\\
Department of Mathematics, Ghent University, Belgium.}

\titlethanks{This work was partially funded by ANR-11-LABX-0040-CIMI within the program ANR-11-IDEX-0002-02.}

\authorrunning{D. Fern\'andez-Duque}

\begin{document}
\maketitle%not jsl

\begin{center}
{\it To the memory of Professor Grigori Mints.}
\end{center}

\begin{abstract}
We give a general overview of ordinal notation systems arising from reflection calculi, and extend the to represent impredicative ordinals up to those representable using Buchholz-style collapsing functions.
\end{abstract}

\section{Introduction}

I had the honor of receiving the G\"odel Centenary Research Prize in 2008 based on work directed by my doctoral advisor, Grigori {`Grisha'} Mints. The topic of my dissertation was {\em dynamic topological logic,} and while this remains a research interest of mine, in recent years I have focused on studying polymodal provability logics. These logics have proof-theoretic applications and give rise to ordinal notation systems, although previously only for ordinals below the Feferman-Sh\"utte ordinal, $\Gamma_0$. I last saw Professor Mints in the {\em First International Wormshop} in 2012, where he asked if we could represent the Bachmann-Howard ordinal, $\uppsi(\varepsilon_{\Omega+1})$, using provability logics. It seems fitting for this volume to once again write about a problem posed to me by Professor Mints.

Notation systems for $\uppsi(\varepsilon_{\Omega+1})$ and other `impredicative' ordinals are a natural step in advancing Beklemishev's $\Pi^0_1$ ordinal analysis\footnote{The $\Pi^0_1$ ordinal of a theory $T$ is a way to measure its `consistency strength'. A different measure, more widely studied, is its $\Pi^1_1$ ordinal; we will not define either in this work, but the interested reader may find details in \cite{Beklemishev:2004:ProvabilityAlgebrasAndOrdinals} and \cite{Pohlers:2009:PTBook}, respectively.} to relatively strong theories of second-order arithmetic, as well as systems based on Kripke-Platek set theory. Indeed, Professor Mints was not the only participant of the Wormshop interested in representing impredicative ordinals within provability algebras. Fedor Pakhomov brought up the same question, and we had many discussions on the topic. At the time, we each came up with a different strategy for addressing it. These discussions inspired me to continue reflecting about the problem the next couple of years, eventually leading to the ideas presented in the latter part of this manuscript.

\subsection{Background}

The G\"odel-L\"ob logic $\sf GL$ is a modal logic in which $\nc\varphi$ is interpreted as {\em`$\varphi$ is derivable in $T$',} where $T$ is some fixed formal theory such as Peano arithmetic. This may be extended to a polymodal logic $\glp_\omega$ with one modality $[n]$ for each natural number $n$, as proposed by Japaridze \cite{Japaridze:1988}. The modalities $[n]$ may be given a natural proof-theoretic interpretation by extending $T$ with new axioms or infinitary rules. However, $\glp_\omega$ is not an easy modal logic to work with, and to this end Dashkov \cite{Dashkov:2012:PositiveFragment} and Beklemishev \cite{Beklemishev:2013:PositiveProvabilityLogic, Beklemishev:2012:CalibratingProvabilityLogic} have identified a particularly well-behaved fragment called the {\em reflection calculus} ($\RC$), which contains the dual modalities $\langle n\rangle$, but does not allow one to define $[n]$.

Because of this, when working within $\RC$, we may simply write $n$ instead of $\langle n\rangle$. With this notational convention in mind, of particular interest are {\em worms,} which are expressions of the form
\[m_1 \hdots m_n \top,\]
which can be read as
\begin{quote}
{\em It is $m_1$-consistent with $T$ that it is $m_2$ consistent with $T$ that $\hdots$ that $T$ is $m_n$-consistent.}
\end{quote}
In \cite{Ignatiev:1993:StrongProvabilityPredicates}, Ignatiev proved that the set of worms of $\glp_\omega$ is well-ordered by consistency strength and computed their order-type. Beklemishev has since shown that trasfinite induction along this well-order may be used to give an otherwise finitary proof of the consistency of Peano arithmetic \cite{Beklemishev:2004:ProvabilityAlgebrasAndOrdinals}.

Indeed, the order-type of the set of worms in $\RC_\omega$ is $\varepsilon_0$, an ordinal which already appeared in Gentzen's earlier proof of the consistency of $\pa$ \cite{Gentzen1936}. Moreover, as Beklemishev has observed \cite{Beklemishev:2005:VeblenInGLP}, worms remain well-ordered if we instead work in $\RC_\Lambda$ (or $\glp_\Lambda$), where $\Lambda$ is an arbitrary ordinal. The worms of $\RC_\Lambda$ give a notation system up to the Feferman-Sch\"utte ordinal $\Gamma_0$, considered the upper bound of predicative mathematics.

This suggests that techniques based on reflection calculi may be used to give a proof-theoretic analysis of theories of strength $\Gamma_0$, the focus of an ongoing research project. However, if worms only provide notations for ordinals below $\Gamma_0$, then these techniques cannot be applied to `impredicative' theories, such as Kripke-Platek set theory with infinity, whose proof-theoretic ordinal is much larger and is obtained by `collapsing' an uncountable ordinal.

\subsection{Goals of the article}

The goal of this article is to give a step-by-step and mostly self-contained account of the ordinal notation systems that arise from reflection calculi. Sections \ref{SecPrelim}-\ref{SecTransW} are devoted to giving an overview of known, `predicative' notation systems, first for $\varepsilon_0$ and then for $\Gamma_0$. However, our presentation is quite a bit different from those available in the current literature. In particular, it is meant to be `minimalist', in the sense that we only prove results that are central to our goal of comparing the reflection-based ordinal notations to standard proof-theoretic ordinals. Among other things, we sometimes do not show that the notation systems considered are computable.

The second half presents new material, providing impredicative notation systems based on provability logics. We first introduce {\em impredicative worms,} which give a representation system for $\uppsi(e^{\Omega+1}1)$, an ordinal a bit larger than the Bachmann-Howard ordinal. Then we introduce {\em spiders,} which are used to represent ordinals up to $\uppsi_0\Upomega^\omega 1$ in Buchholz-style notation \cite{Buchholz}. Here, $\Upomega^\omega 1$ is the first fixed point of the aleph function; unlike the predicative systems discussed above, these notation systems also include notations for several uncountable ordinals. The latter are then `collapsed' in order to represent countable ordinals much larger than $\Gamma_0$.

Although our focus is on notations arising from the reflection calculi and not on proof-theoretic interpretations of the provability operators, we precede each notation system with an informal discussion on such interpretations. These discussions are only given as motivation; further details may be found in the references provided. We also go into detail discussing the `traditional' notation systems for each of the proof-theoretical ordinals involved before discussing the reflection-based version, and thus this text may also serve as an introduction of sorts to ordinal notation systems.

\subsection{Layout of the article}

\begin{itemize}

\item[\S\ref{SecPrelim}:] Review of the basic definitions and properties of the reflection calculus $\sf RC$ and the transfinite provability logic $\sf GLP$.

\item[\S\ref{SecWorms}:] Introduction to {\em worms} and their order-theoretic properties.

\item[\S\ref{SecFiniteW}:] Computation of the order-type of worms with finite entries, and a brief over\-view of their interpretation in the language of Peano arithmetic.

\item[\S\ref{SecTransW}:] Computation of the order-type of worms with ordinal entries, and an overview of their interpretation in the language of second-order arithmetic.

\item[\S\ref{SecImpWrm}:] Introduction and analysis of impredicative worms, obtained by introducing an uncountable modality and its collapsing function.

\item[\S\ref{SecSpiders}:] Introduction to {\em spiders,} variants of worms interpreted using the aleph function and its collapses. 

\item[\S\ref{SecConc}:] Concluding remarks.

\end{itemize}

\section{The reflection calculus}\label{SecPrelim}

%%%%%%%%%%%%%%%%%%%%%%%%%%%%%%%%%

Provability logics are modal logics for reasoning about G\"odel's provability operator and its variants \cite{Boolos:1993:LogicOfProvability}. One uses $\nc\varphi$ to express {\em `$\varphi$ is provable in $T$';} here, $T$ may be Peano arithmetic, or more generally, any sound extension of elementary arithmetic (see Section \ref{SubsecFOA} below). The dual of $\nc$ is $\ps =\neg\nc\neg$, and we may read $\ps\varphi$ as {\em `$\varphi$ is consistent with $T$'.} This unimodal logic is called {\em G\"odel-L\"ob logic,} which Japaridze extended to a polymodal variant with one modality $[n]$ for each natural number in \cite{Japaridze:1986:PhdThesis}, further extended by Beklemishev to allow one modality for each ordinal in \cite{Beklemishev:2005:VeblenInGLP}.

The resulting polymodal logics have some nice properties; for exmample, they are decidable, provided the modalities range over some computable linear order. However, there are also some technical difficulties when working with these logics; most notoriously, they are incomplete for their relational semantics, and their topological semantics are quite complex \cite{BeklemishevGabelaia:2011:TopologicalCompletenessGLP,FernandezJoosten:2012:ModelsOfGLP,Fernandez:2012:TopologicalCompleteness,Icard:2009:TopologyGLP}.

Fortunately, Dashkov \cite{Dashkov:2012:PositiveFragment} and Beklemishev \cite{ Beklemishev:2012:CalibratingProvabilityLogic,Beklemishev:2013:PositiveProvabilityLogic} have shown that for proof-theoretic applications, it is sufficient to restrict to a more manageable fragment of Japaridze's logic called the {\em Reflection Calculus} (\RC). Due to its simplicity relative to Japaridze's logic, we will perform all of our modal reasoning directly within $\RC$.

\subsection{Ordinal numbers and well-orders}\label{SubsecOrdNum}

(Ordinal) reflection calculi are polymodal systems whose modalities range over a set or class of ordinal numbers, which are canonical representatives of well-orders. Recall that if $A$ is a set (or class), a {\em preorder} on $A$ is a trasitive, reflexive relation ${\preccurlyeq}\subseteq A\times A$. The preorder $\preccurlyeq$ is {\em total} if, given $a,b\in A$, we always have that $a\preccurlyeq b$ or $b\preccurlyeq a$, and {\em antisymmetric} if whenever $a\preccurlyeq b$ and $b\preccurlyeq a$, it follows that $a=b$. A total, antisymmetric preorder is a {\em linear order.} We say that $\langle A,\preccurlyeq\rangle$ is a {\em pre-well-order} if $\preccurlyeq$ is a total preorder and every non-empty $B\subseteq A$ has a minimal element (i.e., there is $m\in B$ such that $m\preccurlyeq b$ for all $b\in B$). A {\em well-order} is a pre-well-order that is also linear. Note that pre-well-orders are not the same as well-quasiorders (the latter need not be total).
Pre-well-orders will be convenient to us because, as we will see, worms are pre-well-ordered but not linearly ordered. 

Define $a\prec b$ by $a\preccurlyeq b$ but $b\not\preccurlyeq a$, and $a\approx b$ by $a\preccurlyeq b$ and $b\preccurlyeq a$. The next proposition may readily be checked by the reader:

\begin{proposition}\label{PropWO}
Let $\langle A,\preccurlyeq\rangle$ be a total preorder. Then, the following are equivalent:

\begin{enumerate}

\item $\preccurlyeq$ is a pre-well-order;

\item if $a_0,a_1,\hdots\subseteq A$ is any infinite sequence, then there are $i<j$ such that $a_i\preccurlyeq a_j$;

\item there is no infinite descending sequence
\[a_0 \succ a_1\succ a_2 \succ \hdots\subseteq A;\]

\item\label{PropWOItTI} if $B\subseteq A$ is such that for every $a\in A$,
\[\big (\forall b\prec a \, (b\in B) \big ) \rightarrow a\in B,\]
then $B=A$.

\end{enumerate}
\end{proposition}

We use the standard interval notation for preorders: $(a,b)=\{x:a\prec x\prec b\}$, $(a,\infty)=\{x:a\prec\ x\}$, etc. With this, we are ready to introduce ordinal numbers as a special case of a well-ordered set. Their formal definition is as follows:

\begin{definition}\label{DefOrd}
Say that a set $A$ is {\em transitive} if whenever $B\in A$, it follows that $B\subseteq A$. Then, a set $\xi$ is an {\em ordinal} if $\xi$ is transitive and $\langle \xi,\in\rangle$ is a strict well-order.
\end{definition}

When $\xi,\zeta$ are ordinals, we write $\xi<\zeta$ instead of $\xi\in \zeta$ and $\xi\leq\zeta$ if $\xi<\zeta$ or $\xi=\zeta$. The class of ordinal numbers will be denoted $\ord$. We will rarely appeal to Definition \ref{DefOrd} directly; instead, we will use some basic structural properties of the class of ordinal numbers as a whole. First, observe that $\ord$ is itself a (class-sized) well-order:

\begin{lemma}\label{LemmOrdBasic}
The class $\ord$ is well-ordered by $\leq$, and if $\Theta\subseteq\ord$ is a set, then $\Theta$ is an ordinal if and only if $\Theta$ is transitive.
\end{lemma}

Thus if $\xi$ is any ordinal, then $\xi=\{\zeta \in \ord : \zeta <\xi\}$, and $0 = \varnothing$ is the least ordinal. For $\xi\in \ord$, define $\xi+1=\xi\cup \{\xi\}$; this is the least ordinal greater than $\xi$. It follows from these observations that any natural number is an ordinal, but there are infinite ordinals as well; the set of natural numbers is itself an ordinal and denoted $\omega$. More generally, new ordinals can be formed by taking successors and unions:

\begin{lemma}\label{LemmOrdSucc}\
\begin{enumerate}

\item If $\xi$ is any ordinal, then $\xi+1$ is also an ordinal. Moreover, if $\zeta<\xi+1$, it follows that $\zeta\leq\xi$.

\item If $\Theta$ is a set of ordinals, then $\lambda=\bigcup \Theta$ is an ordinal. Moreover, if $\xi<\lambda$, it follows that $\xi<\theta$ for some $\theta\in \Theta$.

\end{enumerate}

\end{lemma}

These basic properties will suffice to introduce the reflection calculus, but later in the text we will study ordinals in greater depth. A more detailed introduction to the ordinal numbers may be found in a text such as \cite{Jech:2002:SetTheory}.

\subsection{The reflection calculus}

The modalities of reflection calculi are indexed by elements of some set of ordinals $\Lambda$. Alternately, one can take $\Lambda$ to be the class of all ordinals, obtaining a class-sized logic. Formulas of $\RC_\Lambda$ are built from the grammar
\[
\top \ | \ \phi\wedge\psi \ | \ \langle \lambda\rangle \phi,
\]
where $\lambda<\Lambda$ and $\phi,\psi$ are formulas of $\RC_\Lambda$; we may write $\lambda\phi$ instead of $\langle\lambda\rangle\phi$, particularly since $\RC_\Lambda$ does not contain expressions of the form $[\lambda]\phi$. The set of formulas of $\RC_\Lambda$ will be denoted $\lan\Lambda$, and we will simply write $\lan {\RC}$ and $\RC$ instead of $\lan \ord$, $\RC_\ord$. Propositional variables may also be included, but we will omit them since they are not needed for our purposes. Note that this strays from convention, since the variable-free fragment is typically denoted $\RC^0$. Reflection calculi derive {\em sequents} of the form $\phi\rcto\psi$, using the following rules and axioms:
\[\begin{array}{lcr}
\phi\rcto\phi&\phi\rcto\top& \ \ \ \ \  \dfrac{\phi\rcto\psi \ \ \ \ \psi\rcto \theta}{\phi\rcto\theta}\\\\
\phi\wedge\psi\rcto\phi&\phi\wedge\psi\rcto\psi
&\dfrac{\phi\rcto\psi \ \ \  \ \phi\rcto \theta}{\phi\rcto\psi\wedge\theta}\\\\
\lambda\lambda\phi\rcto\lambda\phi&\dfrac{\phi\rcto \psi}{\lambda\phi\rcto\lambda\psi}&\\\\
\lambda\phi\rcto\mu\phi \ \ & \text{for $\mu\leq \lambda$;}\\\\
 \lambda\phi\wedge\mu\psi\rcto \lambda(\phi\wedge\mu\psi)&\text{for $\mu<\lambda$.}
\end{array}
\]

Let us write $\phi\equiv\psi$ if $\RC_\Lambda\vdash \phi \rcto \psi$ and $\RC_\Lambda\vdash \psi \rcto \phi$. Then, the following equivalence will be useful to us:

\begin{lemma}\label{LemmRC}
Given formulas $\phi$ and $\psi$ and ordinals $\mu < \lambda$,
\[(\lambda  \phi \wedge  \mu  \psi) \equiv  \lambda (\phi \wedge  \mu  \psi).\]
\end{lemma}

\proof
The left-to-right direction is an axiom of \RC. For the other direction we observe that $ \lambda\mu \psi \rcto  \mu \psi$ is derivable using the axioms $\lambda\mu  \psi \rcto  \mu\mu \psi$ and $\mu\mu\psi\rcto\mu\psi$, from which the desired derivation can easily be obtained.
\endproof

Reflection calculi enjoy relatively simple relational semantics, where formulas have truth values on some set of points $X$, and each expression $\lambda\varphi$ is evaluated using an accessibility relation $\succ_\lambda$ on $X$.

\begin{definition}\label{DefRCF}
An {\em ${\sf RC}_\Lambda$-frame} is a structure $\mathfrak F=\<X,\<\succ_\lambda\>_{\lambda<\Lambda}\>$ such that for all $x,y,z\in X$ and all $\mu <\lambda<\Lambda$,
\begin{enumerate}[label=(\roman*)]

\item\label{RCFone} if $x\succ_\mu y\succ_\mu z$ then $x\succ_\mu z$,

\item\label{RCFtwo} if $z \succ_\mu x$ and $z \succ_\lambda y$ then $y\succ_\mu x$, and

\item\label{RCFthree} if $x\succ_\lambda y$ then $x\succ_\mu y$.

\end{enumerate}
The {\em valuation} on $\mathfrak F$ is the unique function $\lb\cdot\rb_\mathfrak F:\lan\Lambda\to 2^X$ such that
\[
\begin{array}{lcl}
\lb\bot\rb_\mathfrak F&=&\varnothing\\\\
\lb\neg\phi\rb_\mathfrak F&=&X\setminus\lb\phi\rb_\mathfrak F\\\\
\lb\phi\wedge\psi\rb_\mathfrak F&=&\lb\phi\rb_\mathfrak F\cap\lb\psi\rb_\mathfrak F\\\\
\lb \lambda \phi\rb_\mathfrak F&=&\big \{x\in X : \exists y {\prec_\lambda} x \, (y\in \lb\phi\rb_\mathfrak F)\big \}.
\end{array}
\]

We may write $ ( \mathfrak F,x ) \models\psi$ instead of $x\in\lb \psi\rb_\mathfrak F$. As usual, $\phi$ is {\em satisfied} on $\mathfrak F$ if $\lb\phi\rb_\mathfrak F \not=\varnothing$, and {\em true} on $\mathfrak F$ if $\lb\phi\rb_\mathfrak F=X$.
\end{definition}

\begin{theorem}\label{theoSound}
For any class or set of ordinals $\Lambda$, $\RC_\Lambda$ is sound for the class of $\RC_\Lambda$-frames.
\end{theorem}

\proof
The proof proceeds by a standard induction on the length of a derivation and we omit it.
\endproof

In fact, Dashkov proved that $\RC_\omega$ is also complete for the class of $\RC_\omega$-frames \cite{Dashkov:2012:PositiveFragment};\footnote{Beware that $\RC_\omega$ in our notation is not the same as $\RC\omega$ in \cite{Beklemishev:2013:PositiveProvabilityLogic}.} it is very likely that his result can be generalized to full $\RC$ over the ordinals, either by adapting his proof or by applying reduction techniques as in \cite{BeklemishevFernandezJoosten:2012:LinearlyOrderedGLP}. However, we remark that only soundness will be needed for our purposes.

\subsection{Transfinite provability logic}

The reflection calculus was introduced as a restriction of Japaridze's logic $\glp_\omega$ \cite{Japaridze:1988}, which itself was extended by Beklemishev to full $\glp$ \cite{Beklemishev:2005:VeblenInGLP}, containing one modality for each ordinal number. Although we will work mostly within the reflection calculus, for historical reasons it is convenient to review the logic $\glp$.

The (variable-free) language of $\glp$ is defined by the following grammar:
\[
\top \ | \ \bot \ | \ \phi\wedge\psi \ | \ \phi\to\psi \ | \ \langle \lambda\rangle \phi.
\]
Note that in this language we can define negation (as well as other Boolean connectives), along with $[\lambda]\phi=\neg\langle\lambda\rangle \neg \phi$.

The logic $\mathsf{GLP}_\Lambda$ is then given by the following rules and axioms:
\begin{enumerate}[label=(\roman*)]
\item all propositional tautologies{,}
\item $[\lambda](\phi\to\psi)\to([\lambda]\phi\to[\lambda]\psi)$ for all $\lambda <\Lambda${,}
\item {$[\lambda]([\lambda]\phi \to \phi)\to[\lambda]\phi$ for all $\lambda <\Lambda$}{,}\label{AxLob}
\item $[\mu]\phi\to[\lambda]\phi$ for $\mu<\lambda<\Lambda${,}\label{glpfour}
\item $\<\mu\>\phi\to [\lambda]\<\mu\>\phi$ for $\mu<\lambda<\Lambda$,\label{AxGLPNI}
\item modus ponens and
\item necessitation for each $[\xi]$.
\end{enumerate}

The reader may recognize axiom \ref{AxLob} as L\"ob's axiom \cite{Lob:1955:SolutionProblemHenkin}, ostensibly absent from $\RC$; it is simply not expressible there. However, it was proven by Dashkov that $\glp$ is conservative over $\RC$, in the following sense:

\begin{theorem}\label{TheoGLPCons}
If $\phi,\psi\in\lan\RC$, then $\RC\vdash\phi\rcto\psi$, if and only if $\glp\vdash\phi\to\psi$.
\end{theorem}

\proof
That $\RC\vdash\phi\rcto\psi$ implies $\glp\vdash\phi\to\psi$ is readily proven by induction on the length of a derivation; one need only verify that, for $\mu<\lambda$,
\[\glp\vdash \langle \lambda\rangle \phi\wedge\langle\mu\rangle\psi\to \langle\lambda\rangle(\phi\wedge \langle\mu\rangle \psi),\]
using the $\glp$ axiom \ref{AxGLPNI}.

The other direction was proven for $\RC_\omega$ by Dashkov in \cite{Dashkov:2012:PositiveFragment}. To extend to modalities over the ordinals, assume that $\glp\vdash\phi\to\psi$. Then, there are finitely many modalities appearing in the derivation of $\vdash\phi\to\psi$, hence $\glp_\Theta\vdash\phi\to\psi$ for some finite set $\Theta$. But $\glp_\Theta$ readily embeds into $\glp_\omega$ (see \cite{BeklemishevFernandezJoosten:2012:LinearlyOrderedGLP}), and thus we can use the conservativity of $\glp_\omega$ over $\RC_\omega$ to conclude that $\RC\vdash \phi\rcto\psi$.
\endproof

As we have mentioned, full $\glp$ (with propositional variables), or even $\glp_2$, is incomplete for its relational semantics. Without propositional variables, Ignatiev has built a relational model in which every consistent formula of variable-free $\glp_\omega$ is satisfied \cite{Ignatiev:1993:StrongProvabilityPredicates}, and Joosten and I extended this to variable-free $\glp$ over the ordinals. However, these models are infinite, and even $1\top$ cannot be satisfied on any finite relational model validating variable-free $\glp$. On the other hand, every worm has a relatively small $\RC$-model, as we will see below.

\section{Worms and consistency orderings}\label{SecWorms}

Worms are expressions of $\RC$ (or $\glp$) representing iterated consistency assertions. Ignatiev first observed that the worms in $\glp_\omega$ are well-founded \cite{Ignatiev:1993:StrongProvabilityPredicates}. The order-types of worms in $\glp_2$ were then studied by Boolos \cite{Boolos:1993:LogicOfProvability}, and in full $\glp$ by Beklemishev \cite{Beklemishev:2005:VeblenInGLP} and further by Joosten and I in \cite{FernandezJoosten:2012:WellOrders}, this time working in $\RC$. Moreover, this particular well-order has surprising proof-theoretical applications: Beklemishev has used transfinite induction along the $\RC_\omega$ worms to prove the consistency of Peano arithmetic and compute its $\Pi^0_1$ ordinal \cite{Beklemishev:2004:ProvabilityAlgebrasAndOrdinals}.

In this section we will review the ordering between worms and show that it is well-founded. Let us begin with some preliminaries.

\subsection{Basic definitions}

\begin{definition}\label{defWorm}
A {\em worm} is any \RC formula of the form
\[\fw=\lambda_1\hdots\lambda_n\top,\]
with each $\lambda_i$ an ordinal and $n<\omega$ (including the `empty worm', $\top$). The class of worms is denoted $\Worms$.

If $\Lambda$ is a set or class of ordinals and each $\lambda_i\in \Lambda$, we write $\fw\sqsubset\Lambda$. The set of worms $\fv$ such that $\fv\sqsubset\Lambda$ is denoted $\Worms_\Lambda$.
\end{definition}

`Measuring' worms is the central theme of this work. Let us begin by giving notation for some simple measurements, such as the length and the maximum element of a worm.

\begin{definition}
If $\fw=\lambda_1\hdots\lambda_n\top$, then we set $\lgt{\mathfrak w}=n$ (i.e., $\lgt{\mathfrak w}$ is the {\em length} of $\fw$). Define $\min \fw=\min_{i\in[1,n]}\lambda_i$, and similarly $\max \fw=\max_{i\in[1,n]}\lambda_i$. The class of worms $\fw$ such that $\fw=\top$ or $\mu \leq \min\fw$ will be denoted $\Worms_{\geq \mu}$. We define $\Worms_{> \mu}$ analogously.
\end{definition}

These give us some idea of `how big' a worm is, but what we are truly interested in is in ordering worms by their {\em consistency strength:}

\begin{definition}
Given an ordinal $\lambda$, we define a relation $\wle{\lambda}$ on $\Worms$ by  $\mathfrak v \wle {\lambda} \mathfrak w$ if and only if $\RC\vdash \mathfrak w \rcto \lambda \mathfrak v.$ We also define $\fv\wleq\mu\fw$ if $\fv\wle\mu\fw$ or $\fv\equiv\fw$.
\end{definition}

Instead of $\wle 0,\wleq 0$ we may simply write $\wle {}, \wleq{}$. As we will see, these orderings have some rather interesting properties. Let us begin by proving some basic facts about them:

\begin{lemma}\label{LemmTopMin}
Let $\mu\leq\lambda$ be ordinals and $\fu,\fv,\fw$ be worms. Then:
\begin{enumerate}

\item if $\mathfrak w\not=\top$ and $\mu< \min \mathfrak w$, then $\top \wle \mu \mathfrak w$,

\item if $\fv\wle\lambda\fw$, then $\fv \wle\mu \fw$, and

\item if $\fu\wle\mu\fv$ and $\fv\wle\mu\fw$, then $\fu\wle\mu \fw$.

\end{enumerate}

\end{lemma}

\proof
For the first item, write $\fw=\lambda\fv$, so that $\lambda\geq\mu$. Then, $\fv\rcto\top$ is an axiom of $\RC$, from which we can derive $\lambda\fv\rcto\lambda\top$ and from there use the axiom $\lambda\top\rcto\mu\top$.

For the second item, if $\fv\wle\lambda\fw$, then by definition, $\fw\rcto \lambda\fv$ is derivable. Using the axiom $\lambda\fv\rcto\mu\fv$, we see that $\fw\rcto\mu\fv$ is derivable as well, that is, $\fv\wle\mu\fw$.

Transitivity simply follows from the fact that $\RC \proves \mu\mu\fu\rcto\mu\fu$, so that if $\fu\wle\mu\fv$ and $\fv\wle\mu\fw$, we have that $\RC \proves \fw\rcto\mu\fv\rcto \mu\mu\fu\rcto \mu\fu$, so $\fu\wle\mu\fw$.
\endproof

\subsection{Computing the consistency orders}

The definition of $\fv\wle\lambda\fw$ does not suggest an obvious algorithm for deciding whether it holds or not. Fortunately, it can be reduced to computing the ordering between smaller worms; in this section, we will show how this is done. Let us begin by proving that $\wle\mu$ is always irreflexive. To do this, we will use the following frames.

\begin{definition}
Let $\fw = \lambda_n \hdots \lambda_ 0 \top$ be any worm (note that we are using a different enumeration from that in Definition \ref{defWorm}). Define a frame $\mathfrak F(\fw)= \big \langle X,\langle \succ_\lambda\rangle_{\lambda<\Lambda}\big\rangle$ as follows.

First, set $X = [0,n+1] \subseteq \mathbb N$. To simplify notation below, let $\lambda_{n+1} = 0$. Then, define $x \succ_\eta y$ if and only if:
\begin{enumerate}

\item $x > y$ and for all $i \in [y,x) $, $\lambda_i \geq \eta$, or

\item $x \leq y$ and for all $i\in [x,y]$, $\lambda _{i} > \eta$.

\end{enumerate}
\end{definition}

Although this might not be obvious from the definition, these frames are indeed $\RC$-frames.

\begin{lemma}\label{lemFisFrame}
Given any worm $\fw$, $\mathfrak F(\fw)$ is an $\RC$-frame.
\end{lemma}

\proof
We must check that $\mathfrak F (\fw)$ satisfies each item of Definition \ref{DefRCF}.\\

\noindent \ref{RCFone} Suppose that $x \succ_\eta y \succ_\eta z $. If $x>y$, consider three sub-cases.

\begin{enumerate}[label=\alph*.]

\item

If $y>z$, from $x \succ_\eta y \succ_\eta z $ we see that for all $i \in [z,y) \cup [y,x) = [z,x) $, $\lambda_i \geq \eta$, so that $x \succ_\eta z $.

\item

If $z\in [y,x)$, from $[z,x) \subseteq [y,x)$ and $x \succ_\eta y$ we obtain $\lambda _{i} \geq \eta$ for all $i \in [z,x)$, so $x \succ_\eta z$.

\item

If $z \geq x $, from $[x,z] \subseteq [y,z]$ and $y \succ_\eta z$ we obtain $\lambda _{i } > \eta$ for all $i \in [x,z]$, hence $x \succ_\eta z$.

\end{enumerate}
The cases where $x\leq y$ are analogous.
\\

\noindent \ref{RCFtwo}. As in the previous item, we must consider several cases. Suppose that $\mu<\eta$, $z \succ_\mu x$ and $z \succ_\eta y$.
If $z > x$, we consider three subcases.

\begin{enumerate}[label=\alph*.]

\item

If $y \leq x$, then from $z \succ_ \eta y$ and $[y,x] \subseteq [y,z)$ we obtain $\lambda_{i} \geq \eta > \mu$ for all $i \in [y,x]$, hence $y \succ_\mu x$.

\item

If $y\in (x,z]$, then from $[x,y) \subseteq [x,z)$ and $z\succ_\mu x$ we obtain $\lambda_{i} \geq  \mu$ for all $i\in [x,y)$, hence $y \succ_\mu x$.

\item

If $y>z$, then from $z\succ_\mu x$ we we have that $\lambda_i \geq \mu$ for all $i\in [x,z)$, while from $z\succ_\eta y$ it follows that for all $i\in [z,y)$, $\lambda_i > \eta >\mu$, giving us $y \succ_\mu x$. 

\end{enumerate}
Cases where $z \leq x$ are similar.\\

\noindent \ref{RCFthree}. That $\succ_\mu$ is monotone on $\mu$ is obvious from its definition.
\endproof

Thus to prove that $ \wle \mu$ is irreflexive, it suffices to show that there is $x\in [0,n+1]$ such that $ \big ( \mathfrak F(\fw) ,x \big )  \models \lambda_{n-1} \hdots \lambda_ 0 \top$ but $ \big ( \mathfrak F(\fw) ,x \big ) \not \models \lambda_n \hdots \lambda_ 0 \top$, as then by setting $\mu = \lambda_n$ and $\fv = \lambda_{n-1} \hdots \lambda_0 \top$ we see that $\fv\not\wle \mu \fv$. The following lemma will help us find such an $x$.

\begin{lemma}\label{lemFrameSat}
Let $\fw = \lambda_{n} \hdots \lambda_0 \top$ be a worm, and for any $i \in [0,n+1]$, define $\fw[i]$ recursively by $\fw[0]=\top$ and $\fw[i+1] = \lambda_i \fw [i]$. Then:
\begin{enumerate}

\item $ \big ( \mathfrak F(\fw) , i \big ) \models \fw[i]$, and

\item if $x \in [0,i)$, then $ \big ( \mathfrak F(\fw),x \big ) \not\models \fw[i]$.

\end{enumerate}

\end{lemma}

\proof
The first claim is easy to check from the definition of $\mathfrak F(\fw)$, so we focus on proving the second by induction on $i$.
The base case is vacuously true as $[0,0) = \varnothing$. Otherwise, assume the claim for $i$, and consider $x \in [0,i+1)$; we must show that $ \big ( \mathfrak F(\fw),x \big ) \not \models \fw[i+1] = \lambda_{i}\fw[i]$, which means that for all $y \prec_{\lambda_{i} } x$, $ \big ( \mathfrak F(\fw),y \big ) \not \models \fw[i]$. Note that we cannot have that $y \in [i,n+1]$, as in this case $y \geq i \geq x$; but obviously $\lambda_i \not>\lambda_i$, so that $y \not \prec_{\lambda_i} x$. It follows that $ y \in [0,i)$, and we can apply the induction hypothesis to $\fw[i]$. 
\endproof

\begin{lemma}\label{LemmIrr}
Given any ordinal $\mu$ and any worm $\fv$, we have that $\fv\not\wle{\mu}\fv$.
\end{lemma}

\proof
Let $\mu$ be any ordinal, $\fv$ be any worm, and consider the $\sf RC$-frame $\mathfrak F(\mu\fv)$. If $n=\# \fv$, observe that $(\mu\fv)[n] = \fv$, hence by Lemma \ref{lemFrameSat}, $ \big ( \mathfrak F(\mu\fv),n \big ) \models \fv$ but $ \big ( \mathfrak F(\mu\fv),n \big ) \not \models \mu \fv$; it follows from Theorem \ref{theoSound} that $\fv \not \wle \mu \fv$.
\endproof

Thus the worm orderings are irreflexive.
Next we turn our attention to a useful operation between worms.
Specifically, worms can be regarded as strings of symbols, and as such we can think of concatenating them.

\begin{definition}
Let $\mathfrak v= \xi_1 \hdots \xi_n \top$ and $\mathfrak w= \zeta_1 \hdots \zeta_m \top$ be worms. Then, define
\[\mathfrak v \mathfrak w= \xi_1 \hdots \xi_n  \zeta_1 \hdots \zeta_m \top\]
\end{definition}

Often we will want to put an extra ordinal between the worms, and we write $\mathfrak v \mathrel \lambda \mathfrak w$ for $\fv(\lambda\fw)$.

\begin{lemma}\label{LemmWormConj}
If $\mathfrak w,\mathfrak v$ are worms and $\mu<\min\mathfrak w$, then $\mathfrak w\mathrel \mu\mathfrak v\equiv \mathfrak w\wedge\mu\mathfrak v$.
\end{lemma}

\begin{proof}
By induction on $\lgt\fw$. If $\fw=\top$, the claim becomes $\mu\fv \equiv\top\wedge\mu\fv$, which is obviously true. Otherwise, we write $\fw = \lambda\fu$ with $\lambda>\mu$, and observe that by Lemma \ref{LemmRC},
\[\lambda\fu\wedge \mu\fv\equiv\lambda(\fu\wedge \mu \fv)\stackrel{\text{\sc IH}}\equiv \lambda(\fu\mathrel \mu\fv)=\fw \mathrel \mu\fv.\qedhere\]
\end{proof}

Thus we may ``pull out'' the initial segment of a worm, provided the following element is a lower bound for this initial segment. In general, for any ordinal $\lambda$, we can pull out the maximal initial segment of $\fw$ which is bounded below by $\lambda$; this segment is the {\em $\lambda$-head} of $\fw$, and what is left over (if anything) is its {\em $\lambda$-body.}

\begin{definition}
Let $\lambda$ be an ordinal and $\mathfrak w \in \Worms_{\geq \lambda}$. We define $h_\lambda (\mathfrak w)$ to be the maximal initial segment of $\mathfrak w$ such that $\lambda < \min h_\lambda (\mathfrak w)$, and define $b_\lambda (\mathfrak w)$ as follows: if $\lambda$ appears in $\mathfrak w$, then we set $b_\lambda(\mathfrak w)$ to be the unique worm such that $\mathfrak w=h_\lambda(\mathfrak w)\mathrel \lambda b_\lambda(\mathfrak v)$. Otherwise, set $b_\lambda(\mathfrak w)=\top$.
\end{definition}

We may write $h,b$ instead of $h_0,b_0$.
We remark that our notation is a variant from that used in \cite{FernandezJoosten:2012:WellOrders}, where our $h_\lambda$ would be denoted $h_{\lambda+1}$.

\begin{lemma}\label{LemmDecomp}
Given a worm $\mathfrak w\not=\top$ and an ordinal $\mu\scleq \min\mathfrak w$,
\begin{enumerate}

\item $ h_\mu(\mathfrak w) \in \Worms_{>\mu}$,

\item\label{LemmDecompItTwo} $\lgt{h_\mu(\mathfrak w)}\leq \lgt{\mathfrak w }$, with equality holding only if $\mu\scle  \min \mathfrak w$, in which case $h_\mu(\mathfrak w) = \mathfrak w $;

\item $\lgt{b_\mu(\mathfrak w)}< \lgt{\mathfrak w}$, and

\item\label{LemmDecompItFour} $\mathfrak w\equiv  h_\mu (\mathfrak w)  \wedge\mu  b_\mu(\mathfrak w).$

\end{enumerate}

\end{lemma}

\proof
The first two claims are immediate from the definition of $h_\mu$. For the third, this is again obvious in the case that $\mu$ occurs in $\fw$, otherwise we have that $b_\mu(\fw)=\top$ and by the assumption that $\fw\not=\top$ we obtain $\lgt{b_\mu(\fw)}<\lgt\fw$.

The fourth claim is an instance of Lemma \ref{LemmWormConj} if $\mu$ appears in $\fw$, otherwise $\fw=h_\mu(\fw)$ and we use Lemma \ref{LemmTopMin} to see that $\fw\rcto\mu\top=\mu b_\mu(\top)$ is derivable.
\endproof

With this we can reduce relations between worms to those between their heads and bodies.

\begin{lemma}\label{LemmWOrdRecurs}
If $\mathfrak w,\mathfrak v \not = \top$ are worms and $\mu\leq \min \mathfrak w \mathfrak v$, then
\begin{enumerate}

\item \label{ItWOROne}
$\mathfrak w\wle  \mu \mathfrak v$ whenever
\begin{enumerate}

\item $\mathfrak w\wleq {\mu} b_\mu (\mathfrak v)$, or

\item $b_\mu (\mathfrak w)\wle {\mu} \mathfrak v$ and $h_\mu (\mathfrak w)\wle  {\mu+1} h_\mu (\mathfrak v)$, and

\end{enumerate}

\item \label{ItWORTwo}

$\RC \vdash \mathfrak v \rcto  \mathfrak w$ whenever
$b_\mu (\mathfrak w)\wle  {\mu} \mathfrak v$ and $\RC \vdash h_\mu (\fv) \rcto  h_\mu (\fw)$.

\end{enumerate}

\end{lemma}

\proof
For the first claim, if $\mathfrak w\wleq {\mu} b_\mu (\mathfrak v)$, then by Lemma \ref{LemmDecomp}.\ref{LemmDecompItFour} we have that $\fv\rcto \mu b_\mu (\mathfrak v)$, that is, $b_\mu (\mathfrak v)\wle\mu\fv$. By transitivity we obtain $\fw\wle\mu\fv$. If $b_\mu (\mathfrak w)\wle {\mu} \mathfrak v$ and $h_\mu (\mathfrak w)\wle {\mu+1} h_\mu (\mathfrak v)$, reasoning in $\RC $ we have that
\[\fv\rcto h_\mu(\fv)\wedge \fv \rcto \langle\mu+1\rangle h_\mu(\fw)\wedge \mu b_\mu (\mathfrak w) \equiv \langle\mu+1\rangle  h_\mu(\fw)  \mu b_\mu (\mathfrak w) \rcto \mu \fw,\]
and $\fw\wle\mu\fv$, as needed.

For the second, if $b_\mu (\mathfrak w)\wle  {\mu} \mathfrak v$ and $\RC \vdash h_\mu (\fv) \rcto  h_\mu (\fw)$, we have that
\[\fv\rcto h_\mu(\fv)\wedge \fv \rcto   h_\mu(\fw)\wedge \mu b_\mu (\mathfrak w) \equiv  \fw. \qedhere \]
\endproof

As we will see, Lemma \ref{LemmWOrdRecurs} gives us a recursive way to compute $\wleq{\mu}$. This recursion will allow us to establish many of the fundamental properties of $\wleq\mu$, beginning with the fact that it defines a total preorder.

\begin{lemma}\label{LemmWormLinear}
Given worms $\fv,\fw$ and $\mu\leq\min (\fw\fv)$, exactly one of $\mathfrak w\wleq{\mu}\mathfrak v$ or $\mathfrak v\wle{\mu}\mathfrak w$ occurs.
\end{lemma}

\proof
That they cannot simultaneously occur follows immediately from Lemma \ref{LemmIrr}, since $\wle\mu$ is irreflexive.

To show that at least one occurs, proceed by induction on $\lgt\fw+\lgt\fv$. To be precise, assume inductively that whenever $\lgt{\fw'}+\lgt{\fv'}<\lgt\fw+\lgt\fv$ and $\mu\leq \min (\fw'\fv') $ is arbitrary, then either $\fw'\wleq\mu \fv'$ or $\fv'\wleq\mu \fw'$. If either $\fv=\top$ or $\fw=\top$, then the claim is immediate from Lemma \ref{LemmTopMin}.

Otherwise, let $\lambda=\min (\fw\fv) $, so that $\lambda\geq \mu$. If $\fw\wleq\lambda b_\lambda(\fv)$, then by Lemma \ref{LemmWOrdRecurs}, $\fw\wle\lambda\fv$, and similarly if $\fv\wleq\lambda b_\lambda(\fw)$, then $\fv\wle\lambda\fw$. On the other hand, if neither occurs then by the induction hypothesis we have that $b_\lambda(\fv)\wle \lambda \fw$ and $b_\lambda(\fw)\wle \lambda \fv$.

Since $\lambda$ appears in either $\fw$ or $\fv$, by Lemma \ref{LemmDecomp}.\ref{LemmDecompItTwo} we have that
\[\lgt{h_\lambda(\fw)}+\lgt{h_\lambda(\fv)}<\lgt\fw+\lgt\fv,\]
so that by the induction hypothesis, either $h_\lambda(\fw)\wle \mu h_\lambda(\fv)$, $h_\lambda(\fw)\equiv h_\lambda(\fv)$, or $h_\lambda(\fv)\wle  \lambda h_\lambda(\fw)$. If $h_\lambda(\fw)\wle  \lambda h_\lambda(\fv)$, we may use Lemma \ref{LemmWOrdRecurs}.\ref{ItWOROne} to see that $\fw\wle \lambda\fv$, so that by Lemma \ref{LemmTopMin}, $\fw\wle \mu \fv$. Similarly, if $h_\lambda(\fv)\wle  \lambda h_\lambda(\fw)$, we obtain $\fv \wle \mu \fw$. If $h_\lambda(\fw)\equiv  h_\lambda(\fv)$, then Lemma \ref{LemmWOrdRecurs}.\ref{ItWORTwo} yields both $  \fw \rcto \fv$ and $  \fw \rcto \fv$, i.e., $\fw\equiv \fv$.
\endproof

\begin{corollary}\label{CorImpLeq}
If $\RC\vdash\fw\rcto\fv$, then $\fv\wleq{}\fw$.
\end{corollary}

\proof
Towards a contradiction, suppose that $\RC\vdash\fw\rcto\fv$ but $\fv\not\wleq{}\fw$. By Lemma \ref{LemmWormLinear}, $\fw\wle{}\fv$. Hence $\fv\rcto\fw\rcto 0\fv$, and $\fv\wle{}\fv$, contradicting the irreflexivity of $\wle{}$.
\endproof

Moreover, the orderings $\wle\lambda$, $\wle\mu$ coincide on $\Worms_{\geq\max\{\lambda,\mu\}}$:

\begin{lemma}\label{LemmSameOrders}
Let $\fw,\fv$ be worms and $\mu,\lambda\leq\min ( \mathfrak w\mathfrak v) $. Then, $\mathfrak w\wle {\mu}\mathfrak v$ if and only if $\mathfrak w\wle {\lambda}\mathfrak v$.
\end{lemma}

\proof
Assume without loss of generality that $\mu\leq\lambda$. One direction is already in Lemma \ref{LemmTopMin}. For the other, assume towards a contradiction that $\mathfrak w\wle {\mu}\mathfrak v$ but $\mathfrak w\not\wle {\lambda}\mathfrak v$. Then, by Lemma \ref{LemmWormLinear}, $\fv\wleq\lambda\fw$ and thus $\fv\wleq\mu\fw$, so that $\fv\wleq\mu \fw\wle {\mu}\fv,$ contradicting the irreflexivity of $\wle\mu$ (Lemma \ref{LemmIrr}).
\endproof

With this we can give an improved version of Lemma \ref{LemmWOrdRecurs}, that will be more useful to us later.

\begin{theorem}\label{TheoWormOrder}

The relation $\wleq {\lambda}$ is a total preorder on $\Worms_{\geq\lambda}$, and for all $\mu\leq\lambda$ and $\mathfrak w,\mathfrak v\in\Worms_{\geq\lambda}$ with $\fw,\fv \not = \top$,

\begin{enumerate}

\item\label{TheoWormOrderItA} $\mathfrak w\wle {\mu} \mathfrak v$ if and only if
\begin{enumerate}

\item\label{TheoWormOrderItOne} $\mathfrak w\wleq {\mu} b_\lambda (\mathfrak v)$, or

\item\label{TheoWormOrderItTwo} $b_\lambda (\mathfrak w)\wle {\mu} \mathfrak v$ and $h_\lambda (\mathfrak w)\wle {\mu} h_\lambda (\mathfrak v)$, and

\end{enumerate}

\item\label{TheoWormOrderItB} $\mathfrak w\wleq {\mu} \mathfrak v$ if and only if
\begin{enumerate}

\item\label{TheoWormOrderItBOne} $\mathfrak w\wleq {\mu} b_\lambda (\mathfrak v)$, or

\item\label{TheoWormOrderItBTwo} $b_\lambda (\mathfrak w)\wle {\mu} \mathfrak v$ and $h_\lambda (\mathfrak w)\wleq {\mu} h_\lambda (\mathfrak v)$.

\end{enumerate}

\end{enumerate}

\end{theorem}

\proof
Totality is Lemma \ref{LemmWormLinear}.
Let us prove item \ref{TheoWormOrderItB}; the proof of item \ref{TheoWormOrderItA} is similar. If \eqref{TheoWormOrderItBOne} holds, then by Lemma \ref{LemmSameOrders}, $\mathfrak w\wleq {\lambda} b_\lambda (\mathfrak v)$, so that by Lemma \ref{LemmWOrdRecurs}.\ref{ItWOROne}, $\fw\wle\lambda\fv$, and once again by Lemma \ref{LemmSameOrders}, $\fw\wle \mu\fv$. If \eqref{TheoWormOrderItBTwo} holds, then by Lemma \ref{LemmSameOrders} we obtain $b_\lambda (\mathfrak w)\wle {\lambda} \mathfrak v$ and $h_\lambda (\mathfrak w)\wleq {\lambda+1} h_\lambda (\mathfrak v)$. If $h_\lambda (\mathfrak w)\wle {\lambda+1} h_\lambda (\mathfrak v)$, we may use Lemma \ref{LemmWOrdRecurs}.\ref{ItWOROne} to obtain $\fw\wle\lambda\fv$. Otherwise, by Lemma \ref{LemmWOrdRecurs}.\ref{ItWORTwo}, we see that $\RC\vdash \fv\rcto\fw$, which by Corollary \ref{CorImpLeq} gives us $\fw\wle\lambda\fv$. In either case, $\fw\wle\mu\fv$.

For the other direction, assume that \eqref{TheoWormOrderItBOne} and \eqref{TheoWormOrderItBTwo} both fail. Then by Lemma \ref{LemmWormLinear} together with Lemma \ref{LemmSameOrders}, we have that $ b_\lambda (\mathfrak v)\wle\lambda\mathfrak w$ and either $\fv\wleq\lambda b_\lambda(\fw)$ or $h_\lambda (\fv)\wle  {\lambda+1} h_\lambda (\fw)$. In either case $\fv\wle\lambda\fw$, and thus $\fw\not\wleq\mu\fv$.
\endproof

Before continuing, it will be useful to derive a few straightforward consequences of Theorem \ref{TheoWormOrder}.

\begin{corollary}
Every $\phi\in\lan\RC$ is equivalent to some $\mathfrak w\in \Worms$. Moreover, we can take $\fw$ so that every ordinal appearing in $\fw$ already appears in $\phi$.
\end{corollary}

\proof
By induction on the complexity of $\phi$. We have that $\top$ is a worm and for $\phi=\lambda\psi$, by induction hypothesis we have that $\psi\equiv\fv$ for some worm $\fv$ with all modalities appearing in $\psi$ and hence $\phi\equiv\lambda\fv$.

It remains to consider an expression of the form $\psi\wedge\phi$. Using the induction hypothesis, there are worms $\fw,\fv$ equivalent to $\phi,\psi$, respectively, so that $\psi\wedge\phi\equiv\fw\wedge\fv$. We proceed by a secondary induction on $\lgt\fw+\lgt\fv$. Note that the claim is trivial if either $\fw = \top$ or $\fv = \top$, so we assume otherwise.

Let $\mu$ be the least ordinal appearing either in $\fw$ or in $\fv$, so that
\[\psi\wedge\phi\equiv (h_\mu(\fw)\wedge h_\mu(\fv)) \wedge (\mu b_\mu(\fw)\wedge \mu b_\mu (\fv)).\]
By induction hypothesis, $h_\mu (\fw)\wedge h_\mu (\fv)\equiv\fu_1$ for some $\fu _1 \in \Worms_{\mu +1}$ with all modalities occurring in $\phi\wedge\psi$. Meanwhile, either $b_\mu(\fw)\wle \mu b_\mu (\fv)$, $b_\mu(\fw)\equiv b_\mu (\fv)$ or $b_\mu(\fv)\wle \mu b_\mu (\fw)$. In the first case,
\[\mu b_\mu (\fv)\rcto \mu \mu b_\mu (\fw)\rcto \mu  b_\mu (\fw),\] and in the second $\mu b_\mu (\fv)\rcto \mu b_\mu (\fw)$; in either case, $\mu b_\mu(\fw)\wedge \mu b_\mu (\fv)\equiv \mu b_\mu (\fv)$. Similarly, if $b_\mu(\fv)\wle \mu b_\mu (\fw)$, then $\mu b_\mu(\fw)\wedge \mu b_\mu (\fv)\equiv \mu b_\mu (\fw)$. In either case,
\[\mu b_\mu(\fw)\wedge \mu b_\mu (\fv)\equiv \mu b_\mu (\fu_0)\]
for some worm $\fu_0 \in \{\fw, \fv\}$, and thus
\[\phi\wedge\psi\equiv (h_\mu(\fw)\wedge h_\mu(\fv)) \wedge (\mu b_\mu(\fw)\wedge \mu b_\mu (\fv)) \equiv \fu_1\wedge\mu \fu_0\equiv \fu_1\mathrel \mu\fu_0.\qedhere\]
\endproof
Below, we remark that $\fw\sqsubset\mu$ is equivalent to $\max\fw<\mu$.

\begin{corollary}\label{CorBound}
Let $\mu$ be an ordinal and $\top\not=\fw\in \Worms$. Then,

\begin{enumerate}

\item if $\fw\not=\top$ and $\mu<\max \fw$ then $\mu\top\wle{}\fw$,\label{CorBoundA}

\item if $\fw\not=\top$ and $\mu\leq \max \fw$ then $\mu\top\wleq{}\fw$, and\label{CorBoundB}

\item if $\fw\sqsubset \mu$ then $\fw\wle{}\mu\top$.\label{CorBoundC}

\end{enumerate}

\end{corollary}

\proof
For the first claim, proceed by induction on $\lgt\fw$. Write $\fw=\lambda\fv$ and consider two cases. If $\lambda\leq \mu$, by induction on length, $\mu\top\wle{ }\fv$, so $\mu\top\wle{ }\fv\wle{ }\fw$. Otherwise, $\lambda>\mu,$ so from $\fv\rcto\top$, $\lambda \top \rcto \mu \top$, and Lemma \ref{LemmWormConj} we obtain
\[\fw\rcto\lambda\top \wedge \mu \top \rcto \lambda \mu\top \rcto 0\mu\top.\]
The second claim is similar. Again, write $\fw=\lambda\fv$. If $\mu> \lambda$, we have inductively that $\mu \top \wleq{} \fv \wle{} \fw$. Otherwise, $\mu\leq \lambda$, in which case
\[\fw\rcto\lambda\top \rcto \mu\top,\]
and we may use Corollary \ref{CorImpLeq}.

For the third, we proceed once again by induction on $\lgt\fw$. The case for $\fw = \top$ is obvious. Otherwise, let $\eta=\min\fw$. Then, by the induction hypothesis, $h_\eta(\fw)\wle{}\mu\top=h_\eta(\mu\top)$, while also by the induction hypothesis $b_\eta(\fw) \wle{}\mu\top$, hence $\fw\wle{}\mu\top$ by Theorem \ref{TheoWormOrder}.
\endproof

\subsection{Well-orderedness of worms}

We have seen that $\wleq{\mu}$ is a total preorder, but in fact we have more; it is a pre-well-order. We will prove this using a Kruskal-style argument \cite{Kruskal1960}. It is very similar to Beklemishev's proof in \cite{Beklemishev:2005:VeblenInGLP}, although he uses normal forms for worms. Here we will use our `head-body' decomposition instead.

\begin{theorem}\label{TheoWormsWO}
For any ordinal $\lambda$ and any $\eta \leq\lambda$, $\wle\eta$ is a pre-well-order on $\mathbb W_{\geq\lambda}$.
\end{theorem}

\proof
We have already seen that $\mathbb W_\lambda$ is total in Theorem \ref{TheoWormOrder}, so it remains to show that there are no infinite $\wle\eta$-descending chains. We will prove this by contradiction, assuming that there is such a chain.

Let $\mathfrak w_0$ be any worm such that $\mathfrak w_0$ is the first element of some infinite descending chain $\mathfrak w_0\wge\eta\mathfrak v_1\wge\eta\mathfrak v_2\wge\eta\hdots$ and $\lgt{\mathfrak w_0}$ is minimal among all worms that can be the first element of such a chain. Then, for $i>0$, choose $\mathfrak w_i$ recursively by letting it be a worm such that there is an infinite descending chain
\[\mathfrak w_0\wge\eta\mathfrak w_1\wge\eta\hdots \wge\eta \mathfrak w_i\wge\eta\mathfrak v_{i+1}\wge\eta\hdots,\]
and such that $\lgt{\mathfrak w_i}$ is minimal among all worms with this property (where $\fw_j$ is already fixed for $j<i$). Let $\vec {\mathfrak w}$ be the resulting chain.

Now, let $\mu \geq \eta$ be the least ordinal appearing in $\vec{\mathfrak w}$, and define $h(\vec{\mathfrak w})$ to be the sequence
\[h_\mu(\mathfrak w_0),h_\mu(\mathfrak w_1),\hdots,h_\mu(\mathfrak w_i),\hdots\]
Let $j$ be the first natural number such that $\mu$ appears in $\fw_j$. By Lemma \ref{LemmDecomp}.\ref{LemmDecompItTwo}, $h_\mu(\fw_i)=\fw_i$ for all $i<j$, while $\lgt{h_\mu(\fw_j)}<\lgt{\fw_j}$, so by the minimality of $\lgt{\fw_j}$, $h(\vec{\mathfrak w})$ is not an infinite decreasing chain. Hence for some $k$, $h_\mu(\mathfrak w_k)\wgeq \eta h_\mu(\mathfrak w_{k+1})$.

Next, define $b(\vec{\mathfrak w})$ to be the sequence
\[\mathfrak w_0,\hdots,\mathfrak w_{k-1},b_\mu(\mathfrak w_k),\mathfrak w_{k+2},\mathfrak w_{k+3},\hdots\]
In other words, we replace $\mathfrak w_k$ by $b_\mu(\mathfrak w_k)$ and skip $\mathfrak w_{k+1}$. By the minimality of $\lgt{\mathfrak w_k}$, this cannot be a decreasing sequence, and hence $b_\mu(\mathfrak w_k)\wleq\eta \mathfrak w_{k+2}\wle\eta \mathfrak w_{k+1}$.

It follows from Theorem \ref{TheoWormOrder} that $\mathfrak w_k\wleq\eta \mathfrak w_{k+1}$, a contradiction. We conclude that there can be no decreasing sequence, and $\wle\eta$ is well-founded, as claimed.
\endproof

One consequence of worms being pre-well-ordered is that we can assign them an ordinal number measuring their order-type. In the next section we will make this precise.

\subsection{Order-types on a pre-well-order}

As we have mentioned, any well-order may be canonically represented using an ordinal number. To do this, if $\mathfrak A=\langle A,\preccurlyeq\rangle$ is any pre-well-order, for $a\in A$ define \[o(a)=\bigcup_{b\prec a}(o(b)+1).\]
Observe that $o$ is strictly increasing, in the following sense:

\begin{definition}
Let $\langle A,\preccurlyeq_A\rangle,\langle B,\preccurlyeq_B\rangle$ be preorders, and $f\colon A\to B$. We say that $f$ is {\em stricty increasing} if
\begin{enumerate}

\item  for all $x,y\in A$, $x\preccurlyeq_A y$ implies $f(x)\preccurlyeq_B f(y)$, and

\item  for all $x,y\in A$, $x\prec _A y$ implies $f(x)\prec _B f(y)$.

\end{enumerate}
\end{definition}

We note that if $\prec_A$ is total, then there are other equivalent ways of defining strictly increasing maps:

\begin{lemma}
If $\langle A,\preccurlyeq_A\rangle,\langle B,\preccurlyeq_B\rangle$ are total preorders and $f\colon A\to B$, then the following are equivalent:

\begin{enumerate}

\item $f$ is strictly increasing;

\item for all $x,y\in A$, $x\preccurlyeq_A y$ if and only if $f(x)\preccurlyeq_B f(y)$;

\item for all $x,y\in A$, $x\prec _A y$ if and only if $f(x)\prec _B f(y)$.

\end{enumerate}
\end{lemma}

\proof
Straightforward, using the fact that $a\prec_A b$ if and only if $b\not\preccurlyeq_A a$, and similarly for $\prec_B$.
\endproof

Then, the map $o$ can be characterized as the only strictly increasing, initial map $f\colon A\to\ord$, where $f\colon A\to B$ is {\em initial} if whenever $b\prec_B f(a)$, it follows that $b=f(a')$ for some $a'\prec_A a$:

\begin{lemma}\label{LemmOIff}
Let $\langle A,\preccurlyeq \rangle$ be a pre-well-order. Then,
\begin{enumerate}

\item\label{LemmOIffItOne} for all $x,y\in A$, $x\prec y$ if and only if $o(x)<o(y)$, and

\item\label{LemmOIffItTwo} $o\colon A\to \ord$ is an initial map.

\end{enumerate}

\end{lemma}

The proof proceeds by transfinite induction along $\prec$ and we omit it, as is the case of the proof of the following:

\begin{lemma}\label{LemmOUnique}

Let $\langle A,\preccurlyeq \rangle$ be a pre-well-order. Suppose that $f\colon A\to\ord$ satisfies

\begin{enumerate}

\item $x\prec y$ implies that $f(x)<f(y)$,

\item $x\preccurlyeq y$ implies that $f(x)\leq f(y)$, and

\item if $\xi\in f[A]$ then $\xi\subseteq f[A]$.

\end{enumerate}

Then, $f=o$.
\end{lemma}

Observe that $o(a)=o(b)$ implies that $a\preccurlyeq b$ and $b\preccurlyeq a$, i.e. $a\approx b$. Let us state this explicitly for the case of worms.

\begin{lemma}\label{LemmEquiv}
If $\fw,\fv$ are worms such that $o(\fw)=o(\fv)$, then $ \fw  \equiv  \fv $.
\end{lemma}

\proof
Reasoning by contrapositive, assume that $\fw\not\equiv \fv$. Then by Lemma \ref{LemmWormLinear}, either $\fw\wle{}\fv$, which implies that $o(\fw)< o(\fv)$, or $\fv\wle{}\fw$, and hence $o(\fv)< o(\fw)$. In either case, $o(\fw)\not= o(\fv)$.
\endproof

Computing $o(\fw)$ will take some work, but it is not too difficult to establish some basic relationships between $o(\fw)$ and the ordinals appearing in $\fw$.

\begin{lemma}\label{LemmLowerBound}
Let $\fw\not=\top$ be a worm and $\mu$ an ordinal. Then, 
\begin{enumerate}

\item  if $\mu\leq\max\fw$, then $\mu \leq o(\mu\top)\leq o(\fw)$, and\label{LemmLowerBoundItOne}

\item if $\max\fw<\mu$, then $o(\fw)<o(\mu\top)$.\label{LemmLowerBoundItTwo}

\end{enumerate}

\end{lemma}

\proof
First we proceed by induction on $\mu$ to show that $\mu\leq o(\mu\top)$. Suppose that $\eta<\mu $. Then by Corollary \ref{CorBound}, $\eta\top\wle{}\mu\top $, while by the induction hypothesis $\eta\leq o(\eta\top)$, and hence $\eta \leq o(\eta\top) < o(\mu\top)$. Since $\eta<\mu$ was arbitrary, $\mu \leq o(\mu\top)$. That $o(\mu\top)\leq o(\fw)$ if $\mu\leq\max\fw$ follows from Corollary \ref{CorBound}, since $\mu\top\wleq{}\fw$.

The second claim is immediate from Corollary \ref{CorBound}.\ref{CorBoundC}.
\endproof

Let us conclude this section by stating a useful consequence of the fact that $o\colon\Worms\to\ord$ is initial.

\begin{corollary}\label{CorSurj}
For every ordinal $\xi$ there is a worm $\fw\wleq{}\xi\top$ such that $\xi=o(\fw)$. 
\end{corollary}

\proof
By Lemma \ref{LemmLowerBound}, $\xi\leq o(\xi\top)$, so this is a special case of Lemma \ref{LemmOIff}.\ref{LemmOIffItTwo}.
\endproof

\section{Finite worms}\label{SecFiniteW}

In the previous section we explored some basic properties of $o$, but they are not sufficient to compute $o(\fw)$ for a worm $\fw$. In this section we will provide an explicit calculus for $o\upharpoonright \Worms_\omega$ (where $\upharpoonright$ denotes domain restriction). $\Worms_\omega$ is a particularly interesting case-study in that it has been used by Beklemishev for a $\Pi^0_1$ ordinal analysis of Peano arithmetic. Before we continue, it will be illustrative to sketch the relationship between $\Worms_\omega$ and $\pa$.

\subsection{First-order arithmetic}\label{SubsecFOA}

Expressions of $\RC_\omega$ have a natural proof-theoretical interpretation in first-order arithmetic. We will use the language $\Pi_\omega$ of first-order arithmetic containing the signature
\[\{{\tt 0, 1,  +, \cdot, 2^{\cdot},= }\}\]
so that we have symbols for addition, multiplication, and exponentiation, as well as Boolean connectives and quantifiers ranging over the natural numbers. Elements of $\Pi_\omega$ are {\em formulas.} The set of all formulas where all quantifiers are {\em bounded,} that is, of the form $\forall \, x{<}t \ \phi$ or $\exists\,  x{<}t \ \phi$ (where $t$ is any term), is denoted ${\Delta}_0$. A formula of the form $\exists x_n\forall x_{n-1}\hdots \delta(x_1,\hdots,x_n)$, with $\delta\in \Delta_0$, is $\Sigma_n$, and a formula of the form $\forall x_n\exists x_{n-1}\hdots \delta(x_1,\hdots,x_n)$ is $\Pi_n$. These classes are extended modulo provable equivalence, so that every formula falls into one of them. Note that the negation of a $\Sigma_n$ formula is $\Pi_n$ and vice-versa.

To simplify notation we may assume that some additional function symbols are available, although these are always definable from the basic arithmetical operations. In particular, we assume that we have for each $n$ a function $\langle x_1,\hdots,x_n\rangle$ coding a sequence as a single natural number.

In order to formalize provability within arithmetic, we fix some G\"odel numbering mapping a formula $\psi\in\Pi_\omega$ to its corresponding G\"odel number $\ulcorner \psi \urcorner$, and similarly for terms and sequences of formulas, which can be used to represent derivations. We also define the {\em numeral} of $n\in\mathbb N$ to be the term
\[\bar n={\tt 0}+\underbrace{\tt 1 + \hdots+1}_{n\text{ times}}.\]
In order to simplify notation, we will often identify $\psi$ with $\ulcorner \psi \urcorner$.

We will assume that every theory $T$ contains classical predicate logic, is closed under modus ponens, and that there is a ${ \Delta}_0$ formula ${\provfor}_T(x,y)$ which holds if and only if $x$ codes a derivation in $T$ of a formula coded by $y$. Using Craig's trick, any theory with a computably enumerable set of axioms is deductively equivalent to one in this form, so we do not lose generality by these assumptions.

If $\phi$ is a natural number (supposedly coding a formula), we use $\Box_T\phi$ as shorthand for $\exists y\ {\provfor}_T(y, {\bar\phi})$. We also write $\Box_T \phi(\dot x_0, \ldots, \dot x_n)$ as short for $\exists \psi\ (\psi = \phi(\bar x_0, \ldots, \bar x_n) \wedge \Box_T \psi)$. To get started on proving theorems about arithmetic, we need a minimal `background theory'. This will use Robinson's arithmetic $\rm Q$ enriched with axioms for the exponential; call the resulting theory $\Robinson$. To be precise, $\Robinson$ is axiomatized by classical first-order logic with equality, together with the following:

\begin{multicols}2
\begin{itemize}

\item $\forall x \ (x+{\tt 0}=x)$

\item $\forall x \ (x \not = {\tt 0}\leftrightarrow \exists y \ x=y+{\tt 1})$

\item $\forall x \forall y \ (x+{\tt 1}  =y+{\tt 1} \rightarrow x=y)$

\item $\forall x \forall y \ \big  ( x+(y+{\tt 1})=(x+y)+{\tt 1} \big )$

\item $\forall x \ (x\times {\tt 0}={\tt 0})$

\item $\forall x \forall y \ \big  ( x\times (y+{\tt 1})=(x\times y)+y \big )$

\item ${\tt 2}^{\tt 0}={\tt 1}$

\item $\forall x  \ \big  ( {\tt 2}^{x+{\tt 1}}={\tt 2}^{x}+{\tt 2}^{x} \big )$

\end{itemize}
\end{multicols}
Aside from these basic axioms, the following schemes will be useful in axiomatizing many theories of interest to us. Let $\Gamma$ to denote a set of formulas. Then, the induction schema for $\Gamma$ is defined by
\begin{center}
${\tt I} \Gamma$:\ \ $\phi({\tt 0})\wedge\forall x\big(\phi(x)\to\phi(x+{\tt 1})\big)\to\forall x\phi(x)$,\hskip 20pt where $\phi\in\Gamma$.
\end{center}
\emph{Elementary arithmetic} is the first-order theory
\[{\rm EA}=\Robinson+\mathrm{I}{\Delta}_0,\]
and \emph{Peano arithmetic} is the first-order theory
\[{\rm PA}=\Robinson+\mathrm{I}{\Pi}_\omega.\]

As usual, $\ps_T\phi$ is defined as $\neg\nc_T\neg\varphi$, and this will be used to interpret the $\RC$-modality $0$. Other modalities can be interpreted as stronger notions
of consistency. For this purpose it is very useful to consider the
provability predicates $[n]_T$, where $[n]_T$ is a natural first-order formalization of
``provable from the axioms of $T$ together with some true $\Pi_n$
sentence''. More precisely, let ${\tt True}_{\Pi_n}$ be the standard
partial truth-predicate for $\Pi_n$ formulas, which is itself of
complexity $\Pi_n$ (see \cite{HajekPudlak:1993:Metamathematics} for
information about partial truth definitions within $\ea$). Then, we
define
$$[n]_T\varphi\leftrightarrow \exists \pi \ \big ( {\tt
True}_{\Pi_n}(\pi)\wedge\nc_T (\pi\rightarrow \varphi) \big ).$$

\begin{definition}\label{DefArithInt}
Given a theory $T$, we then define $\cdot_T\colon \lan\RC\to\Pi_\omega$ given recursively by
\begin{enumerate}[label=(\roman*)]

\item $\top_T=\top$,

\item $(\phi\wedge\psi)_T=\phi_T\wedge\psi_T$, and

\item $(n\phi)_T=\langle n\rangle_{T}\phi_T$.

\end{enumerate}
\end{definition}

The next theorem follows from the arithmetical completeness of $\glp_\omega$ proven by Ignatiev \cite{Ignatiev:1993:StrongProvabilityPredicates} together with the conservativity of $\glp_\omega$ over $\RC_\omega$ (Theorem \ref{TheoGLPCons}).

\begin{theorem}
Let $T$ be any sound, representable extension of $\pa$. Given a formula $\phi$ of $\RC_\omega$, $\RC_\omega\vdash\phi$ if and only if $T\vdash \phi_T$.
\end{theorem}

We remark that Japaridze first proved a variant of this result, where $[n]_T$ is defined using iterated $\omega$-rules \cite{Japaridze:1988}. A similar interpretation will be discussed in Section \ref{SubsecOmegaRule} in the context of second-order arithmetic. However, the interpretation we have sketched using proof predicates has been used by Beklemishev to provide a consitency proof of Peano arithmetic as well as a $\Pi^0_1$ ordinal analysis. Here we will briefly sketch the consistency proof; for details, see \cite{Beklemishev:2004:ProvabilityAlgebrasAndOrdinals}.

The first step is to represent Peano arithmetic in terms of $n$-consistency:

\begin{theorem}\label{TheoPaDiamond}
It is provable in $\ea$ that
\[\pa\equiv \ea+\{\langle n\rangle_\ea\top : n <\omega \}.\]
\end{theorem}

This is a reformulation of a result of Kreisel and L\'evy \cite{KreiselLevy:1968:ReflectionPrinciplesAndTheirUse}, although they used {\em primitive recursive arithmetic} in place of $\ea$. The variant with $\ea$ is due to Beklemishev.

The consistency proof will be realized mostly within a `finitary base theory', $\ea^+$, which is only a bit stronger than $\ea$. To describe it, first define the {\em superexponential}, denoted $2^n_m$, to be the function given recursively by
\begin{enumerate*}[label=(\roman*)]
\item $2^n_0=2^n$ and

\item $2^n_{m+1}=2^{2^n_m}$.

\end{enumerate*}
Thus, $2^1_m$ denotes an exponential tower of $m$ $2$'s. Then, we let $\ea^+$ be the extension of $\ea$ with an axiom stating that the superexponential function is total. With this, we may enunciate Beklemishev's {\em reduction rule:}

\begin{theorem}\label{TheoReduct}
If $\fw\sqsubset\omega$ is any worm, then $\ea^+$ proves that
\[\big (\forall \fv\wle{}\fw \, ( \, \ps_\ea\fv_\ea \, ) \big ) \rightarrow \ps_\ea\fw_\ea.\]
\end{theorem}

This extends a previous result by Schmerl \cite{Schmerl:1978:FineStructure}. Meanwhile, the reader may recognize this as the premise of the {\em transfinite induction scheme} for worms. To be precise, if $\phi(x),x\prec y$ are arithmetical formulas, then the transfinite induction scheme for $\phi$ along $\prec$ is given by:
\[{\tt TI}_\prec(\phi)=\Big (\forall x \, \big ((\forall y\prec x \, \phi(y))\rightarrow \phi(x)\big ) \Big)\rightarrow \forall x\, \phi(x).\]
If $\Gamma$ is a set of formulas, then ${\tt TI}_\prec(\Gamma)$ is the scheme $\{{\tt TI}_\prec(\phi) : \phi\in\Gamma\}.$

Observe that $\ps_\ea\phi\in \Pi_1$ independently of $\phi$; with this in mind, we obtain the following as an immediate consequence of Theorem \ref{TheoReduct}:
 
\begin{theorem}\label{TheoPACons} $\ea^+ + {\tt TI}_{\wle{}\upharpoonright\Worms_\omega}(\Pi _1) \vdash \ps_\pa\top.$
\end{theorem}

In words, we can prove the consistency of Peano arithmetic using $\ea^+$ and transfinite induction along $\langle \Worms_\omega,\wle{}\rangle $. In fact, we use only one instance of transfinite induction for a predicate $\phi(x)$ expressing ``$x\sqsubset\omega$ and $\ps_\ea x_\ea$''.

Compare this to Gentzen's work \cite{Gentzen1936}, where he proves the consistency of Peano arithmetic with transfinite induction up to the ordinal $\varepsilon_0$. In the remainder of this section, we will see how finite worms and $\varepsilon_0$ are closely related.

\subsection{The ordinal $\varepsilon_0$}

The ordinal $\varepsilon_0$ is naturally defined by extending the arithmetical operations of addition, multiplication and exponentiation to the transfinite. In view of Lemma \ref{LemmOrdSucc}, we may have to consider not only successor ordinals, but also unions of ordinals. Fortunately, these operations are exhaustive.

\begin{lemma}\label{LemmOrdClass}
Let $\xi$ be an ordinal. Then, exactly one of the following occurs:

\begin{enumerate}[label=(\roman*)]

\item $\xi=0$;

\item there exists $\zeta$ such that $\xi=\zeta+1$, in which case we say that $\xi$ is a {\em successor;} or

\item $\xi=\bigcup_{\zeta<\xi}\zeta$, in which case we say that $\xi$ is a {\em limit.}

\end{enumerate}
\end{lemma}

Thus we may recursively define operations on the ordinals if we consider these three cases. For example, ordinal addition is defined as follows:

\begin{definition}
Given ordinals $\xi,\zeta$, we define $\xi+\zeta$ by recursion on $\zeta$ as follows:
\begin{enumerate}

\item $\xi+0=\xi$

\item $\xi+(\zeta+1)=(\xi+\zeta)+1$

\item $\xi+\zeta=\displaystyle\bigcup_{\vartheta<\zeta}(\xi+\vartheta)$, for $\zeta$ a limit ordinal.

\end{enumerate}

\end{definition}

Ordinal addition retains some, but not all, of the properties of addition on the natural numbers; it is associative, but not commutative. For example, $1+\omega=\omega<\omega+1$, and more generally $1+\xi=\xi<\xi+1$ whenever $\xi$ is infinite. We also have a form of subtraction, but only on the left:

\begin{lemma}\label{theorem:BasicPropertiesOrdinalArithmetic}
If $\zeta {<} \xi$ are ordinals, there exists a unique $\eta$ such that $\zeta + \eta = \xi.$
\end{lemma}
The proof follows by a standard transfinite induction on $\xi$. We will denote this unique $\eta$ by $-\zeta + \xi$. It will be convenient to spell out some of the basic properties of left-subtraction:

\begin{lemma}\label{LemmLeftSubt}
Let $\alpha,\beta,\gamma$ be ordinals. Then:

\begin{enumerate}[label=(\roman*)]

\item $-0+\alpha=\alpha$ and $-\alpha+\alpha=0;$

\item if $\alpha\leq \beta$ and $-\alpha+\beta \leq \gamma$ then $-\alpha+(\beta+\gamma)=(-\alpha+\beta)+\gamma;$

\item\label{LemmLeftSubtItLast} if $\alpha+\beta\leq\gamma$ then $-\beta +(-\alpha+\gamma)=-(\alpha+\beta)+\gamma;$

\item if $\alpha\leq\beta\leq\alpha+\gamma$ then $-\beta+(\alpha+\gamma)=-(-\alpha+\beta)+\gamma.$

\end{enumerate}
\end{lemma}

\proof
These properties are proven using the associativity of addition and the fact that $-\mu+\lambda$ is unique. We prove only \ref{LemmLeftSubtItLast} as an example. Observe that
\begin{align*}
(\alpha+\beta)+(-\beta +(-\alpha+\gamma))=\alpha+(\beta+(-\beta+(-\alpha+\gamma)))\\
= \alpha+(-\alpha+\gamma)=\gamma;
\end{align*}
but $-(\alpha+\beta)+\gamma$ is the unique $\eta$ such that $(\alpha+\beta)+\eta=\gamma$, so we conclude that \ref{LemmLeftSubtItLast} holds. The other properties are proven similarly.
\endproof

The definition of addition we have given can be used as a template to generalize other arithmetical operations. Henceforth, if $\langle \mu_\xi\rangle_{\xi<\lambda}$ is an increasing sequence of ordinals, we will write $\lim_{\xi<\lambda}\mu_\xi$ instead of $\bigcup_{\xi<\lambda}\mu_\xi$.

\begin{definition}
Given ordinals $\xi,\zeta$, we define $\xi\cdot \zeta$ by recursion on $\zeta$ as follows:
\begin{enumerate}

\item $\xi\cdot 0=0$,

\item $\xi\cdot(\zeta+1)=\xi\cdot\zeta+\xi$, and

\item $\xi\cdot\zeta=\displaystyle\lim_{\vartheta<\zeta}\xi\cdot\vartheta$, for $\zeta$ a limit ordinal.

\end{enumerate}
Similarly, we define $\xi^ \zeta$ by:

\begin{enumerate}

\item $\xi^0=1$,

\item $\xi^{\zeta+1}=\xi^\zeta\cdot\xi$, and

\item $\xi^\zeta=\displaystyle\lim_{\vartheta<\zeta}\xi^\vartheta$, for $\zeta$ a limit ordinal.

\end{enumerate}

\end{definition}

Addition, multiplication and exponentiation give us our first examples of {\em normal functions.} These are functions that are increasing and continuous, in the following sense:

\begin{definition}
A function $f\colon \ord\to\ord$ is {\em normal} if:
\begin{enumerate}

\item whenever $\xi<\zeta$, it follows that $f(\xi)<f(\zeta)$, and

\item whenever $\lambda$ is a limit ordinal, $f(\lambda)=\displaystyle\lim_{\xi<\lambda}f(\xi)$.

\end{enumerate}
\end{definition}

Normal functions are particularly nice to work with. Among other things, they have the following property, proven by an easy transfinite induction:

\begin{lemma}
If $f\colon\ord\to\ord$ is normal, then for every ordinal $\xi$, $\xi\leq f(\xi)$.
\end{lemma}

Of course this does not rule out the possibility that $\xi=f(\xi)$, and in fact the identity function is an example of a normal function. As we have mentioned, the elementary arithmetical functions give us further examples:

\begin{lemma}
Let $\alpha$ be any ordinal. Then, the functions $f,g,h\colon \ord\to\ord$ given by
\begin{enumerate}

\item $f(\xi)=\alpha+\xi$,

\item $g(\xi)=(1+\alpha)\cdot \xi$,

\item $h(\xi)=(2+\alpha)^\xi$

\end{enumerate}
are all normal.

\end{lemma}

Note, however, that the function $\xi\mapsto\xi+\alpha$ is not normal in general, and neither are $\xi\mapsto 0\cdot\xi$, $\xi\mapsto 1^\xi$. But $\xi\mapsto\omega^\xi$ is normal, and this function is of particular interest, since it is the basis of the Cantor normal form representation of ordinals (similar to a base-$n$ representation of natural numbers), where we write
\[\xi=\omega^{\alpha_n}+\hdots +\omega^{\alpha_0}\]
with the $\alpha_i$'s non-decreasing. Moreover, the ordinals of the form $\omega^\beta$ are exactly the {\em additively indecomposable} ordinals; that is, non-zero ordinals that cannot be written as the sum of two smaller ordinals. Let us summarize some important properties of this function:

\begin{lemma}\label{LemmCantorNormal}
Let $\xi\not=0$ be any ordinal. Then:
\begin{enumerate}

\item There are ordinals $\alpha,\beta$ such that $\xi=\alpha+\omega^\beta$. The value of $\beta$ is unique.

\item We can take $\alpha=0$ if and only if, for all $\gamma,\delta<\xi$, we have that $\gamma+\delta<\xi$.

\end{enumerate}

\end{lemma}

We call this the {\em Cantor decomposition} of $\xi$. Cantor decompositions can often be used to determine whether $\xi<\zeta$:

\begin{lemma}\label{LemmCantorOrder}
Given ordinals $\xi=\alpha+\omega^\beta$ and $\zeta=\gamma+\omega^\delta$,

\begin{multicols}2

\begin{enumerate}

\item $\xi < \zeta$ if and only if \label{LemmCantorOrderItLe}

\begin{enumerate}

\item $\xi\leq \gamma$, or\label{LemmCantorOrderItLeBody}

\item $\alpha<\zeta$ and $\beta < \delta$, and\label{LemmCantorOrderItLeHead}

\end{enumerate}

\item $\xi\leq \zeta$ if and only if\label{LemmCantorOrderItLeq}

\begin{enumerate}

\item $\xi\leq \gamma$, or\label{LemmCantorOrderItLeqBody}

\item $\alpha<\zeta$ and $\beta \leq \delta$.\label{LemmCantorOrderItLeqHead}

\end{enumerate}

\end{enumerate}

\end{multicols}

\end{lemma}

Note, however, that this decomposition is only useful when $\beta<\xi$ or $\gamma<\zeta$, which as we will see is not always the case. In particular, the ordinal $\varepsilon_0$ is the first ordinal such that $\varepsilon_0=\omega^{\varepsilon_0}$. Roughly, it is defined by beginning with $0$ and closing under the operation $\langle\alpha,\beta\rangle\mapsto \alpha+\omega^\beta$. Since many proof-theoretical ordinals are defined by taking the closure under a family of functions, it will be convenient to formalize such a closure with some generality.

The general scheme is to consider a family of ordinal functions $f_1,\hdots,f_n$, then considering the least ordinal $\xi$ such that $f_i(\alpha_1,\hdots,\alpha_m)<\xi$ whenever each $\alpha_i<\xi$. To simplify our presentation, let us make a few preliminary observations:
\begin{enumerate}

\item The functions $f_i$ may be partial or total. Since a total function is a special case of a partial function, we may in general consider $f_i\colon \ord^m\dashrightarrow \ord$ (where $f\colon A\dashrightarrow B$ indicates that $f$ is a partial function).

\item We may have functions with fixed or variable arity. Given a class $A$, let $A^{<\omega}$ denote the class of finite sequences $\langle a_1,\hdots,a_m\rangle$ with $m<\omega$ and each $a_i\in A$. An ordinal function with fixed arity $m$ may be regarded as a partial function on $\ord^{<\omega}$, whose domain is $\ord^m\subseteq \ord^{<\omega}$. Thus without loss of generality, we may assume that all partial functions have variable arity.

\item We may represent the family $f_1,\hdots,f_n$ as a single function by setting
\[f(i,\alpha_1,\hdots,\alpha_m)=f_i(\alpha_1,\hdots,\alpha_m).\]
Note that this idea can also be used to represent infinite families of functions as a single function.

\end{enumerate}
Thus we may restrict our discussion to ordinals closed under a single partial function of variable arity, and will do so in the next definition.

\begin{definition}\label{DefFClose}
Let $f\colon \ord^{<\omega}\dashrightarrow \ord$ be a partial function. Given a set of ordinals $\Theta$, define $\iter f\Theta$ to be the set of all ordinals $\lambda$ such that there exist $\mu_1,\hdots,\mu_n\in\Theta$ (possibly with $n=0$) such that $\lambda=f(\mu_1,\hdots,\mu_n)$.

For $n<\omega$, define inductively $\Theta^f_0=\Theta$ and $\Theta^f_{n+1}=\Theta^f_n\cup\iter f{(\Theta^f_n)}$. Then, define
\[\close f\Theta=\bigcup_{n<\omega}\Theta^f_n.\]
\end{definition}

The set $\close f\Theta$ is the {\em closure of $\Theta$ under $f$,} and indeed behaves like a standard closure operation:

\begin{lemma}\label{LemmPropFClose}
Let $f\colon \ord^{<\omega}\dashrightarrow \ord$ and let $\Theta$ be any set of ordinals. Then,
\begin{enumerate}

\item $\Theta\cup\iter f{(\close f\Theta)}\subseteq \close f\Theta$,

\item if $\Theta\cup \close f\Xi\subseteq \Xi$ then $\close f\Theta\subseteq \Xi$, and

\item for any ordinal $\lambda$, $\lambda \in (\close f\Theta)\setminus\Theta$ if and only if there are $\mu_1,\hdots,\mu_n\in \close f\Theta\setminus \{\lambda\}$ with $\lambda=f(\mu_1,\hdots,\mu_n)$.\label{LemmPropFCloseItFour}

\end{enumerate}

\end{lemma}

\proof
For the first item, note that if $\lambda_1,\hdots,\lambda_n\in \close f\Theta$ then $\lambda_1,\hdots,\lambda_n\in \Theta^f_m$ for $m$ large enough and hence $f( \lambda_1,\hdots,\lambda_n) \in \Theta^f_{m+1}\subseteq \close f\Theta$.
The second follows by showing indutively that $\Theta^f_n\subseteq \Xi$ for all $n$, hence $\close f\Theta\subseteq \Xi$.
For the third, assume otherwise, and consider $\Xi=\close f\Theta\setminus \{\lambda\}$. One can readily verify that $\Theta\cup \iter f\Xi\subseteq\Xi$, contradicting the previous item.
\endproof

With this, we are ready to define the ordinal $\varepsilon_0$. Below, recall that we are following the standard set-theoretic convention that $1=\{0\}$.

\begin{definition}
Define ${\rm Cantor}\colon \ord^2\to \ord$ by ${\rm Cantor}(\alpha,\beta)=\alpha+\omega^\beta$. Then, we define
\[\varepsilon_0=\close {{\rm Cantor}} 1.\]
\end{definition}

As promised, $ \varepsilon_0$ is the first fixed-point of the function $\xi\mapsto\omega^\xi$:

\begin{theorem}\label{TheoEpxilon}
The set $\varepsilon_0$ is an ordinal and satisfies the identity $\varepsilon_0=\omega^{\varepsilon_0}$. Moreover, if $0<\xi<\varepsilon_0$, there are $\alpha,\beta<\xi$ such that $\xi=\alpha+\omega^\beta$.
\end{theorem}

\proof
First we will show that if $0<\xi\in\varepsilon_0$, then there are $\alpha,\beta<\xi$ such that $\xi=\alpha+\omega^\beta$. By Lemma \ref{LemmPropFClose}.\ref{LemmPropFCloseItFour}, there are $\alpha,\beta\in\varepsilon_0$ with $\alpha,\beta\not=\xi$ and such that $\xi=\alpha+\omega^\beta$. Since $\omega^\beta>0$ it follows that $\alpha<\xi$, and since $\beta\leq \omega^\beta\leq \xi$ it follows that $\beta\leq \xi$; but $\beta\not=\xi$, so $\beta<\xi$.

Now, since every element of $\varepsilon_0$ is an ordinal, in view of Lemma \ref{LemmOrdBasic}, in order to show that $\varepsilon_0$ is also an ordinal it suffices to show that if $\xi<\zeta\in\varepsilon_0$, then $\xi\in\varepsilon_0$. We proceed by induction on $\zeta$ with a secondary induction on $\xi$. Write $\zeta=\alpha+\omega^\beta$ and $\xi=\gamma+\omega^\delta$ with $\alpha,\beta\in\varepsilon_0\cap \zeta$. Since $\xi<\zeta$, by Lemma \ref{LemmCantorOrder}, we have that either $\xi\leq\alpha$ or $\gamma<\zeta$ and $\delta<\beta$. In the first case, our induction hypothesis applied to $\alpha<\zeta$ gives us $\xi\in \varepsilon_0$, in the second the secondary induction hypothesis on $\gamma<\xi$ gives us $\gamma\in\varepsilon_0$ and the induction hypothesis on $\beta<\zeta$ gives us $\delta\in \varepsilon_0$, hence $\xi=\alpha+\omega^\beta\in\varepsilon_0$.
\endproof

\subsection{Order-types of finite worms}\label{SubsecFino}

Our work on elementary ordinal operations and the ordinal $\varepsilon_0$ will suffice to compute the order-types of `finite' worms, i.e., worms where every entry is finite. In order to give a calculus for these order-types, we will need to consider, in addition to concatenation, `promotion' ($\uparrow$) and `demotion' ($\downarrow$) operations on worms. Below, let us write $\lan{\geq\lambda}$ for the sublanguage of $\lan{\RC}$ which only contains modalities $\xi\geq\lambda$.

\begin{definition}
Let $\phi\in\lan\RC$ and $\lambda$ be an ordinal. We define $\lambda\uparrow\phi$ to be the result of replacing every ordinal $\xi$ appearing in $\phi$ by $\lambda+\xi$. Formally, $\lambda\uparrow\top=\top$, $\lambda\uparrow(\phi\wedge\psi)=(\lambda\uparrow\psi)\wedge(\lambda\uparrow\psi)$, and $\lambda\uparrow\mu\phi=\langle\lambda+\mu\rangle(\lambda\uparrow\phi)$.

If $\phi\in\lan{\geq\lambda}$, we similarly define $\lambda\downarrow\phi$ by replacing every occurrence of $\xi$ by $-\lambda+\xi$.
\end{definition}

The relationship between $\uparrow$ and $\downarrow$ is analogous to that between ordinal addition and subtraction. The following are all straightforward consequences of Lemma \ref{LemmLeftSubt} and we omit the proofs.

\begin{lemma}\label{LemmUparroAlg}
Let $\alpha,\beta$ be ordinals and $\phi\in \lan{\RC}$. Then,

\begin{enumerate}[label=(\roman*)]

\item $0\promote\phi=\phi$;

\item $\alpha \uparrow (\beta \uparrow \phi) = (\alpha + \beta)\uparrow \phi$;\label{LemmUparroAlgItPlus}

\item if $\phi\in\lan{\geq \beta+\alpha}$ then $\alpha \downarrow (\beta \downarrow \phi) = (\beta + \alpha)\downarrow \phi$;

\item if $\alpha\leq \beta$ then
$\alpha\downarrow (\beta\promote\phi)=(-\alpha+\beta)\promote\phi,$ and

\item if $\alpha\leq \beta$ and $\phi\in\lan{\geq -\alpha+\beta}$ then $\alpha\promote\phi \in \lan{\geq \beta}$ and
\[\beta\downarrow (\alpha\promote\phi)=(-\alpha+\beta)\downarrow \phi.\]

\end{enumerate}
\end{lemma}

The operation $\phi\mapsto\lambda\promote\phi$ is particularly interesting in that it provides a sort of self-embedding of $\RC$:

\begin{lemma}\label{LemmRCUparrow}
Let $\alpha,\beta$ be ordinals and $\phi,\psi\in\lan{\RC}$. If $\phi\rcto\psi$ is derivable in $\RC$, then so is $(\lambda\promote\phi)\rcto(\lambda\promote\psi)$.
\end{lemma}

\proof
By induction on the length of a derivation of $\phi\rcto\psi$; intuitively, one replaces every formula $\theta$ appearing in the derivation by $\lambda\uparrow\theta$. The details are straightforward and left to the reader.
\endproof

The promotion operator gives us an order-preserving transformation on the class of worms:

\begin{lemma}\label{LemmUparrow}
Given a worm $\mathfrak w\in\Worms_{\geq\mu}$ and an ordinal $\lambda$, the following are equivalent:

\begin{enumerate}[label=(\roman*)]
\item $\mathfrak w\wle\mu\mathfrak v$;\label{LemmUparrowItOne}

\item $\lambda\uparrow \mathfrak w\wle\mu \lambda\uparrow \mathfrak v$, and\label{LemmUparrowItTwo}

\item $\lambda\uparrow \mathfrak w\wle\lambda \lambda\uparrow \mathfrak v$.\label{LemmUparrowItThree}

\end{enumerate}
\end{lemma}

\proof
The equivalence between \ref{LemmUparrowItTwo} and \ref{LemmUparrowItThree} is immediate from Lemma \ref{LemmSameOrders}, so we focus on the equivalence between \ref{LemmUparrowItOne} and \ref{LemmUparrowItThree}.

If $\mathfrak w\wle\mu\mathfrak v$, then $\fw\wle{}\fv$, so $\RC$ derives $\fv\rcto 0\fw$. By Lemma \ref{LemmRCUparrow}, $\RC$ also derives $(\lambda\promote\fv)\rcto\lambda(\lambda\promote\fw)$, that is, $(\lambda\promote\fw)\wle\lambda (\lambda\promote\fv)$.

Conversely, if $\lambda\uparrow \mathfrak w\wle\lambda \lambda\uparrow \mathfrak v$, assume towards a contradiction that $\fw\not\wle{\mu}\fv$, so that by Lemma \ref{LemmWormLinear}, $\fv\wleq{\mu}\fw$. Again by Lemma \ref{LemmRCUparrow}, $(\lambda\uparrow\fv)\wleq\lambda(\lambda\uparrow\fw)$, so $(\lambda\uparrow\fv)\wleq\lambda(\lambda\uparrow\fw)\wle\lambda(\lambda\uparrow\fv)$, contradicting irreflexivity.
\endproof

Lemma \ref{LemmUparrow} is useful for comparing worms; if we wish to settle whether $\lambda \uparrow\fw\wle{} \lambda \uparrow\fv$, then it suffices to check whether $\fw\wle{}\fv$. More generally, we obtain the following variant of Theorem \ref{TheoWormOrder}. Below, recall that we write $h,b$ instead of $h_0,b_0$.

\begin{lemma}\label{LemmWormOrder}

Given worms $\mathfrak w,\mathfrak v \neq \top$,

\begin{enumerate}

\item $\mathfrak w\wle { } \mathfrak v$ if and only if
\begin{enumerate}

\item $\mathfrak w\wleq { } b (\mathfrak v)$, \ or

\item  $b (\mathfrak w)\wle { } \mathfrak v$ \  and \ $ 1 \downarrow h (\mathfrak w)\wle { } 1 \downarrow  h (\mathfrak v);$

\end{enumerate}

\item $\mathfrak w\wleq { } \mathfrak v$ if and only if
\begin{enumerate}

\item $\mathfrak w\wleq { } b (\mathfrak v)$, \ or

\item $b (\mathfrak w)\wle { } \mathfrak v$ \ and \ $ 1 \downarrow  h (\mathfrak w)\wleq { } 1 \downarrow  h (\mathfrak v).$

\end{enumerate}

\end{enumerate}

\end{lemma}

If all entries of $\fv\not=\top$ are natural numbers, $1\downarrow h(\fw)$ will be `smaller' than $\fw$. To be precise, it will have a smaller {\em $1$-norm,} defined as follows:

\begin{definition}
We define $\nrmone\cdot\colon\Worms_\omega\to \omega$ recursively by
\begin{enumerate}

\item $\nrmone\top=0$;

\item if $\mathfrak w\not=\top$ and $\min\fw=0$,
\[\nrmone\fw=\nrmone{h(\fw)}+\nrmone{b(\fw)}+ 1;\]

\item if $\mathfrak w\not=\top$ and $\min\fw>0$,
\[\nrmone\fw =\nrmone{1\downarrow\fw}+ 1.\]

\end{enumerate}

\end{definition}

Recall that we use $h$ and $b$ as shorthands for $h_0$, $b_0$.

\begin{lemma}\label{LemmHBLess}
For every worm $\mathfrak w\sqsubset\omega$ with $\fw\not=\top$,
\begin{enumerate}

\item $\nrmone{b(\fw)}<\nrmone{\fw}$, and

\item $\nrmone{1\downarrow h(\fw)} < \nrmone{\fw}$.

\end{enumerate}

\end{lemma}

\proof
For the first claim, note that if $0$ appears in $\fw$ then $\nrmone{b(\fw)}+1\leq \nrmone{\fw}$. If $0$ does not appear, $\nrmone{b(\fw)}=0<\nrmone\fw$.

For the second, if $h(\fw)=\top$ then once again $\nrmone{1\downarrow h(\fw)}=0<\nrmone\fw$, and if $h(\fw)\not =\top$ then
\[\nrmone{1\downarrow h(\fw)}+1=\nrmone{h(\fw)}\leq \nrmone\fw,\]
so $\nrmone{1\downarrow h(\fw)} < \nrmone{h(\fw)}\leq \nrmone\fw$.
\endproof

We remark that there are other possible ways to define $\nrmone\cdot$ that would also satisfy Lemma \ref{LemmHBLess}; for example, we can define $\|\fw\|'_1=\lgt\fw+\max\fw,$ or
\[\|m_1\hdots m_n\top\|''_1=\sum_{i=1}^n (m_i+1).\]
However, these definitions do not generalize well to worms with transfinite entries, which will be the focus of Section \ref{SecTransW}. On the other hand, our norm $\nrmone\cdot$ can be applied to transfinite worms with only a minor modification.

Our goal now is to give an explicit calculus for computing $o(\fw)$ if $\fw\sqsubset\omega$. In view of Lemma \ref{LemmOUnique}, it is sufficient to propose a candidate function for $o$ and show that it has the required properties. Now, if we compare Lemma \ref{LemmWormOrder} with Lemma \ref{LemmCantorOrder}, we observe that the clauses for checking whether $\fw\wle{}\fv$ in terms of
\[b(\fw), 1\downarrow h(\fw), b(\fv), 1\downarrow h(\fv)\]
are analogous to the clauses for checking whether $\alpha+\omega^\beta<\gamma+\omega^\delta$ in terms of $\alpha,\beta,\gamma,\delta$, respectively. This suggests that
\begin{equation}\label{EqFinoCantor}
o(\fw)=ob(\fw)+\omega^{o(1\downarrow h(\fw))},
\end{equation}
and we will use this idea to define our `candidate function'.

\begin{definition}
Let $\mathfrak v,\mathfrak w $ be worms and $\alpha $ an ordinal.

Then, define a map $\fino\colon\Worms_\omega\to\ord$ by
\begin{enumerate}
\item $\fino(\top)=0,$ and
\item if $\fw\not=\top$ then \ \ $\fino(\fw)={\fino(b(\fw))}+\omega^{\fino(1\downarrow h(\fw))}.$
\label{second}
\end{enumerate}
\end{definition}

First, let us check that $\fino$ is indeed a function:

\begin{lemma}\label{LemmFinoWellD}
The map $\fino$ is well-defined.
\end{lemma}

\proof
This follows from an easy induction on $\nrmone\fw$ using Lemma \ref{LemmHBLess}.
\endproof

It remains to check that $\fino$ is strictly increasing and initial. Let us begin with the former:

\begin{lemma}\label{LemmFinoMon}
The map $\fino\colon \Worms_\omega\to \ord$ is strictly increasing.
\end{lemma}

\proof
We will prove by induction on $\nrmone\fw+\nrmone\fv$ that $\fw\wle  {}\fv$ if and only if $\fino(\fw) < \fino(\fv)$.
Note that $\fw\wle{}\top$ is never true, nor is $\xi< \fino(\top)=0$, so we may assume that $\fv\not=\top$. Then, if $\fw=\top$ it follows that $\fino(\top)=0$, so both sides are true. Hence we may also assume that $\fw\not=\top$. 

By Lemma \ref{LemmWormOrder}, $\fw\wle{}\fv$ if and only if either $\fw\wleq {}b(\fv)$ or $b(\fw)\wle{}\fv$ and ${1\downarrow h(\fw)}\wle{} {1\downarrow h(\fv)}$. Observe that, by the induction hypothesis,
\begin{enumerate}

\item $\fw\wleq{}b(\fv)$ if and only if $\fino(\fw)\leq \fino b(\fv) $, since
\[\nrmone\fw+\nrmone{b(\fv)}<\nrmone\fw+\nrmone{\fv};\]

\item $b(\fw)\wle{} \fv$ if and only if $\fino b(\fw) < \fino(\fv)$, since
\[\nrmone {b(\fw)}+\nrmone{ \fv }<\nrmone\fw+\nrmone{\fv},\]
and

\item ${1\downarrow h(\fw)}\wle{} {1 \downarrow  h(\fv)}$ if and only if $\fino(1\downarrow h(\fw )) < \fino (1\downarrow h(\fv))$, since
\[\nrmone {1\downarrow h(\fw)}+\nrmone{1 \downarrow h(\fv) }<\nrmone\fw+\nrmone{\fv}.\]

\end{enumerate}
This implies that $\fw\wle{}\fv$ if and only if either $\fino (\fw) \leq \fino b(\fv)$, or $\fino b(\fw)< \fino (\fv)$ and $\fino (1\downarrow h(\fw)) < \fino (1\downarrow h(\fv))$. But by Lemma \ref{LemmCantorOrder}.\ref{LemmCantorOrderItLe}, the latter is equivalent to
\[{\fino b(\fw)}+\omega^{\fino(1\downarrow h(\fw))} < {\fino b(\fv)}+\omega^{\fino(1\downarrow h(\fv))},\]
i.e., $\fino (\fw) <\fino (\fv)$.
\endproof

It remains to check that the range of $\fino$ is $\varepsilon_0$. We will use the following lemma:

\begin{lemma}\label{LemmMBound}
For all $m<\omega$, $\fino(m\top)<\varepsilon_0$.
\end{lemma}

\proof
By induction on $n$; if $n=0$ then $\fino(0\top)=0+\omega^0=1<\varepsilon_0$. Otherwise, by induction hypothesis $\fino(n\top)<\varepsilon_0$, so
\[\fino(\langle n+1\rangle \top)=\omega^{\fino(n\top)}<\varepsilon_0,\]
as claimed.
\endproof

\begin{lemma}\label{LemmFinoEpsilon}
An ordinal $\xi$ lies in the range of $\fino$ if and only if $\xi<\varepsilon_0$.
\end{lemma}

\proof

First, assume that $\xi<\varepsilon_0$; we must find $\fw\sqsubset\omega$ such that $\xi=\fino(\fw)$. Proceed by induction on $\xi$. If $\xi=0$, then $\xi=\fino(\top)$. Otherwise, by Theorem \ref{TheoEpxilon}, $\xi=\alpha+\omega^\beta$ for some $\alpha,\beta<\xi$. By the induction hypothesis, there are worms $\fu,\fv$ such that $\alpha=\fino(\fu)$ and $\beta=\fino(\fv)$, thus
\[\fino((1\promote \fv)\mathrel 0 \fu)=\fino(\fu)+\omega^{\fino(\fv)}=\alpha+\omega^\beta=\xi.\]

Next we check that if $\fw\sqsubset \omega$, then $\fino(\fw)<\varepsilon_0$. Fix $M>\max\fw$; then, by Corollary \ref{CorBound}.\ref{CorBoundC}, $\fw\wle{}M\top$, so that $\fino(\fw)\wle{}\fino(M\top)$. But by Lemma \ref{LemmMBound}, $\fino(M\top)<\varepsilon_0$, as claimed.
\endproof

We now have all the necessary ingredients to show that $\fino=o$.

\begin{lemma}\label{LemmFinoIsO}
For all $\fw\sqsubset\omega$, $o(\fw)=\fino(\fw)$.
\end{lemma}

\proof
By Lemma \ref{LemmFinoWellD}, $\fino$ is well-defined on $\Worms_\omega$, and by Lemmas \ref{LemmFinoMon} and \ref{LemmFinoEpsilon}, it is strictly increasing and initial. By Lemma \ref{LemmOUnique}, $o =\fino$ on $\Worms_\omega$.
\endproof

Let us conclude this section by summarizing our main results:

\begin{theorem}\label{TheoFiniteO}
The map $o\colon \Worms_\omega\to\varepsilon_0$ is surjective and satisfies
\begin{enumerate}
\item $o(\top)=0$, and\label{TheoFiniteOItOne}
\item $o((1\uparrow \fv) \mathrel 0\mathfrak w)={o(\mathfrak w)}+\omega^{o(\fv)}$.\label{TheoFiniteOItTwo}
\end{enumerate}
\end{theorem}

\proof
Immediate from Lemma \ref{LemmFinoIsO} and the definition of $\fino$.
\endproof

\section{Transfinite worms}\label{SecTransW}

We have now seen that finite worms give a notation for $\varepsilon_0$, the proof-theoretic ordinal of Peano arithmetic. However, stronger theories, including many important theories of reverse mathematics, have much larger proof-theoretic strength, suggesting that $\RC_\omega$ is not suitable for their $\Pi^0_1$ ordinal analysis. Fortunately, Theorem \ref{TheoWormsWO} is valid even when worms have arbitrary ordinal entries. In this section, we will extend Theorem \ref{TheoFiniteO} to all of $\Worms$.

\subsection{Subsystems of second-order arithmetic}

Let us begin by discussing proof-theoretic interpretations of $\RC_\Lambda$ with $\Lambda>\omega$. It will be convenient to pass to the language $\Pi^1_\omega$ of second-order arithmetic. This language extends that of first-order arithmetic with new variables $X,Y,Z,\hdots$ denoting sets of natural numbers, along with new atomic formulas $t\in X$ and second-order quantifiers $\forall X,\exists X$. As is standard, we may define $X\subseteq Y$ by $\forall x (x\in X\rightarrow x\in Y)$, and $X = Y$ by $X\subseteq Y\wedge Y\subseteq X$.

When working in a second-order context, we write $\Pi^0_n$ instead of $\Pi_n$ (note that these formulas could contain second-order parameters, but no quantifiers over sets). The classes ${\Sigma}^1_n,{\Pi}^1_n$ are defined analogously to their first-order counterparts, but using alternating second-order quantifiers and setting ${\Sigma}_0^1 = {\Pi}^1_0 = {\Delta}^1_0 = {\Pi}^0_\omega$. It is well-known that every second-order formula is equivalent to another in one of the above forms.

When axiomatizing second-order arithmetic, the focus passes from induction to {\em comprehension;} that is, axioms stating the existence of sets whole elements satisfy a prescribed property. Some important axioms and schemes are:

\begin{description}

\item[$\Gamma\mbox{-}\compax$:] $\exists X\forall x\ \big (x\in X\leftrightarrow \phi(x)\big )$, where $\phi\in\Gamma$ and $X$ is not free in $\phi$;

\item[${ \Delta}^0_1\mbox{-}\compax$:] $\forall x \big (\pi(w)\leftrightarrow\sigma(x) \big )\rightarrow\exists X\forall x\ \big (x\in X\leftrightarrow \sigma(x)\big )$, where $\sigma\in{ \Sigma}^0_1$, $\pi\in{ \Pi}^0_1$, and $X$ is not free in $\sigma$ or $\pi$;

\item[${\tt Ind}$:] ${\tt 0}\in X\wedge \forall x\ \big (x\in X\rightarrow x+{\tt 1}\in X \big )\ \to\ \forall x\, (x\in X).$

\end{description}

We mention one further axiom that requires a more elaborate setup. We may represent well-orders in second-order arithmetic as pairs of sets $\Lambda=\langle |\Lambda|,\leq_\Lambda\rangle$, and define
\[{\tt WO}(\Lambda)={\tt linear}(\Lambda)\wedge\forall X \subseteq |\Lambda| \ (\exists x\in X \rightarrow \exists y \in X\forall z\in X  y\leq_\Lambda z),\]
where ${\tt linear}(\Lambda)$ is a formula expressing that $\Lambda$ is a linear order.

Given a set $X$ whose elements we will regard as ordered pairs $\langle\lambda,n\rangle$, let $X_{<_\Lambda \lambda}$ be the set of all $\langle \mu ,n\rangle$ with $\mu <_\Lambda \lambda$. With this, we define the {\em transfinite recursion} scheme by
\[{\tt TR}_\phi(X,\Lambda)= \forall \lambda\in |\Lambda| \ \forall n \ \big (n\in X\leftrightarrow \phi(n,X_{<_\Lambda\lambda}) \big ).\]
Intuitively, ${\tt TR}_\phi(X,\Lambda)$ states that $X$ is made up of ``layers'' indexed by elements of $\Lambda$, and the elements of the $\lambda^{\rm th}$ layer are those natural numbers $n$ satisfying $\phi(n,X_{<_\Lambda\lambda})$, where $X_{<_\Lambda\lambda}$ is the union of all previous layers. If $\Gamma$ is a set of formulas, we denote the {\em $\Gamma$-transfinite recursion} scheme by
\[\Gamma\text{-}{\tt TR}=\Big \{ \forall \Lambda \big ( {\tt WO}(\Lambda)\rightarrow \exists X \ {\tt TR}_\phi(X,\Lambda) \big ) : \phi\in \Gamma \Big\}.\]
Now we are ready to define some important theories:
\begin{center}
\begin{tabular}{ll}
$\eca:$&\Robinson + ${\tt Ind}$+${ \Delta}^0_0$-$\compax$;\\
${\rm RCA}_0^\ast:$&\Robinson + ${\tt Ind}$+${ \Delta}^0_1$-$\compax$;\\
$\rca:$&$\Robinson + {\tt I}{ \Sigma}^0_1$+${ \Delta}^0_1$-$\compax$;\\
$\aca :$&$\Robinson + {\tt Ind}$+${ \Sigma}^0_1$-$\compax$;\\
$\atr :$&$\Robinson + {\tt Ind}+ \Pi^0_\omega\text{-}{\tt TR}$;\\
$\pica:$&$\Robinson + {\tt Ind}$+${ \Pi}^1_1$-$\compax$.\\
\end{tabular}
\end{center}
These are listed from weakest to strongest. The theories $\rca$, $\aca$ $\atr$ and $\pica$, together with the theory of {\em weak K\"onig's lemma,} ${\rm WKL}_0$, are the `Big Five' theories of reverse mathematics, where $\rca$ functions as a `constructive base theory', and the stronger four theories are all equivalent to many well-known theorems in mathematical analysis. For a detailed treatment of these and other subsystems of second-order arithmetic, see \cite{Simpson:2009:SubsystemsOfSecondOrderArithmetic}.

$\eca$ (the theory of {\em elementary comprehension}) is the second-order analogue of elementary arithmetic, and is a bit weaker than the more standard ${\rm RCA}_0^\ast$. Meanwhile, {\em arithmetical comprehension} ($\aca$) is essentially the second-order version of $\pa$, and has the same proof-theoretic ordinal, $\varepsilon_0$. Thus the next milestone in the $\Pi^0_1$ ordinal analysis program is naturally $\atr$, the theory of {\em arithmetical transfinite recursion.} Appropriately, the constructions we will use to interpret the modalities $\langle \lambda \rangle$ for countable $\lambda>\omega$ may be carried out within $\atr$.

\subsection{Iterated $\omega$-rules}\label{SubsecOmegaRule}

If we wish to interpret $\provx {\lambda} T \, \phi$ for transfinite $\lambda$, we need to consider a notion of provability that naturally extends beyond $\omega$. One such notion, which is well-studied in proof theory (see, e.g., \cite{Pohlers:2009:PTBook}), considers infinitary derivations with the {\em $\omega$-rule.} Intuitively, this rule has the form
\[\dfrac{\phi(\bar 0) \ \ \ \ \phi(\bar 1) \ \ \ \ \phi(\bar 2) \ \ \ \  \phi(\bar 3)  \ \ \ \  \phi(\bar 4)  \ \ \ \  \hdots}{\forall x \, \phi(x)}\]
The parameter $\lambda$ in $\provx {\lambda} T \, \phi$ denotes the nesting depth of $\omega$-rules that may be used for proving $\phi$. The notion of $\lambda$-provability is defined as follows:

\begin{definition}\label{DefLambdaProv}
Let $T$ be a theory of second-order arithmetic and $\phi\in\Pi^1_\omega$. For an ordinal $\lambda$, we define $[\lambda]_T\phi$ recursively if either
\begin{enumerate}[label=(\roman*)]

\item $\nc_T\phi$, or

\item there are an ordinal $\mu<\lambda$ and a formula $\psi(x)$ such that
\begin{enumerate}

\item for all $n<\omega$, $[\mu]_T \psi(\bar n)$, and

\item $\nc_T(\forall x\psi(x)\to\phi)$.

\end{enumerate}

\end{enumerate}
\end{definition}

This notion can be formalized by representing $\omega$-proofs as infinite trees, as presented by Arai \cite{Arai1998} and Girard \cite{GirardProofTheory}. Here we will instead use the formalization of Joosten and I \cite{FernandezJoosten:2013:OmegaRuleInterpretationGLP}. We use a set $P$ as an {\em iterated provability class,} whose elements are codes of
pairs $\langle\lambda,\varphi\rangle$, with $\lambda$ a code for an
ordinal and $\varphi$ a code for a formula. The idea is that we want $P$ to be a set of pairs
$\langle\lambda,\varphi\rangle$ satisfying Definition \ref{DefLambdaProv} if we set $\provx{\lambda} T \, \varphi \leftrightarrow
\langle\lambda,\varphi\rangle \in P$. Thus we may write $[\lambda]_P
\varphi$ instead of $\langle \lambda,{ \varphi}\rangle\in P$. 

\begin{definition}
Fix a well-order $\Lambda$ on $\mathbb N$. Say that a set $P$ of natural numbers is an {\em iterated provability class for $\Lambda$} if it satisfies the expression
\begin{equation*}
\provx\lambda P\, \varphi \ \leftrightarrow \ \Big( \Box_T \varphi \vee
\exists \, \psi\, \exists\, \xi{{<_\Lambda}} \lambda \ \big(\forall
n \ \provx\xi P\, \psi({\dot{n}}) \ \wedge \ \Box_T
(\forall x \psi (x) \to \varphi) \big) \Big).
\end{equation*}
Let $\Provfor^\Lambda_{T}(P)$ be a $\Pi^0_\omega$ formula stating that $P$ is an iterated provabiltiy class for $\Lambda$. Then, define
\[[\lambda]^\Lambda_ T \, \phi  \ := \  \forall P \, \big (\Provfor^\Lambda_{T}(P)\rightarrow [\lambda]_P\phi \big ).\]
\end{definition}

Note that $[\lambda]^\Lambda_ T$ is a $\Pi^1_1$ formula. Alternately, one could define $[\lambda]^\Lambda_ T$ as a $\Sigma^1_1$ formula, but the two definitions are equivalent due to the following.

\begin{lemma}\
\begin{enumerate}
\item It is provable in $\aca$ that if $\Lambda$ is a countable well-order and $P,Q$ are both iterated provability classes for $\Lambda$, then $P=Q$.

\item It is provable in $\atr$ that if $\Lambda$ is a countable well-order, then there exists an iterated provability class for $\Lambda$.

\end{enumerate}

\end{lemma}

The first claim is proven by considering two IPC's $P,Q$ and showing by transfinite induction on $\lambda$ that $[\lambda]_P \, \phi \leftrightarrow [\lambda]_Q \, \phi$; this induction is readily available in $\aca$ since the expression $[\lambda]_P\phi$ is arithmetical. For the second, we simply observe that the construction of an IPC is a special case of arithmetical transfinite recursion. See \cite{FernandezJoosten:2013:OmegaRuleInterpretationGLP} for more details.

If we fix a computable well-order $\Lambda$ and a theory $T$ in the language of second-order arithmetic, we can readily define $\cdot^\Lambda_T\colon \lan{\Lambda}\to\Pi^1_\omega$ as in Definition \ref{DefArithInt}, but setting $(\lambda \phi)^\Lambda_T=\langle\bar \lambda\rangle^\Lambda_{T}\phi_T$ We then obtain the following:

\begin{theorem}\label{TheoSOAcomplete}
Let $\Lambda$ be a computable well-order and $T$ be a theory extending $\aca$ such that it is provable in $T$ that $\Lambda$ is well-ordered, and that there is a set $P$ satisfying $\Provfor^\Lambda_{T}(P)$.

Then, for any sequent $\phi\rcto\psi$ of $\lan\Lambda$, $\RC\vdash \phi\rcto\psi$ if and only if $T\vdash \phi^\Lambda_T \to\psi^\Lambda_T$.
\end{theorem}

\proof
This is proven in \cite{FernandezJoosten:2013:OmegaRuleInterpretationGLP} with $\glp_\Lambda$ in place of $\RC_\Lambda$, and this version is obtained by observing that $\glp_\Lambda$ is conservative over $\RC_\Lambda$ by Theorem \ref{TheoGLPCons}.
\endproof

The computability condition in $\Lambda$ is included due to the fact that in the proof of Theorem \ref{TheoSOAcomplete}, we need to be able to prove properties about $\Lambda$ within $T$; for example, we need for
\[\forall x \, \forall y \, \big  (x\leq_\Lambda y\to \nc_T (\dot x\leq_\Lambda \dot y) \big )\]
to hold. However, we can drop this condition if we allow an {\em oracle} for $\Lambda$; or, more generally, for any set of natural numbers. To do this, we add a set-constant $O$ to the language of second-order arithmetic in order to `feed' information about any set of numbers into $T$.

To be precise, given a theory $T$ and $A\subseteq \mathbb N$, define $T|A$ to be the theory whose rules and axioms are those of $T$ together with all instances of $\bar n\in O$ for $n\in X$, and all instances of $\bar n\not\in O$ for $n\not\in X$. Then, for any formula $\phi$, we define
\[[\lambda|X]^\Lambda_T\phi=[\lambda]^\Lambda_{T|X}\phi.\]
Its dual, $\langle \lambda|X\rangle^\Lambda_T\phi$, is defined in the usual way. With this, we obtain an analogue of Theorem \ref{TheoPaDiamond} for $\atr$, proven by Cord\'on-Franco, Joosten, Lara-Mart\'in and myself in \cite{CordonFernandezJoostenLara:2014:PredicativityThroughTransfiniteReflection}:

\begin{theorem}\label{TheoATRDiamond}
$\atr\equiv \eca + \forall \Lambda \, \forall X \, \langle \lambda|X\rangle^\Lambda_T\top.$
\end{theorem}

This result may well be the first step in a consistency proof of $\atr$ in the style of Theorem \ref{TheoPACons}. Moreover, the proof-theoretic strength of $\atr$ is measured by the Feferman-Sch\"utte ordinal, $\Gamma_0$. In the rest of this section, we will see how the worm ordering relates to this ordinal.

\subsection{Ordering transfinite worms}

Let us extend our calculus for computing $o$ to worms that may contain transfinite entries. In Section \ref{SecFiniteW}, we used the operations $b,h$ and $1\downarrow $ to simplify worms and compute their order-types. However, this will not suffice for transfintie worms. For example, if $\fw=\omega 0\omega\top$, we have that $h(\fw)=\omega\top$ while $b(\fw)=\omega\top$, both of which are shorter than $\omega$. However,
\[1\downarrow (\omega\top) =\langle -1+\omega\rangle\top=\omega\top;\]
thus, demoting by $1$ will not get us anywhere. Instead, we could demote by $\omega$, and obtain $\omega\downarrow(\omega\top)=0\top$, which is indeed `simpler'. As we will see, this is the appropriate way to decompose infinite worms:

\begin{lemma}\label{LemmUparrowDecomp}
Given a worm $\mathfrak w\not=\top$, there exist unique $\mu<\Lambda$ and worms $\mathfrak w_1,\mathfrak w_0$ such that either $\fw_1 = \top$ or $0<\min\mathfrak w_1$ and
\[\mathfrak w=\mu\promote (\mathfrak w_1 \mathrel 0 \mathfrak w_0).\]
\end{lemma}

\proof
Take $\mu=\min\fw$, $\fw_1=h(\mu\downarrow \fw)$ and $\fw_0=b(\mu\downarrow \fw)$; evidently these are the only possible values that satisfy the desired equation.
\endproof

With this we may define the norm of a worm $\fw$, which roughly corresponds to the number of operations of $0$-concatenation and $\mu$-promotion needed to construct $\fw$.

\begin{definition}\label{DefCnorm}
For $\fw\sqsubset\ord$ we define $\cnorm{\mathfrak w}$ inductively by
\begin{enumerate}

\item $\cnorm\top=0$;

\item if $\mathfrak w\not=\top$ and $\min\fw=0$, set
\[\cnorm \fw=\cnorm{h(\fw)}+\cnorm{b(\fw)}+ 1;\]

\item otherwise, let $\mu=\min\fw>0$, and set
\[\cnorm \fw=\cnorm{\mu\downarrow\fw}+ 1.\]

\end{enumerate}

\end{definition}

The following is obvious from Definition \ref{DefCnorm} and Lemma \ref{LemmUparrowDecomp}:

\begin{lemma}
For every worm $\mathfrak w$, $\cnorm{\mathfrak w}\in\mathbb N$ is well-defined. Moreover, if $\mathfrak w=\alpha\promote (\mathfrak w_1 \mathrel 0 \mathfrak w_0)$ with $0<\min\mathfrak w_1$, then $\cnorm{\mathfrak w_1},\cnorm{\mathfrak w_0}<\cnorm{\mathfrak w}$.
\end{lemma}

Thus we may try to compute $o(\fw)$ by recursion on $\cnorm\fw$. Assuming that the identity $o(\fw)=ob(\fw)+\omega ^{o(1\downarrow h(\fw))}$ remains valid for transfinite worms, we only have to find a way to compute $o(\mu\uparrow\fw)$ in terms of $o(\fw)$. Fortunately, the map $o(\fw)\mapsto o(\mu\promote\fw)$ is well-defined; let us denote it by $\upsigma^\mu$.

\begin{lemma}
There exists a unique family of functions $\vec\upsigma=\langle\upsigma^\xi\rangle_{\xi\in{\sf Ord}}$ such that $\upsigma^\xi\colon{\sf Ord}\to {\sf Ord}$ and, for every ordinal $\xi$ and every worm $\mathfrak w$, $\upsigma^\xi o(\mathfrak w)= o(\xi\promote \mathfrak w)$.
\end{lemma}

\proof
Given ordinals $\xi,\zeta$, we need to see that there exists a unique ordinal $\vartheta$ such that $\vartheta= o(\xi\promote\fw)$ whenever $\zeta=o(\fw)$.

First observe that, by Corollary \ref{CorSurj}, there is some worm $\fw_\ast$ such that $\zeta = o(\fw_\ast)$. Since by Theorem \ref{TheoWormsWO}, the class of worms is well-ordered, $o(\xi\promote \fw_\ast)$ is well-defined. It remains to check that if $\fw$ is an arbitrary worm such that $o(\fw)=\xi$, then also $o(\xi\promote \fw)=o(\xi\promote \fw_\ast)$. But if $o(\fw)=o(\fw_\ast)$, by Lemma \ref{LemmEquiv} we have that $\fw\equiv\fw_\ast$, and thus by Lemma \ref{LemmRCUparrow}, $\xi\promote\fw\equiv\xi\promote\fw_\ast$. The latter implies that $o(\xi\promote \fw)=o(\xi\promote \fw_\ast)$, as needed.
\endproof

\begin{lemma}\label{LemmSigma}
The family of functions $\vec\upsigma$ has the following properties:
\begin{enumerate}

\item $\upsigma^\alpha$ is strictly increasing for all $\alpha$;\label{LemmSigmaItOne}

\item $\upsigma^0 \xi=\xi$, and\label{LemmSigmaItTwo}

\item $\upsigma^{\alpha+\beta}=\upsigma^\alpha \upsigma ^\beta$.\label{LemmSigmaItThree}

\end{enumerate}
\end{lemma}

\proof
For item \ref{LemmSigmaItOne}, suppose that $\xi<\zeta$. If $\xi=o(\fw)$ and $\zeta=o(\fv)$, then by Lemma \ref{LemmOIff}, $\fw\wle{}\fv$, so that by Lemma \ref{LemmUparrow}, $\alpha\uparrow\fw\wle{}\alpha\uparrow \fv$ and thus $o(\alpha\uparrow\fw)<o(\alpha\uparrow \fv)$; a similar argument shows that if $o(\fw)\wleq{}o(\fv)$, then $o(\alpha\uparrow \fw)\wleq{}o(\alpha\uparrow\fv)$. Item \ref{LemmSigmaItTwo} follows from the fact that $0\uparrow\fw=\fw$ for all $\fw$, so if $\zeta=o(\fw)$ we have that $\upsigma^0 \zeta=o(0\uparrow \fw)=\zeta$. 

Item \ref{LemmSigmaItThree} is immediate from Lemma \ref{LemmUparroAlg}.\ref{LemmUparroAlgItPlus}, since if $o(\fw)=\zeta$ then $o(\beta\promote\fw)=\upsigma^\beta(\zeta)$, which means that
\[o(\alpha\promote(\beta\promote\fw))=\upsigma^\alpha \upsigma^\beta(\zeta).\]
But, on the other hand, $\alpha\promote(\beta\promote\fw)=(\alpha+\beta)\promote\fw,$ and
\[o((\alpha+\beta)\promote\fw)=\upsigma^{\alpha+\beta}\zeta,\]
and we conclude that $\upsigma^{\alpha+\beta}\zeta=\upsigma^\alpha \upsigma^\beta\zeta$.
\endproof

Observe also that if $\zeta<\varepsilon_0$, then by Theorem \ref{TheoFiniteO}, there is $\fw\sqsubset\omega$ such that $\zeta=o(\fw)$, and hence by Theorem \ref{TheoFiniteO}, $\upsigma^1\zeta = o(1\promote \fw)=-1+\omega^{o(\fw)}$ (where we subtract $1$ to account for the case $\fw = \top$). Thus for $\zeta<\varepsilon_0$, $\upsigma^1\zeta =-1+\omega^\zeta$. It is thus natural to conjecture that $\upsigma^1\zeta =-1+\omega^\zeta$ for all $\zeta$. In the next section we will discuss how a family of ordinal functions satisfying these properties can be constructed, and show that they are closely related to the Feferman-Sch\"utte ordinal $\Gamma_0$.

\subsection{Hyperations and the Feferman-Sch\"utte ordinal}\label{ordinalintro}

Beklemishev has shown how provability algebras give rise to a notation system for $\Gamma_0$. Such ordinals are usually presented using Veblen progressions \cite{Veblen:1908}, but alternatively they may be defined through {\em hyperations,} which are more convenient in our present context.
\begin{definition}\label{hyperdef}
Let $f$ be a normal function. Then, we define the {\em hyperation} of $f$ to be the unique family of normal functions $\langle f^\zeta\rangle_{\zeta\in\sf On}$ such that
\begin{enumerate}[label=(\roman*)]
\item\label{DefHyperOne} $f^1=f$
\item\label{DefHyperTwo} $f^{\alpha+\beta}=f^\alpha f^\beta$ for all ordinals $\alpha,\beta$

\item $\langle f^\zeta\rangle_{\zeta\in\sf On}$ is pointwise minimal amongst all families of normal functions satisfying the above clauses\footnote{That is, if $\langle g^\zeta\rangle_{\zeta\in\sf On}$ is a family of functions satisfying conditions \ref{DefHyperOne} and \ref{DefHyperTwo}, then for all ordinals $\xi,\zeta$, $f^\zeta\xi\leq g^\zeta\xi$.}.
\end{enumerate}
\end{definition}

It is not obvious that such a family of functions exists, but a detailed construction is given by Joosten and myself in \cite{FernandezJoosten:2012:Hyperations}. It is also shown there that they may be computed by the following recursion:

\begin{lemma}\label{LemmExpRec} Let $f$ be a normal function such that $f(0)=0$. Then, given ordinals $ \lambda,\mu$,
\label{recexp}
\begin{enumerate}[label=(\roman*)]
\item  $f^0\mu=\mu$;
\item $f^{\lambda+1}\mu=f^\lambda f \mu$;
\item  \label{RecHypIII}if $\mu$ is a limit, $f^\lambda \mu = \displaystyle\lim_{\xi < \mu} f^\lambda\xi$;
\item  \label{RecHypIV}if $\lambda$ is a limit, $f^\lambda(\mu+1) = \displaystyle \lim_{\xi < \lambda}f^\xi(f^\lambda(\mu) + 1)$.
\end{enumerate}
\end{lemma}

Although each function $f^\xi$ is normal, the function $\xi\mapsto f^\xi\mu$ typically is not, even when $\mu=0$, since if $f(0)=0$ then it follows that $f^\xi 0=0$ for all $\xi$. However, when $f(0)>0$ then $\xi\mapsto f^\xi 0$ {\em is} normal, and more generally, we have the following:

\begin{lemma}\label{LemmMuNormal}
Assume that $f\colon\ord\to \ord$ is normal and suppose that $\mu$ is the least ordinal such that $f(\mu)>\mu$ (if it exists).

Then, the function $\xi\mapsto f^\xi\mu$ is normal, and for all $\xi$, $f^\xi \upharpoonright \mu $ is the identity (where $\upharpoonright$ denotes domain restriction).
\end{lemma}

We omit the proof which proceeds by transfinite induction using Lemma \ref{LemmExpRec}. We are particularly interested in hyperating $e(\xi)=-1+\omega^\xi$; the family of functions $\langle e^\xi\rangle_{\xi\in \ord}$ are the {\em hyperexponentials.} Observe that, in view of Lemma \ref{LemmMuNormal}, $e^\xi 0=0$ for all $\xi$ and the function $\xi\mapsto e^\xi 1$ is normal. Aside from the clauses mentioned above, we remark that to entirely determine the value of $e^\lambda\mu$ we need the additional clause
\[e^1(\mu+1)= \lim_{n<\omega} \big ( (1+e^1\mu) \cdot n \big ) ;\]
this follows directly from the definitions of ordinal exponentiation and the function $e$.

Aguilera and I proved the following in \cite{Aguilera:StrongCompleteness}:

\begin{proposition}\label{PropEHNF}
For every ordinal $\xi>0$, there exist unique ordinals $\alpha$, $\beta$ such that $\beta$ is $1$ or additively decomposable and $\xi=e^\alpha \beta$.
\end{proposition}

We call $\alpha$ above the {\em degree of indecomposability} of $\xi$; in particular, if $\xi$ is already additively decomposable, then $\alpha=0$. More generally, $e^\alpha\beta$ is always additively indecomposable if $\alpha,\beta>0$, since
\[e^\alpha\beta=e{e^{-1+\alpha}\beta} =-1+\omega^{e^{-1+\alpha}\beta}=\omega^{e^{-1+\alpha}\beta}.\]
Note that by writing $\beta$ as a sum of indecomposables we may iterate this lemma and thus write any ordinal in terms of $e,+,0$ and $1$. This form is unique if we do not allow sums of the form $\xi+\eta$ where $\xi+\eta=\eta$.

We will not review Veblen progressions here; however, as these are more standard than hyperexponentials, we remark that notations using hyperexponentials or Veblen functions can be easily translated from one to the other using the following proposition. Below, $\upvarphi_\alpha$ denotes the Veblen functions as defined in \cite{Pohlers:2009:PTBook}.
\begin{proposition}\label{PropExptoveb}
Given ordinals $\alpha,\beta$,
\begin{enumerate}
\item $e^\alpha(0)=0$,
\item $e^1(1+\beta)=\upvarphi_0(1+\beta)$,
\item $e^{\omega^{1+\alpha}} (1+\beta)=\upvarphi_{1+\alpha}(\beta)$.
\end{enumerate}
\end{proposition}

The proof can be found in \cite{FernandezJoosten:2012:Hyperations}. We have seen that every ordinal $\xi<\varepsilon_0$ can be written as a sum of the form $\alpha+\omega^\beta$ with $\alpha,\beta<\xi$. In general, it is desirable in any ordinal notation system that, if we have a notation for an additively indecomposable $\xi$, then we also have notations for ordinals $\alpha,\beta<\xi$ such that $\alpha+\beta=\xi$. If instead $\xi$ is additively indecomposable, it is also convenient to have notations for $\alpha,\beta$ such that $\xi=e^\alpha\beta$ (although we cannot always guarantee that $\alpha<\xi$). The following definition captures these properties.

\begin{definition}
Let $\Theta$ be a set of ordinals.
\begin{enumerate}

\item We say that $\Theta$ is {\em additively reductive} if whenever $\xi$ is additively decomposable, we have that $\xi\in\Theta$ if and only if there are $\alpha,\beta\in\xi\cap\Theta$ such that $
\xi=\alpha+\beta$.

\item We say that $\Theta$ is {\em hyperexponentailly reductive} if whenever $\xi>1$ is additively indecomposable, we have that $\xi\in\Theta$ if and only if there are $\alpha,\beta\in\Theta$ such that $\beta<\xi$ and $
\xi=e^\alpha \beta$.

\item We say that $\Theta$ is {\em reductive} if it is additively and hyperexponentially reductive.

\end{enumerate}

\end{definition}

Additively reductive sets of ordinals always contain Cantor decompositions of their elements and are closed under left subtraction by {\em arbitrary} ordinals:

\begin{lemma}\label{LemmSumCantor}
Let $\Theta$ be an additively reductive set of ordinals such that $0\in\Theta$. Then:

\begin{enumerate}

\item If $0\not=\xi\in\Theta$ is arbitrary, there are ordinals $\alpha,\beta$ such that $\alpha,\omega^\beta\in \Theta$ and $\xi=\alpha+\omega^\beta$.

\item If $\beta\in\Theta$ and $\alpha<\beta$  (not necessarily a member of $\Theta$), then $-\alpha+\beta\in \Theta$.\label{LemmSumCantorMinus}

\end{enumerate}
\end{lemma}

\proof
For the first claim, if $\xi$ is additively indecomposable there is nothing to do, since we already have that $\xi=\omega^\beta$ for some $\beta$. Otherwise, using the assumption that $\Theta$ is additively reductive, write $\xi=\gamma+\delta$ with $\gamma,\delta\in \xi\cap\Theta$.

By the induction hypothesis applied to $\delta$, there are $\eta,\beta$ such that $\eta,\omega^\beta\in \Theta$ and $\delta=\eta+\omega^\beta$. Again using the assumption that $\Theta$ is additively reductive, we may set $\alpha=\gamma+\eta\in\Theta$, and see that $\xi=\alpha+\omega^\beta$.

Now we prove the second item by induction on $\xi$. We may assume that $\beta$ is additively indecomposable, since otherwise $-\alpha+\beta\in\{0,\beta\}\subseteq\Theta$. Thus we may write $\beta=\gamma+\delta$ with $\gamma,\delta\in\beta\cap\Theta$. If $\alpha\leq\gamma$, by the induction hypothesis $-\alpha+\gamma\in \Theta$, and thus $-\alpha+\beta=(-\alpha+\gamma)+\delta\in \Theta$. Otherwise, also by the induction hypothesis applied to $\delta<\xi$,
\[-\alpha+\beta=-(-\gamma+\alpha)+\delta\in\Theta.\qedhere\]
\endproof

Meanwhile, hyperexponentially reductive sets of ordinals always contain hyperexponential normal forms for their elements:

\begin{lemma}
If $\Theta$ contains $0$ and is hyperexponentially reductive, then for every $\xi\in\Theta$, there are $\alpha,\beta\in\Theta$ such that $\beta=1$ or is additively decomposable, and $\xi=e^\alpha\beta$.
\end{lemma}

\proof
By induction on $\xi;$ if $\xi$ is additively decomposable or $1$ then $\xi=e^0\xi$, otherwise there are $\alpha',\beta'\in \Theta$ with $\beta<\xi$ such that $\xi=e^{\alpha'}\beta'$. By induction hypothesis there are $\gamma,\beta\in\Theta$ such that $\beta=1$ or is additively decomposable and $\beta'=e^\gamma\beta$. Setting $\alpha=\alpha'+\gamma$, we see that $\xi=e^\alpha\beta$, as desired.
\endproof

The ordinal $\Gamma_0$ can be constructed by closing $\{0,1\}$ under addition and hyperexponentiation, or more succinctly by the function $\alpha,\beta,\gamma\mapsto e^\alpha(\beta+\gamma)$. In fact, $\Gamma_0$ is the least {\em hyperexponentially perfect} set, in the sense of the following definition:

\begin{definition}
Define a function ${\rm HE}\colon \ord^3\to \ord$ by
\[{\rm HE}(\alpha,\beta,\gamma)=e^\alpha(\beta+\gamma).\] 
Given a set of ordinals $\Theta$, say that $\Theta$ is {\em hyperexponentially closed} if $2\cup\iter{{\rm HE}}\Theta\subseteq\Theta$. We say that $\Theta$ is {\em hyperexponentially perfect} if it is reductive and hyperexponentially closed.
\end{definition}

It is easy to see that $\Theta$ is hyperexponentially perfect if and only if it is reductive and $0,1\in \Theta$.
Note also that hyperexponentially closed sets are closed under both addition and hyperexponentiation:

\begin{lemma}\label{LemmEclSumExp}
If $0\in \Theta$ and $\alpha,\beta\in\Theta$, then $\alpha+\beta,e^\alpha\beta\in\iter{{\rm HE}}\Theta$.
\end{lemma}

\proof
If $\Theta$ is hyperexponentially closed then by definition we have that $0\in\Theta$, hence if $\alpha,\beta\in\Theta$, $\alpha+\beta=e^0(\alpha+\beta)\in\ecl\Theta$ and $e^\alpha\beta=e^\alpha(\beta+0)\in \ecl\Theta$.
\endproof

With this, we are ready to define the ordinal $\Gamma_0$:

\begin{theorem}
Let $\Gamma_0=\close {{\rm HE}}{2}$. Then, $\Gamma_0$ is an ordinal and for every $\xi<\Gamma_0$ with $\xi>1$, there are ordinals $\alpha,\beta,\gamma<\xi$ such that $\xi=e^\alpha(\beta+\gamma)$.
\end{theorem}

\proof
The proof closely mimics that of Theorem \ref{TheoEpxilon}. First we will show that if $1<\xi\in\Gamma_0$, then there are $\alpha,\beta,\gamma<\xi$ such that $\xi=e^\alpha(\beta + \gamma)$. By Lemma \ref{LemmPropFClose}.\ref{LemmPropFCloseItFour}, there are $\alpha,\beta,\gamma\in\Gamma_0$ with $\alpha,\beta,\gamma\not=\xi$ and such that $\xi=e^\alpha(\beta+\gamma)$. Since $\xi\not=0$ it follows that $\beta+\gamma\geq 1$, and since the function $e^\alpha$ is normal, $\beta+\gamma\leq \xi$, from which we obtain $\beta,\gamma<\xi$. Similarly, $\alpha\leq \xi$ since $e^\alpha (\beta+\gamma) \geq e^\alpha 1$ and the function $\alpha\mapsto e^\alpha 1$ is normal. Thus we also have $\alpha<\xi$.

Next we show that $\Gamma_0$ is transitive. We proceed by induction on $\zeta$ with a secondary induction on $\xi$ to show that $\xi<\zeta\in\Gamma_0$ implies that $\xi\in\Gamma_0$. We may without loss of generality assume that $\xi,\zeta>1$. Write $\zeta=e^\alpha(\beta + \gamma)$ with $\alpha,\beta,\gamma\in\Gamma_0\cap \zeta$. Then, using Proposition \ref{PropEHNF}, write $\xi=e^\lambda\mu$ with $\mu=1$ or additively decomposable.

Now consider two cases. If $\lambda=0$, we have that $\mu = \xi>1$, hence $\xi$ is additively decomposable and we can write $\xi =\nu+\eta$, with $\nu,\eta<\xi$. By the secondary induction hypothesis, $\nu,\eta\in\Gamma_0$, hence $\xi=e^0(\nu+\eta)\in \Gamma_0$.

Otherwise, $\lambda>0$, and we consider two subcases. If $\alpha\geq \lambda$, by the induction hypothesis applied to $\alpha<\zeta$, $\lambda\in \Gamma_0$. But $\xi>1$ and is additively indecomposable, while $\mu\leq \xi$ is $1$ or additively decomposable, so $\mu<\xi$. By the secondary induction hypothesis, $\mu\in \Gamma_0$, hence $\xi =e^\lambda\mu \in \Gamma_0$.
If instead $\alpha<\lambda$, we observe that $e^\lambda\mu=e^\alpha e^{-\alpha+\lambda}\mu$, and by normality of $e^\alpha$, $e^{-\alpha+\lambda}\mu < \beta + \gamma$. Since $e^{-\alpha+\lambda}\mu$ is additively indecomposable, it follows that $e^{-\alpha+\lambda}\mu \leq \max \{ \beta, \gamma \}$, so that by the induction hypothesis applied to $ \max \{ \beta, \gamma \} < \zeta$, we have that $e^{-\alpha+\lambda}\mu\in \Gamma_0$. Since $\alpha\in\Gamma_0$, $\xi=e^\alpha e^{-\alpha+\lambda}\mu \in \Gamma_0$.
\endproof

Thus $\Gamma_0$ can be characterized as the least hyperexponentially closed ordinal, or alternatively the least hyperexponentially perfect ordinal. Later we will see that it can also be obtained using worms, by closing under $o$. 

\subsection{Order-types of transfinite worms}

As in Section \ref{SubsecFino}, our strategy for giving a calculus for computing $o$ will be to guess a candidate function and prove that it has the required properties. Let us assume that Theorem \ref{TheoFiniteO} remains true for transfinite worms. Moreover, note that the functions $e^\xi$ satisfy all desired properties of our functions $\upsigma^\xi$. Thus we will conjecture that $e^\xi=\upsigma^\xi$ for every ordnal $\xi$, and propose the following candidate:

\begin{definition}\label{DefTrano}
Let $\Lambda$ be an ordinal, $\mathfrak v,\mathfrak w\in \mathbb W^\Lambda$ be worms and $\alpha<\Lambda$ an ordinal.
Then, define

\begin{enumerate}

\item $\trano(\top)=0$,

\item $\trano(\fw)={\trano b(\mathfrak w)}+\omega^{\trano(1\downarrow h(
\fw))}$ if $\fw\not=\top$ and $\min\fw=0$,

\item $\trano(\fw)=e^{\mu} \trano(\mu\downarrow \mathfrak w)$ if $\fw\not=\top$ and $\mu=\min\fw>0$.

\end{enumerate}
\end{definition}

The next few lemmas establish that $\trano$ behaves as it should.

\begin{lemma}\label{LemmTranoNonZero}
If $\fw\not=\top$ is any worm, then $\trano(\fw)\not=0$.
\end{lemma}

\proof
If $\min\fw=0$, this is obvious since $\omega^\xi>0$ independently of $\xi$. Otherwise, $\trano(\fw)=e^{\mu} \trano(\mu\downarrow \mathfrak w)$ with $\mu=\min\fw>0$. But $\min{\mu\downarrow\fw}=0$, so by the previous case $\trano(\mu\downarrow\fw)\not=0$ and hence $\trano(\fw)=e^{\mu} \trano(\mu\downarrow \mathfrak w)\not=0$.
\endproof

\begin{lemma}\label{LemmTransCantor}
For any worm $\mathfrak w\not=\top$, $\trano (\mathfrak w)=\trano b (\fw) +\omega^{\trano(1\downarrow h (\fw) )}$.
\end{lemma}

\proof
If $\min\fw=0$, there is nothing to prove. Otherwise, $\min\fw>0$, so we can write $\min\fw = 1+\eta$ for some $\eta$. Moreover, $h(\fw)=\fw$ and $b(\fw)=\top$, so $\trano{h(\fw)}=\trano(\fw)\not=0$ and $\trano b(\fw)=0$. Meanwhile, $(1+\eta)\downarrow\fw\not=\top$, so $\trano \big ( (1+\eta)\downarrow\fw \big ) \not=0$ and thus ${e^\eta  \trano((1+\eta)\downarrow\fw) } > 0 $, from which it follows that
\begin{equation}\label{EqTranoOmega}
-1+\omega^{e^\eta (\trano((1+\eta)\downarrow\fw))}=\omega^{e^\eta (\trano((1+\eta)\downarrow\fw))}.
\end{equation}
Finally, observe that
\begin{equation}\label{EqTrano}
\trano(1\downarrow h(\fw)) = e^\eta \trano(\eta\downarrow(1\downarrow h(\fw)))=e^\eta \trano((1+\eta) \downarrow h(\fw))).\end{equation}
Putting all of this together,
\begin{align*}
\trano(\fw)&=e^{1+\eta}\trano((1+\eta)\downarrow \fw)&\text{by definition}\\
&=0+e{e^\eta (\trano((1+\eta)\downarrow\fw))}&\text{since $e^{1+\eta}=ee^{\eta}$}\\
&=\trano(b(\fw))+e{e^\eta (\trano((1+\eta)\downarrow\fw))}&\text{since $\trano(b(\fw))=0$}\\
&=\trano(b(\fw))+(-1+\omega^{e^\eta (\trano((1+\eta)\downarrow\fw))})&\text{by definition of $e$}\\
&=\trano(b(\fw))+\omega^{e^\eta (\trano((1+\eta)\downarrow\fw))}&\text{by \eqref{EqTranoOmega}}\\
&=\trano(b(\fw))+\omega^{\trano(1\downarrow h(\fw))} &\text{by \eqref{EqTrano}},
\end{align*}
as claimed.
\endproof

\begin{lemma}\label{LemmTranEUparrow}
For any worm $\mathfrak w$ and ordinal $\lambda$, $\trano(\lambda\promote\fw)=e^\lambda\trano (\fw)$.
\end{lemma}

\proof
If $\fw=\top$, then
\[\trano(\lambda \uparrow\top)=\trano(\top)=0=e^\lambda 0= e^\lambda\trano(\top).\]
Otherwise, $\fw\not=\top$. If $\lambda=0$ the lemma follows from the fact that $0\promote\fw=\fw$ and $e^0$ is the identity, and if $\min\fw=0$ then $\min(\lambda\promote\fw)=\lambda$ and
\[\trano(\lambda\uparrow\mu)=e^\lambda\trano(\lambda\downarrow(\lambda\uparrow\fw))=e^\lambda\trano(\fw).\]
If not, let $\mu=\min \fw>0$, so that $\trano(\fw)=e^\mu(\mu\downarrow\fw)$. Observe that $\min(\lambda\promote\fw)=\lambda+\mu$. Hence,
\begin{align*}
\trano(\lambda\promote\fw)&=e^{\lambda+\mu}((\lambda+\mu)\downarrow(\lambda\promote\fw))\\
&=e^{\lambda+\mu}(\mu\downarrow(\lambda\downarrow(\lambda\promote\fw)))\\
&=e^\lambda e^\mu(\mu\downarrow\fw)\\
&=e^\lambda \trano(\fw),
\end{align*}
as claimed.
\endproof

With this we can prove that $\trano$ is strictly increasing and initial.

\begin{lemma}\label{LemmTranoInc}
The map $\trano\colon \Worms\to \ord$ is strictly increasing.
\end{lemma}

\proof
We proceed by induction on $\cnorm{\mathfrak w}+\cnorm{\mathfrak v}$ to show that $\mathfrak w\wle{}\mathfrak v$ if and only if $\trano(\mathfrak w) < \trano(\mathfrak v)$. If $\fw=\top$ the claim is immediate from Lemma \ref{LemmTranoNonZero}, so we assume otherwise. Note that in this case $\fw\wge{}\top$ and $\trano(\fw)>\trano(\top)$, so we may also assume that $\fv\not=\top$.

Thus we consider $\mathfrak w,\mathfrak v\not=\top$, and define $\mu=\min ( \fw\fv ) $. If $\mu=0$, we observe that either $\cnorm{h(\fw)}<\cnorm{ \fw }$ or $\cnorm{h(\fv)}<\cnorm{ \fv }$, and we can proceed exactly as in the proof of Lemma \ref{LemmFinoMon}. Thus we consider only the case for $\mu>0$.

%If $\mu=0$, then by Theorem \ref{TheoWormOrder} we have that either

%\begin{enumerate}

%\item $\mathfrak w\wleq {} b  (\mathfrak v)$, or

%\item $b (\mathfrak w)\wle {} \mathfrak v$ and $h (\mathfrak w)\wle {} h (\mathfrak v)$.

%\end{enumerate}

%In the first case, we may use the induction hypothesis since $ \| b (\mathfrak v) \| < \| \mathfrak v\| $ to conclude that \[\trano(\mathfrak w)\leq \trano(b (\mathfrak v))<\trano b\fw+\omega^{\trano(1\downarrow h\fw)}=\trano(\fv),\]
%where the last equality follows from Lemma \ref{LemmTransCantor}.

%In the second, we use the fact that $ \| b (\mathfrak w) \| < \| \mathfrak w\| $ together with the fact that $\| h(\mathfrak w) \|+\| h (\mathfrak v) \| < \| \mathfrak w \| +\| \mathfrak v\| $ (since $\mu$ appears in at least one of the two worms) to conclude that $\trano(b (\mathfrak w) ) < \trano( \mathfrak w ) $ and $\trano( h(\mathfrak w) ) < \trano( h (\mathfrak v))$. Hence,
%\begin{align*}
%\trano\mathfrak w&=\trano (b\mathfrak w)+\omega^{1\downarrow\trano(h(\mathfrak w))},&\text{by Lemma \ref{LemmTransCantor}}\\
%&<\trano (b\mathfrak v)+\omega^{1\downarrow \trano(h(\mathfrak v))}&\text{by Lemma \ref{LemmCantorOrder}.\ref{LemmCantorOrderItA}}
%\end{align*}
%But by Lemma \ref{??}, $1+\trano(h(\mathfrak u))=\omega^{1\downarrow\trano(h(\mathfrak u))}$ for $\mathfrak u=\mathfrak w,\mathfrak v$, so that $\omega^{1\downarrow\trano(h(\mathfrak u))}\leq \omega^{1\downarrow\trano(h(\mathfrak u))}$, and thus by Lemma \ref{??}, $\trano(\mathfrak w)\leq \trano(\mathfrak v)$.

Note that in this case we have that
\[\cnorm{\mu\downarrow\mathfrak w}+\cnorm{\mu\downarrow \mathfrak v}<\cnorm{ \mathfrak w}+\cnorm{ \mathfrak v},\]
so we may apply the induction hypothesis to $\mu\downarrow\mathfrak w$ and $\mu\downarrow \mathfrak v$. Hence we obtain:
\begin{align*}
\fw\wle {}\fv &\Leftrightarrow (\mu\downarrow\fw)\wle{}(\mu\downarrow\fv)&\text{by Lemma \ref{LemmUparrow}}\\
&\Leftrightarrow \trano(\mu\downarrow\fw)<\trano(\mu\downarrow\fv) & \text{by induction hypothesis}\\
&\Leftrightarrow e^\mu\trano(\mu\downarrow\fw) < e^\mu\trano(\mu\downarrow\fv)&\text{by normality of $e^\mu$}\\
&\Leftrightarrow \trano \fw < \trano \fv & \text{by Lemma \ref{LemmTranEUparrow},}
\end{align*}
as needed.
\endproof

\begin{lemma}\label{LemmTranoSur}
The map $\trano\colon\Worms\to\ord$ is surjective.
\end{lemma}

\proof
Proceed by induction on $\xi\in \ord$ to show that there is $\fw$ with $\trano(\fw)=\xi$. For the base case, $\xi=0=\trano(\top)$. Otherwise, by Proposition \ref{PropEHNF}, $\xi$ can be written in the form $e^{\alpha}\beta$ with $\beta$ additively decomposable or $1$. Write $\beta=\gamma+\omega^\delta$, so that $\gamma,\delta<\beta\leq \xi$. By the induction hypothesis, there are worms $\fu,\fv$ such that $\trano (\fu) =\gamma$ and $\trano (\fv)=\delta$. Then, $\xi=\trano(\alpha\promote((1\uparrow\fv)\mathrel 0 \fu))$, as needed.
\endproof

\begin{lemma}\label{LemmTranoIsO}
For every worm $\fw$, $\trano(\fw)=o(\fw)$.
\end{lemma}

\proof
Immediate from Lemmas \ref{LemmTranoInc} and \ref{LemmTranoSur} using Lemma \ref{LemmOUnique}.
\endproof

Before giving the definitive version of our calculus, let us show that the clasue for $\fw\mathrel 0\fv$ can be simplified somewhat.

\begin{lemma}\label{LemmSumOne}
Given arbitrary worms $\fw$, $\fv$, $o(\fw \mathrel 0 \fv)=o(\fv)+1+o(\fw)$.
\end{lemma}

\proof
Observe that by Lemma \ref{LemmTranEUparrow} together with Lemma \ref{LemmTranoIsO}, we have that for any worm $\fu$, $o(1\promote \fu)=e o(\fu)=-1+\omega^{o(\fu)}$, so that
\begin{equation}\label{EqPlusOne}
\omega^{o(\fu)}=1+o(1\promote\fu).
\end{equation}
With this in mind, proceed by induction on $\lgt \fv + \lgt\fw$ to prove the lemma. First consider the case where $0< \min \fv$. In this case, $h(\fv\mathrel 0\fw)=\fv$, so that
\[o(\fv \mathrel 0\mathfrak w)=  o(\fw)+\omega^{ o(1\downarrow \fv)}=o(\fw)+1+o(\fv),\]
where the first equality is by Defintion \ref{DefTrano} and the second follows from \eqref{EqPlusOne}.

If $\fv$ does contain a zero, we have that $\fv=h(\fv)\mathrel 0 b(\fv)$, so that
\[\fv\mathrel 0\fw= h(\fv)\mathrel 0 b(\fv) \mathrel 0\fw.\]
This means that $h(\fv\mathrel 0\fw)=h(\fv)$ and $b(\fv\mathrel 0\fw)=b(\fv) \mathrel 0\fw$. Applying the induction hypothesis to $b(\fv) \mathrel 0\fw$, we obtain
\[ob(\fv\mathrel 0\fw) = o(\fw)+1+ob(\fv),\]
and thus
\begin{align*}
o(\fv \mathrel 0\mathfrak w)&=o b(\fw\mathrel 0\fv )+\omega^{o(1\downarrow h(\fv))}\\
&\stackrel{\text{\sc ih}}=o(\fw)+1+o b(\fv) +\omega^{o(1\downarrow h(\fv))}=o(\fw)+1+o(\fv),
\end{align*}
as needed.
\endproof

Let us put our results together to give our definitive calculus for $o$.

\begin{theorem}\label{TheoTranOrder}
Let $\mathfrak v,\mathfrak w $ be worms and $\alpha $ be an ordinal. Then,
\begin{enumerate}

\item $o(\top)=0$,

\item $o(\mathfrak v\mathrel 0\mathfrak w) ={o(\mathfrak w)}+1+o(\mathfrak v),$ and

\item $o(\alpha\uparrow\mathfrak w) =e^\alpha {o(\mathfrak w)}.$\label{TheoTranOrderItThree}

\end{enumerate}
\end{theorem}

\proof
The first item is immediate from Definition \ref{DefTrano}, the second from Lemma \ref{LemmSumOne}, and the third from Lemma \ref{LemmTranEUparrow}, respectively, using the fact that $o=\trano$ by Lemma \ref{LemmTranoIsO}.
\endproof

Note that Theorem \ref{TheoTranOrder} can be applied to any worm $\fw$, and hence it gives a complete calculus for computing $o$. Next, let us see how this gives rise to a notation system for $\Gamma_0$.

\subsection{Beklemishev's predicative worms}

Now we review results from \cite{Beklemishev:2005:VeblenInGLP} showing that $\Gamma_0$ is the least set definable by iteratively taking order-types of worms. Let us begin by discussing the properties of sets of worms obtained from additively reductive sets of ordinals. Recall that $\fw\sqsubset \Theta$ means that every ordinal appearing in $\fw$ belongs to $\Theta$.

\begin{lemma}\label{LemmSplitsArrow}
Let $\Theta$ be an additively reductive set of ordinals such that $0\in\Theta$, and let $\fw\sqsubset\Theta$. Then,
\begin{enumerate}

\item If $\mu\in\Theta$, $\mu\uparrow\fw\sqsubset\Theta$, and

\item if $\mu\leq\fw$ is arbitrary, then $\mu\downarrow\fw\sqsubset\Theta$.

\end{enumerate} 
\end{lemma}

\proof
Suppose that $\fw=\lambda_1\hdots\lambda_n\top\sqsubset\Theta$. If $\mu\in\Theta$, using the fact that $\Theta$ is closed under addition, for each $i\in[1,n]$ we have that $\mu+\lambda_i\in\Theta$. Thus $\mu\uparrow\fw\sqsubset\Theta$.

Similarly, by Lemma \ref{LemmSumCantor}.\ref{LemmSumCantorMinus}, if $\mu$ is arbitrary then $-\mu+\lambda_i\in \Theta$ for each $i\in[1,n]$, so $\mu\downarrow\fw\sqsubset\Theta$.
\endproof

Now, let us make the notion of ``closing under $o$'' precise.

\begin{definition}
Observe that $o$ may be regarded as a function $o\colon \ord^{<\omega}\to\ord$ by setting
\[o(\mu_1,\hdots,\mu_n)=o(\mu_1 \hdots \mu_n\top).\]
Then, given a set of ordinals $\Theta$, if $\iter o\Theta\subseteq\Theta$ we say that $\Theta$ is {\em worm-closed,} and if $\Theta=\iter o\Theta$ we say that $\Theta$ is {\em worm-perfect.}
\end{definition}

Even when $\Theta$ is not worm-perfect, sets of the form $\iter o\Theta$ are rather well-behaved:

\begin{lemma}\label{LemmSum}
If $\Theta$ is any set of ordinals, then $ 0\in \iter o\Theta$. If moreover $0\in \Theta$, then also $1\in \iter o\Theta$, and $\iter o\Theta$ is additively reductive.
\end{lemma}

\proof
Observe that $0=o(\top)$, and $\top\sqsubset\Theta$ since $\top$ contains no ordinals, so $0\in \ocl\Theta$. Similarly, $1=o(0\top)$, and $0\top\sqsubset\Theta$ if $0\in\Theta$.

Let us see that $\ocl\Theta$ is additively reductive when $0\in \Theta$. First assume that $\alpha,\beta\in \ocl\Theta$. Then, there are worms $\fu,\fv\sqsubset\Theta$ such that $\alpha=o(\fu)$ and $\beta=o(\fv)$. If $\beta\geq\omega$, then
\[o(\fv\mathrel 0\fu)=o(\fu)+1+o(\fv)=\alpha+1+\beta=\alpha+\beta,\]
otherwise
\[o( \langle 0\rangle ^\beta \fu)=o(\fu)+\beta=\alpha+\beta,\]
where we define $\langle\lambda\rangle^n=\underbrace{\langle \lambda\rangle\hdots\langle\lambda \rangle}_{n\text{ times}}.$ Both $\fv\mathrel 0\fu, \langle 0\rangle ^\beta\fu\sqsubset\Theta$, so $\alpha+\beta\in \ocl\Theta$.

Conversely, if $\xi\in \ocl\Theta$ is additively decomposable, write $\xi=o(\fw)$. Then, $\xi=ob(\fw)+1+oh(\fw)$, and since $1+oh(\fw)$ is additively indecomposable, we have that $\xi\not=1+oh(\fw)$ and hence $ob(\fw),1+oh(\fw) <\xi$. Clearly $ob(\fw)\in \ocl\Theta$, while $1+oh(\fw)$ is either $1$ or $oh(\fw)$, both of which belong to $\ocl\Theta$.
\endproof

\begin{lemma}\label{LemmClosedSum}
Let $\Theta$ be any set of ordinals. Then, $\Theta$ is worm-perfect if and only if it is hyperexponentially perfect.
\end{lemma}

\proof Assume first that $\Theta$ is worm-perfect. By Lemma \ref{LemmSum}, $0\in\Theta$, thus also $1\in \Theta$ and $\Theta$ is additively reductive. It remains to prove that $\ecl\Theta \subseteq \Theta$ and that $\Theta$ is hyperexponentially reductive.

To show that $\ecl\Theta\subseteq\Theta$, it suffices to check that $e^\alpha\beta\in\Theta$ whenever $\alpha,\beta\in\Theta$, given that we already know that $\Theta$ is closed under addition. If $\alpha,\beta\in \Theta$, since $\Theta$ is worm-perfect, there is $\fw \sqsubset\Theta$ such that $o(\fw)=\beta$. By Lemma \ref{LemmSplitsArrow}, $\alpha\uparrow\fw\sqsubset\Theta$, and by Theorem \ref{TheoTranOrder}, $e^\alpha\beta=o(\alpha\uparrow\fw)\in \Theta$.

Next we show that if $1<\xi\in\Theta$, there are $\alpha,\beta\in \Theta$ such that $\xi=e^\alpha\beta$ and $\beta<\xi$. Since $\Theta$ is worm-perfect, $\xi =o(\fw)$ for some $\fw\sqsubset \Theta$. We proceed by induction on $\cnorm{\fw}$ to find suitable $\alpha,\beta\in\Theta$. We may assume that $\fw\not=\top$ since $\xi>0$, and we set $\mu=\min\fw$. If $\mu=0$, then $h(\fw),b(\fw)\sqsubset\Theta$, and since $\Theta$ is worm-perfect, $ob(\fw),oh(\fw)\in\Theta$. Now, if $oh(\fw)=\xi$, by induction on $\cnorm{h(\fw)}$ we see that there exist suitable $\alpha,\beta\in\Theta$. If instead $oh(\fw)< \xi$, this means that $\xi = ob(\fw) + 1 + oh(\fw)$ is additively decomposable, contrary to our assumption.

Now consider $\mu>0$. By Lemma \ref{LemmSplitsArrow}, $\mu\downarrow\fw\sqsubset\Theta$. Hence by induction on $\cnorm{\mu\downarrow\fw}<\cnorm\fw$, we have that $o(\mu\downarrow\fw)=e^\eta\beta$ for some $\eta,\beta\in\Theta$ with $\beta< o(\mu\downarrow\fw)$. It follows that
\[o(\fw)=e^\mu o(\mu\downarrow \fw)=e^\mu e^\eta \beta =e^{\mu+\eta}\beta,\]
and since $\Theta$ is closed under addition, we may set $\alpha=\mu+\eta\in\Theta$.

For the other direction, assume that $\Theta$ is hyperexponentially perfect. To show that $\ocl\Theta\subseteq\Theta$, we will prove by induction on $\cnorm\fw$ that if $\fw\sqsubset \Theta$, then $o(\fw)\in \Theta$. For the base case, if $\fw=\top,$ then $o(\fw)=0 \in \Theta.$ Otherwise, let $\mu=\min\fw$.

If $\mu=0$, then by induction hypothesis $oh(\fw),ob(\fw)\in \Theta$. Since also $1\in\Theta$, then $o(\fw)=ob(\fw)+1+oh(\fw)\in\Theta$. Otherwise, $\cnorm{\mu\downarrow\fw}<\cnorm{\fw}$, and as before, $\mu\downarrow\fw\sqsubset\Theta$. It follows by the induction hypothesis that $o(\mu\downarrow\fw)\in\Theta$. Moreover, since $\mu$ appears in $\fw$ we must have that $\mu\in\Theta$, thus $o(\fw)=e^\mu o(\mu\downarrow\fw)\in \Theta$, using the fact that $\Theta$ is hyperexponentially closed.

Finally, we show that $\Theta\subseteq\ocl\Theta$. We prove by induction on $\xi$ that if $\xi\in\Theta$, then $\xi=o(\fw)$ for some $\fw\sqsubset\Theta$. If $\xi=0$ we may take $\fw=\top$. If not, using the fact that $\Theta$ is hyperexponentially perfect, write $\xi=e^\alpha\beta$ with $\alpha,\beta\in \Theta$ and $\beta=1$ or additively decomposable. If $\beta=1$, then $\xi=e^\alpha 1=o(\alpha \top)$. Otherwise, since $\Theta$ is additively reductive, we may write $\beta=\gamma+\delta'$ with $\gamma,\delta'\in\beta\cap\Theta$. Using Lemma \ref{LemmSumCantor} we see that $\delta=-1+\delta'\in\Theta$. By the induction hypothesis, there are worms $\fu,\fv\sqsubset \Theta$ such that $\gamma=o(\fu)$, $\delta=o(\fv)$, and thus
\[\beta=\gamma+\delta'=\gamma+1+\delta = o(\fu)+1+o(\fv)=o(\fv\mathrel 0\fu).\]
But $\fv\mathrel 0\fu\sqsubset\Theta$, and thus by Lemma \ref{LemmSplitsArrow}, $\alpha\uparrow (\fv\mathrel 0\fu)\sqsubset\Theta$, and $o(\alpha\uparrow (\fv\mathrel 0\fu))=e^\alpha\beta$, as needed.
\endproof

With this, we obtain our worm-based characterization of $\Gamma_0$:

\begin{theorem}\label{TheoPred}
$\Gamma_0$ is the least worm-perfect set of ordinals.
\end{theorem}

\proof
$\Gamma_0$ is the least hyperexponentially perfect set, and since it is transitive and closed under addition, it is additively reductive. Hence $\Gamma_0$ is also worm-perfect, and since any worm-perfect set is hyperexponentially perfect, there can be no smaller worm-perfect set.
\endproof

\subsection{Autonomous worms and predicative ordinal notations}\label{SubsecAutWorms}

The map $o\colon\Worms\to{\sf Ord}$ suggests that worms could themselves be used as modalities. This gives rise to Beklemishev's {\em autonomous worms} \cite{Beklemishev:2005:VeblenInGLP}:

\begin{definition}
We define the set of {\em autonomous worms} $\WAut$ to be the least set such that $\top\in\WAut$ and, if ${\aw},{\av}\in \WAut$, then $\text{\tt (}{\aw}\text{\tt )}{\av}\in\WAut$.
\end{definition}

The idea is to interpret autonomous worms as regular worms using $o$:

\begin{definition}
We define a map $\Bmap\cdot\colon \WAut \to {\Worms}$ given recursively by

\begin{multicols}2
\begin{enumerate}

\item $\Bmap\top=\top$

\item $\Bmap{\big (\text{\tt (}{\aw}\text{\tt )}{\av}\big )}=\langle o(\Bmap{\aw})\rangle\Bmap{{\av}}$.

\end{enumerate}
\end{multicols}
\noindent We then define ${\sf o}\colon \WAut \to \ord$ by setting ${\sf o}({\aw})=o(\Bmap{{\aw}})$.
\end{definition}

As Beklemishev has noted, autonomous worms give notations for any ordinal below $\Gamma_0$.

\begin{theorem}\label{TheoAutWorm}
If $\gamma$ is any ordinal, then $\gamma<\Gamma_0$ if and only if there is ${\aw}\in\WAut$ such that $\gamma={\sf o}({\aw})$.
\end{theorem}

\proof
To see that $\Gamma_0\subseteq {\sf o}[\WAut]$, it suffices in view of Theorem \ref{TheoPred} to observe that ${\sf o}[\WAut]$ is worm-perfect by construction.

To see that ${\sf o}[\WAut]\subseteq\Gamma_0$, one proves by induction on the number of parentheses in $\aw$ that if $\Theta$ contains $0$ and is worm-closed, then ${\sf o}({\aw})\in\Theta$. In particular, ${\sf o}({\aw})\in\Gamma_0$.
\endproof

\begin{figure}
\[\begin{array}{rlrlrl}
1&
\text{\tt ()}&
\omega&
\text{\tt (())}&
\varepsilon_0&
\text{\tt ((()))}\\
\varepsilon_1&
\text{\tt ((()))((()))}&
\varepsilon_\omega+\varepsilon_0
&\text{\tt ((()))()(()(()))}
&e^{e^{e^{e^1 1}1}1}1
&\text{\tt ((((()))))}
\\
\end{array}\]
\caption{Some ordinals represented as autonomous worms. We use the identity $\varepsilon_\xi = e^\omega(1+\xi)$, which is a special case of Proposition \ref{PropExptoveb}.}
\end{figure}

\section{Impredicative worms}\label{SecImpWrm}

Now we turn to a possible solution to Mints' and Pakhomov's problem of representing the Bachmann-Howard ordinal using worms. This ordinal is related to {\em inductive definitions,} that is, least fixed points of monotone operators $F\colon 2^\mathbb N\to 2^\mathbb N$. Let us begin by reviewing these operators and their fixed points.

\subsection{Inductive definitions}\label{SecIndDef}

Let $F\colon 2^\mathbb N\to 2^\mathbb N$. We say that $F$ is {\em monotone} if $F(X)\subseteq F(Y)$ whenever $X\subseteq Y$. For example, if $f\colon \mathbb N^{<\omega}\to \mathbb N$, we obtain a monotone operator by setting $F(X)=f[X]$; as we have seen in Lemma \ref{LemmPropFClose}, we can reach a fixed point for such an $F$ by iterating it $\omega$-many times and taking the union of these iterations. More generally, any monotone operator has a least fixed point:

\begin{definition}
Let $F\colon 2^\mathbb N\to 2^\mathbb N$ be monotone. We define $\upmu F$ to be the unique set such that:
\begin{enumerate}

\item $\upmu F = F(\upmu F)$, and

\item If $X\subseteq \mathbb N$ is such that $F(X)\subseteq X$, then $\upmu F\subseteq X$.

\end{enumerate}
\end{definition}

The Knaster-Tarski theorem states that the set $\upmu F$ is always well-defined \cite{GranasFixedPoint}; it can always be reached ``from below'' by iterating $F$, beginning from the empty set. However, in general, we may need to iterate $F$ far beyond $\omega$.

\begin{definition}
Let $F\colon 2^\mathbb N\to 2^\mathbb N$. For an ordinal $\xi,$ we define an operator $F^\xi\colon 2^\mathbb N\to 2^\mathbb N$ inductively by
\begin{enumerate}

\item $F^0(X)=X$,

\item $F^{\xi+1}(X)=F(F^{\xi}(X))$,

\item $F^{\lambda}(X)=\bigcup_{\xi<\lambda}F^{\xi}(X)$ for $\lambda$ a limit ordinal.

\end{enumerate}
\end{definition}

These iterations eventually become constant, but the ordinal at which they stabilize can be rather large; in principle, our only guarantee is that it is countable, since at each stage before reaching a fixed point we must add at least one natural number. Below, recall that $\omega_1$ denotes the first uncountable cardinal.

\begin{lemma}\label{LemmFromBelow}
If $F\colon 2^\mathbb N\to 2^\mathbb N$ is monotone, then there is $\lambda<\omega_1$ such that $F^{\lambda}(\varnothing)=\upmu F$.
\end{lemma}

We omit the proof, which follows from cardinality considerations. Alternately, it is possible to construct least fixed points `from above', by taking the intersection of all $F$-closed sets. 

\begin{lemma}\label{LemmExistsFix}
If $F\colon 2^\mathbb N\to 2^\mathbb N$ is monotone, then
\[\upmu F =\bigcap \{Y \subseteq \mathbb N:F(Y)\subseteq Y\}.\]
\end{lemma}

%\begin{proof}
%Let $M=\bigcap \{Y:F(Y)\subseteq Y\}$. It is obvious that whenever $F(X)\subseteq X$, it follows that $M\subseteq X$. To show that $M$ is indeed a fixed point of $F$, first let us see that $F(M)\subseteq M$. 

%Let $n\in F(M)$. To show that $n\in M$, we must check that, for an arbitrary $X$ such that $F(X)\subseteq X$, we have that $n\in X$. But for such an $X$, since $M\subseteq X$ we have by monotonicity that $n\in F(X)$, and therefore $n\in X$. Since $X$ was arbitrary, we conclude that $n \in M$.

%With this we have that $F(M)\subseteq M$; it remains to check that $M\subseteq F(M)$. Let $n\in M$ and consider the set $M'=M\setminus \{n\}$. Since $M'\subseteq M$ it follows that $F(M')\subseteq F(M)$; but by the definition of $M$, we cannot have that $F(M')\subseteq M'$, so
%\[\varnothing \not = F(M')\setminus M' \subseteq F(M)\setminus M'\subseteq M\setminus M'=\{n\},\]
%and $n\in F(M')\subseteq F(M)$, as needed.
%\end{proof}

Monotone operators and their fixed points can be formalized in second-order arithmetic, provided they are definable. Any formula $\phi(n,X)\in \Pi^1_\omega$ (with no other free variables) can be regarded as an operator on $2^\mathbb N$ given by $X\mapsto \{n \in \mathbb N:\phi(n,X)\}$. Say that a formula $\phi$ is in {\em negation normal form} if it contains no instances of $\to$, and $\neg$ occurs only on atomic formulas. It is well-known that every formula is equivalent to one in negation normal form, obtained by applying De Morgan's rules iteratively.

\begin{definition}
Let $\phi$ be a formula in negation normal form and $X$ a set-variable. We say {\em $\phi$ is positive on $X$} if $\phi$ contains no occurrences of $t\not \in X$.
\end{definition}

Positive formulas give rise to monotone operators, due to the following:

\begin{lemma}\label{LemmMonotone}
Given a formula $\phi(n,X)$ that is positive on $X$, it is provable in $\eca$ that
\[\forall X \ \forall Y \ \Big( X\subseteq Y\rightarrow \forall n \ \big ( \phi( n, X)\rightarrow \phi( n, Y) \big ) \Big).\]
\end{lemma}

Thus if we define $F_\phi\colon 2^\mathbb N \to 2^\mathbb N$ by $F_\phi(X)=\{n \in \mathbb N : \phi(n,X)\}$, $F_\phi$ will be monotone on $X$ whenever $\phi$ is positive on $X$. Moreover, if $\phi$ is arithmetical, Lemma \ref{LemmExistsFix} may readily be formalized in $\pica$, by defining
\[M= \Big \{n\in\mathbb N : \forall X \Big ( \forall m \big (\phi(m,X)\rightarrow m\in X \big )\rightarrow n\in X\Big ) \Big \}.\]

Thus we arrive at the following:

\begin{lemma}\label{LemmFixPica}
Let $\phi(n,X)$ be arithmetical and positive on $X$. Then, it is provable in $\pica$ that there is a least set $M$ such that, for all $n$,
\[n\in M \leftrightarrow \phi(n,M).\]
We will denote this set $M$ by $\upmu X.\phi$.
\end{lemma}

With these tools in mind, we are now ready to formalize $\omega$-logic in second-order arithmetic.
%%%%%%%%%%%%%%%%%%%%%%%%%%%%%%%%%%%%%%%%%%%%%%%%%%%%%%%%%%

\subsection{Formalizing full $\omega$-logic}

We have discussed before how the $\omega$-rule can be iterated along a well-order. However, we may also consider full $\omega$-logic based on a theory  $T$; that is, the set of formulas that can be derived using the $\omega$-rule and reasoning in $T$, regardless of the nesting depth of these $\omega$-rules. Let us write $\InfPro T\phi$ if $\phi$ is derivable in this fashion. To be precise, we want $\InfPro T \phi$ to hold whenever:
\begin{enumerate}[label=(\roman*)]

\item $\nc_T\phi$,

\item $\phi=\forall x \psi(x)$ and for all $n$, $\InfPro T\psi(\bar n)$, or

\item there is $\psi$ such that $\InfPro T\psi$ and $\InfPro T ( \psi \to \phi)$.
\end{enumerate}
In words, $\InfPro T$ is closed under $T$ and the $\omega$-rule. This notion may be formalized using $ \omega$-trees to represent infinite derivations, as in \cite{Arai1998,GirardProofTheory}. We follow a different approach, using a fixed-point construction as in \cite{FernandezImpredicative}.

\begin{definition}
Fix a theory $T$, possibly with oracles. Let $\spc T{Q}$ be a $\Pi^1_1$ formula naturally expressing that $Q$ is the least set such that $\phi\in Q$ whenever
\begin{enumerate*}[label=(\roman*)]

\item  $\nc_T \phi$ holds,

\item $\phi=\forall v\, \psi(v)$ and for all $n$, $\psi(\bar n)\in Q$, or

\item there exists $\psi\in Q$ such that $\psi\to \phi\in Q$.

\end{enumerate*}

Then, define
\[
\InfPro T\phi \ \equiv \   \forall Q \big ( \spc T{Q} \rightarrow \phi\in Q \big ). 
\]
\end{definition}

In view of Lemma \ref{LemmFromBelow}, this fixed point is reached after some countable ordinal, which gives us the following:

\begin{proposition}\label{PropCountableInfty}
Given a theory $T$ and $\phi\in \Pi^1_\omega$, $[\infty]_T\phi$ holds if and only if $[\xi]_T\phi$ holds for some $\xi<\omega_1$.
\end{proposition}

As before, we may also consider saturated provabiltiy operators with oracles, and we write $[\infty|A]_T\phi$ instead of $[\infty]_{T|A}\phi$. Since these provability operators are defined via a least fixed point, in view of Lemma \ref{LemmFixPica}, their existence can be readily proven in $\pica$.

\begin{lemma}\label{LemmPICAExists}
Let $T$ be any theory, possibly with oracles. Then, it is provable in $\pica$ that there exists a set $Q$ such that $\spc {T}Q$ holds.
\end{lemma}

This notion of provability allows us to represent $\pica$ in terms of a strong consistency assertion, in the spirit of Theorems \ref{TheoPaDiamond} and \ref{TheoATRDiamond}. The following is proven in \cite{FernandezImpredicative}:

\begin{theorem}\label{TheoPicaDiamond}
$\pica\equiv \eca + \forall X \, \langle \infty|X\rangle_T\top.$
\end{theorem}

This suggests that studying worms which contain the modality $\langle\infty\rangle$ may be instrumental in studying theories capable of reasoning about least fixed points. In view of Proposition \ref{PropCountableInfty}, we may identify $\langle \infty\rangle$ with $\langle\Omega\rangle$ for some ordinal $\Omega$ large enough so that $[\infty]_T\phi$ is equivalent to $[\Omega]_T\phi$; we can take $\Omega=\omega_1$, for example, but a large enough countable ordinal will do. In the next section, we will see how adding uncountable ordinals to our notation system allows us to provide notations for much larger countable ordinals as well.

\subsection{Beyond the Bachmann-Howard ordinal}

It is not hard to see that $\varepsilon_0$ and $\Gamma_0$ are countable; for example, it is an easy consequence of Theorem \ref{TheoAutWorm}. With a bit of extra work, one can see that they are computable as well, for example representing elements of $\Gamma_0$ as in Theorem \ref{TheoAutWorm}. However, this does not mean that uncountable ordinals cannot appear as a ``detour'' in defining proof-theoretic ordinals. Indeed, the Bachmann-Howard ordinal precisely arises by adding a symbol for an uncountable ordinal. Before continuing, let us recall a few basic properties of cardinals and cardinalities.

\begin{definition}
Given a set $A$, we define $\card A$ to be the least ordinal $\kappa$ such that there is a bijection $f\colon A\to\kappa$.
If $\kappa=\card\kappa$, we say that $\kappa$ is a {\em cardinal.}
\end{definition}

The following properties are well-known and discussed in detail, for example, in \cite{Jech:2002:SetTheory}.

\begin{lemma}\label{LemmCardinal}
Let $A,B$ be sets. Then,

\begin{enumerate}

\item  $\card{A\cup B}\leq\max\{\omega,\card A,\card B\}$;

\item if at least one of $A,B$ is infinite, then $\card{A\cup B}=\max\{\card A,\card B\}$;

\item $\card{A\times B}\leq\max\{\omega,\card A,\card B\}$,

\item if one of $A,B$ is infinite and both are non-empty, $\card{A\times B}=\max\{\card A,\card B\}$, and

\item if $\{A_i : i\in I\}$ is a family of sets, then
\[\card{\bigcup_{i\in I} A_i}\leq\max \Big \{\omega,\sup_{i\in I}\card {A_i},\card I \Big \}.\]

\end{enumerate}

\end{lemma}

These results readily allow us to compute the cardinalities of ordinals obtained using addition and multiplication.

\begin{lemma}\label{LemmCardOrd}

Let $\alpha,\beta$ be ordinals. Then,
\begin{enumerate}

\item  $\card{\alpha+\beta}\leq\max\{\omega,\card\alpha,\card\beta\}$;

\item  $\card{\alpha+\beta}=\max\{\card\alpha,\card\beta\}$ if one of the two is infinite;

\item  $\card{\alpha\beta}\leq\max\{\omega,\card\alpha,\card\beta\}$, and

\item  $\card{\alpha\beta}=\max\{\card\alpha,\card\beta\}$ if one of the two is infinite and both are non-zero.

\end{enumerate}
\end{lemma}

\proof
These claims are immediate from Lemma \ref{LemmCardinal} if we observe that $\alpha+\beta$ is the disjoint union of $\alpha$ with $[\alpha,\alpha+\beta)$, and $\card{[\alpha,\alpha+\beta)}=\card\beta$, while $\alpha\beta$ is in bijection with $\alpha\times\beta$ (via the map $\alpha\xi+\zeta\mapsto (\zeta,\xi)\in \alpha\times\beta$).
\endproof

Similar claims hold for the hyperexponential function:

\begin{lemma}\label{LemmExpCard}
Let $\alpha,\beta$ be arbitrary ordinals. Then, $\card{e^\alpha\beta }\leq \max\{\omega,\card\alpha,\card\beta\}.$ If moreover $\beta>0$ and $\max\{\alpha,\beta\}\geq\omega$, then $\card{e^\alpha\beta }= \max\{\card\alpha,\card\beta\}.$
\end{lemma}

\proof
To bound $\card{e^\alpha\beta}$, we proceed by induction on $\alpha$ with a secondary induction on $\beta$ to show that $\card{e^\alpha\beta}\leq\max\{\omega,\card\alpha,\card\beta\}$. We consider several cases, using Lemma \ref{LemmExpRec}. If $\alpha=0$, then $e^0\beta=\beta$, so the claim is obviously true. If $\beta=0$, we see that $e^\alpha 0=0$, so the claim holds as well. For $\alpha=1$ and $\beta=\gamma+1$,
\begin{align*}
e(\gamma+1)&= \lim_{n<\omega} (1+e\gamma) \cdot n\stackrel{\text{\sc ih}} \leq \max\{\omega,\card\alpha,\card\beta\}.
\end{align*}
If $\alpha$ is a limit and $\beta=1$,
\begin{align*}
e^{\alpha}1&=\lim_{\gamma<\alpha} e^\gamma 1 \stackrel{\text{\sc ih}} \leq\max\{\omega,\card\alpha,\card\beta\}.
\end{align*}
For $\alpha=\gamma+1$ with $\gamma>0$ we obtain
\begin{align*}
e^{\gamma+1}\beta&=e^\gamma e \beta \stackrel{\text{\sc ih}} \leq \max\{\card\alpha,\card{e\beta}\}\}
\\
&\stackrel{\text{\sc ih}} \leq \max\{\omega,\card\alpha,\max\{\omega,\card\alpha,\card{\beta}\}\}=\max\{\omega,\card\alpha,\card\beta\}.
\end{align*}
If $\beta$ is a limit, then we obtain
\[e^\alpha \beta =\lim_{\gamma<\beta} e^\alpha \gamma \stackrel{\text{\sc ih}} \leq\max\{ \omega,\card\alpha,\card\beta\}.\]
Finally, for limit $\alpha$ and $\beta=\delta+1$ we obtain
\begin{align*}
e^{\alpha}(\delta+1)&=\lim_{\gamma<\alpha} e^ \gamma (e^\alpha(\delta)+1)\stackrel{\text{\sc ih}} \leq \max\{\omega,\card\alpha,\card{\beta}\}.
\end{align*}
Since this covers all cases, the result follows.

For the second claim, if $\beta>0$, then ${e^\alpha \beta}\geq\max\{\alpha,\beta\}$, so $\card{e^\alpha \beta}\geq \max\{\card\alpha,\card\beta\}$ and we obtain the desired equality if one of the two is infinite.
\endproof

\begin{corollary}\label{CorKappaExp}
If $\kappa$ is an uncountable cardinal, then $\kappa$ is additively indecomposable and $e^\kappa 1=\kappa$.
\end{corollary}

\proof
We know that $e^\kappa 1 \geq \kappa$. However, from Lemma \ref{LemmExpCard}, $|e^\xi 1| < \kappa$ whenever $\xi < \kappa$, so that $e^\xi 1 < \kappa$. But $e^\kappa 1 = \lim_{\xi<\kappa} e^\xi 1$, so $e^\kappa 1 = \kappa$, from which it also follows that $\kappa =  \omega^\kappa$ and thus is additively indecomposable.
\endproof

We have a simiar situation with worms; it is very easy to infer the cardinality of $o(\fw)$ by looking at the entries in $\fw$.

\begin{lemma}\label{LemmWormCard}
If $\fw \in\Worms$ then $\card{o(\fw)}\leq\card{\max\omega\fw}$. If moreover $\fw\not=\top$ and $\max\fw\geq \omega$, then $\card{o(\mathfrak w)}= \card{\max \fw}$.
\end{lemma}

\proof
We prove by induction on $\cnorm\fw$ that $\card{o(\fw)}\leq\card{\max\omega\fw}$. For $\fw=\top$ this is obvious. Otherwise, let $\mu=\min\fw$. If $\mu=0$, then $o(\fw)=ob(\fw)+1+oh(\fw)$, so that by Lemma \ref{LemmExpCard},
\[\card{o(\fw)}=\card{ob(\fw)+1+oh(\fw)}\leq\max\{\omega,\card{ob(\fw)}, 1, \card{oh(\fw)}\}.\]
By the induction hypothesis $\card{oh(\fw)}\leq\card{\max\omega h(\fw)}\leq\card{\max\omega\fw}$ and similarly for $\card{ob(\fw)}$, so we obtain $\card{o(\fw)}\leq  \card{\max\omega\fw}$.

If $\mu>0$, then $o(\fw)=e^\mu(\mu\downarrow\fw)$. Since $\mu,\max(\mu\downarrow\fw)\leq\max\fw$ and $ \cnorm{\mu \downarrow \fw} < \cnorm\fw$, we use the induction hypothesis and Lemma \ref{LemmExpCard} once again to see that 
\[\card{o(\fw)}=\card{e^\mu o(\mu\downarrow\fw)}\leq\max\{\omega, \card\mu,\card{\max(\mu\downarrow\fw)}\}\leq\card{\max\omega\fw}.\]
The claim follows.

For the second claim, if $\fw\not=\top$ and $\max\fw\geq\omega$, then by Lemma \ref{LemmLowerBound}.\ref{LemmLowerBoundItOne},
$o(\fw)\geq\max\fw$, so
\[\card{o(\fw)}\geq \card{\max\fw}=\card{\max\omega\fw},\]
and thus we obtain equality.
\endproof

Similarly, closure under a function $f$ does not produce many more ordinals than we had to begin with:

\begin{lemma}\label{LemmCardFClose}
If $f\colon\ord^{<\omega}\dashrightarrow\ord$ and $\Theta$ is a set of ordinals, then
\[\card\Theta\leq \card{\close f\Theta} \leq \max\{\omega,\card\Theta\}.\]
\end{lemma}

\proof
We inductively check that
\begin{equation}\label{EqCardTheta}
\card\Theta\leq \card{\Theta^f_n} \leq \max\{\omega,\card\Theta\},
\end{equation}
from which the lemma follows using the fact that $\close f\Theta=\bigcup_{n<\omega}\Theta^f_n$.

We have that $\Theta^f_0=\Theta$, so \eqref{EqCardTheta} holds. Now, assume inductively that \eqref{EqCardTheta} holds for $n$. Then, $\Theta^f_{n+1}=\Theta^f_n\cup \iter f{\Theta^f_n}$; by the induction hypothesis,
\[\card\Theta\leq \card{\Theta^f_n}\leq \card{\Theta^f_{n+1}}.\]
Now, elements of $\iter f\Theta$ are of the form $f(\xi_1,\hdots,\xi_m)$ with $\xi_1,\hdots,\xi_m\in \Theta^f_n$; but there are at most $\max\{\omega,\card  {\Theta^f_n}\}$ of these, so
\[\card{\iter f{\Theta^f_n}}\leq \max \Big \{\omega, \card {\Theta^f_n} \Big \}\stackrel{\text{\sc ih}}\leq \max\{\omega, \card {\Theta}\},\]
from which it follows that
\[\card {\Theta^f_{n+1}}=\card{\Theta^f_n\cup \iter f{\Theta^f_n}}\leq \max \Big \{\omega, \card{\Theta^f_n}, \card{\iter f{\Theta^f_n}} \Big \}\stackrel{\text{\sc ih}}\leq \max\{\omega, \card {\Theta}\}.\qedhere\]
\endproof

This tells us that none of the ordinal operations we have discussed so far will give rise to any uncountable ordinals. So, we may add one directly; we can then use it to produce more countable ordinals using {\em collapsing functions.} We shall present them using hyperexponentials rather than Veblen functions, although this change is merely cosmetic as the two define the same ordinals. It is standard to use $\Omega$ to denote a `big' ordinal, which for convenience may be assumed to be $\omega_1$. However, we mention that, with some additional technical work, one can take $\Omega=\omega^{CK}_1$, the first non-computable ordinal \cite{Rathjen1993}.

\begin{definition}
Let $\Omega,\xi$ be ordinals. We simultaneously define the sets $C (\xi)$ and the ordinals $\uppsi (\xi)$ by induction on $\xi$ as follows:
\begin{enumerate}
\item $C (\xi)$ is the least set such that
\begin{enumerate}

\item $\Omega\in C (\xi)$,
\item $C(\xi)$ is hyperexponentially closed, and
\item if $\alpha \in  C (\xi)$ and $\alpha<\xi$ then $\uppsi (\alpha)\in  C (\xi)$.
\end{enumerate}
\item $\uppsi (\xi)$ is the least $\lambda$ such that $\lambda\not\in C (\xi)$.
\end{enumerate}
\end{definition}

In the notation of Definition \ref{DefFClose}, let ${\rm BH}_\xi$ be the pair of functions $\{{\rm HE},\uppsi\upharpoonright\xi\}$. Then,
\[C(\xi)=\close{{\rm BH}_\xi}{\{0,1,\Omega\}}.\]
Thus our previous work on closures under ordinal functions readily applies to the sets $C(\xi)$. The function $\uppsi$ appears in the ordinal analysis of systems such as ${\rm ID}_1$ and Kripke-Platek set-theory with infinity \cite{Pohlers:2009:PTBook}.

\begin{lemma}\label{LemmExpPsi}
If $\xi$ is any ordinal, then $\uppsi(\xi)$ is additively indecomposable and $\uppsi(\xi)=e^{\uppsi(\xi)}1$.
\end{lemma}

\proof
To see that $\uppsi(\xi)$ is additively indecomposable, we will assume otherwise and reach a contradiction. Hence, suppose that $\uppsi(\xi)=\alpha+\beta$ with $\alpha,\beta<\uppsi(\xi)$. By definition of $\uppsi(\xi)$ we have that $\alpha,\beta\in C(\xi)$, hence $\uppsi(\xi)=\alpha+\beta\in C(\xi)$, contradicting its definition.

Next we show that $\uppsi(\xi)=e^{\uppsi(\xi)}1$. By Proposition \ref{PropEHNF}, there are $\alpha,\beta$ with $\beta$ either $1$ or additively decomposable such that $\uppsi(\xi)=e^\alpha\beta$. Since $\uppsi(\xi)$ is additively indecomposable we have that $\beta\not=\uppsi(\xi)$, and since $e^\alpha$ is normal, we have that $\beta<\uppsi(\xi)$. Now, towards a contradiction, assume that $\alpha<\uppsi(\xi)$; then $\alpha,\beta\in C(\xi)$ so $\uppsi(\xi)\in C(\xi)$, contrary to its definition. We conclude that $\alpha=\uppsi(\xi)$, and again since $e^{\uppsi(\xi)}$ is normal and $e^{\uppsi(\xi)}1\geq\uppsi(\xi)$, that $\beta=1$.
\endproof

We remark that the above lemma already tells us that the {\em countable} ordinals we can construct using $\uppsi$ are much bigger than $\Gamma_0$; indeed, we already have that $\Gamma_0=\uppsi(0)$, and this is only scratching the surface of our notation system: ordinals such as $\uppsi(\Omega)$ or $\uppsi(e^{\omega}(\Omega+1))$ are much larger. The latter is the Howard-Bachmann ordinal $\uppsi(\varepsilon_{\Omega+1})$, as one can readily check that $e^\omega\xi=\varepsilon_\xi$ for all $\xi$ using Proposition \ref{PropExptoveb}.

\begin{lemma}\label{LemmCClosedO}
Assume that $\Omega$ is such that $\Omega=e^\Omega 1$. If $\xi$ is any ordinal, then $C(\xi)$ is hyperexponentially perfect.
\end{lemma}

\proof We already know that $C(\xi)$ is hyperexponentially closed, so it remains to show that it is reductive. Let $\zeta\in C(\xi)$. By Lemma \ref{LemmPropFClose}.\ref{LemmPropFCloseItFour}, either $\zeta\in\{0,1,\Omega\}$, there are $\alpha,\beta,\gamma\not=\zeta$ with $\zeta=e^{\alpha}(\beta+\gamma)$, or $\zeta=\uppsi(\alpha)$ for some $\alpha\in C(\xi)\cap\xi$. If $\zeta<2$, there is nothing to prove, so we assume otherwise.

First assume that $\zeta=e^{\alpha}(\beta+\gamma)$. If $\zeta$ is additively decomposable, by Lemma \ref{LemmExpPsi}, we cannot have that $\alpha>0$, so we conclude that $\zeta=e^0(\beta+\gamma)=\beta+\gamma$, as needed. If it is additively indecomposable, since $\beta+\gamma\in C(\xi)$, then we already have that $\zeta=e^\alpha(\beta+\gamma)$ with $\alpha,\beta+\gamma\in C(\xi)$. In all other cases, $\zeta$ must be additively indecomposable. If $\zeta=\Omega$, then $\zeta=e^\Omega 1$ and $\Omega,1\in C(\xi)$, and if $\zeta=\uppsi(\alpha)$, by Lemma \ref{LemmExpPsi}, $\zeta=e^\zeta 1$, with $\zeta,1\in C(\xi)$.
\endproof

The intention of the function $ \uppsi$ is to produce new countable ordinals from possibly uncountable ones. Let us see that this is the case:

\begin{lemma}\label{LemmPsiCountable}
Let $\xi$ be any ordinal and $\Omega=\omega_1$. Then, $C(\xi)$ is countable and $\uppsi(\xi)<\Omega$.
\end{lemma}

\proof
The first claim is an instance of Lemma \ref{LemmCardFClose}, while the second is immediate from the first.
\endproof

Observe that $\sup C(\xi)=\Gamma_{\Omega+1}$, the first hyperexponentially closed ordinal which is greater than $\Omega$, and thus the smallest ordinal not contained in any $C(\xi)$ is $\uppsi(\Gamma_{\Omega+1})$. However, our worm notation will give slightly smaller ordinals. Thus it will be convenient to consider a ``cut-off'' version of the sets $C(\xi)$. Let us see that these cut-off versions maintain a restricted version of the minimality property of $C(\xi)$.

\begin{lemma}\label{LemmBoundMinC}
If $\mu\leq \lambda$ are ordinals such that $\Omega<\lambda$, then $C(\mu)\cap \lambda$ is the least set $D$ such that:
\begin{enumerate}[label=(\roman*)]

\item $0,1,\Omega\in D$;\label{LemmBoundMinCOne}

\item if $\alpha,\beta,\gamma\in D$ and $e^\alpha(\beta+\gamma)<\lambda$ then $e^\alpha(\beta+\gamma)\in D$, and \label{LemmBoundMinCTwo}

\item if $\alpha\in D\cap\mu$ then $\uppsi (\alpha) \in D$.\label{LemmBoundMinCThree}

\end{enumerate}
\end{lemma}

\proof
First we observe that $C(\mu)\cap \lambda$ indeed satisfies \ref{LemmBoundMinCOne}-\ref{LemmBoundMinCThree}, where for the first item we use the assumption that $\Omega<\lambda$ and for the third we use Lemma \ref{LemmPsiCountable} to see that $\uppsi(\alpha)<\Omega<\lambda$. Now, let $D$ be the least set satisfying \ref{LemmBoundMinCOne}-\ref{LemmBoundMinCThree}, and consider
\[D'=D\cup \big ( C(\mu)\setminus\lambda \big ) .\]
One readily verifies that $0,1, \Omega\in D'$, and that if $\alpha,\beta,\gamma\in D'$ then $e^\alpha(\beta+\gamma)\in D'$ (using the fact that $D\subseteq C(\mu)\cap \lambda\subseteq C(\mu)$ by minimality of $D$). Finally, if $\alpha<\mu$ and $\alpha\in D'$, then since $\mu\leq\lambda$ we have that $\alpha\in D$, and since $D$ satisfies \ref{LemmBoundMinCThree} we have that $\uppsi(\alpha)\in D\subseteq D' $. But by definition $C(\mu)$ is the least set with these properties, so we obtain $C(\mu)\subseteq D'$, and hence
\[C(\mu)\cap\lambda\subseteq D' \cap \lambda = D,\]
as was to be shown.
\endproof

We remark that the ordinal $\uppsi(\Gamma_{\Omega + 1})$ is computable, meaning that it is isomorphic to an ordering $\langle A,{\peq}\rangle$, where $A\subseteq \mathbb N$ and both $A$ and $\peq$ are $\Delta^0_1$-definable; however, we will not go into details here, and instead refer the reader to a text such as \cite{Pohlers:2009:PTBook}.

\subsection{Collapsing uncountable worms}

Now let us turn our attention to uncountable worms. The general idea is as follows. We have seen in Theorem \ref{TheoAutWorm} that worms give us a notation system for $\Gamma_0$ if we interpret $\langle \fw\rangle$ as $\langle o(\fw)\rangle$. Meanwhile, now we have a new modality $\langle\infty\rangle$, which we can regard as $\langle \omega_1 \rangle$. Note that, by Corollary \ref{CorKappaExp},
\[o(\langle\omega_1\rangle\top)=e^{\omega_1} o(\langle 0\rangle\top)= \omega_1 .\]
Thus if we add the new symbol $\Omega$ representing $\langle \omega_1\rangle$ to Beklemishev's autonomous worms, we see inductively that
\[\langle\omega_1 \rangle\top=\Bmap{\Omega }=\Bmap{\text{\tt (}\Omega\text{\tt )} }=\Bmap{\text{\tt ((}\Omega\text{\tt ))}}\hdots\]
Moreover, if such operations are to be interpreted proof-theoretically using iterated $\omega$-rules, then in view of Proposition \ref{PropCountableInfty} we have that $\langle \omega_1 \rangle\top\equiv \langle \omega_1 +\xi \rangle\top$ for any ordinal $\xi$. Thus we also would have, for example,
\[\langle\omega_1\rangle\top=\Bmap{\text{\tt (}\Omega\text{\tt )}}=\Bmap{\text{\tt (()}\Omega\text{\tt )}}=\Bmap{\text{\tt (}\Omega \Omega\text{\tt )}}\hdots\]
This would lead to quite a wasteful notation system! Thus we will adopt the following rule: when writing an autonomous worm $({\aw})\top$, if ${\sf o}({\aw})$ is countable, then we will take it at face-value and interpret $({\aw})\top$ as $\langle {\sf o}({\aw})\rangle \top$. However, if ${\sf o}({\aw})$ is uncountable, we will first ``project'' it to a countable ordinal, in order to represent large countable worms.

Of course, projections will be very similar to collapsing functions; however, given that countable ordinals are taken at face value, these projections will have the property that $\uppi=\uppi\circ\uppi$ (thus their name). Other than that, their construction is very similar to that of $\uppsi$:

\begin{definition}\label{DefU}
Given a worm $\fw \in \mathbb W$ and an ordinal $\Omega$, we define $U (\fw)\subseteq \ord$ and a map $\cond \colon \Worms \to \ord$ by induction on $\fw$ along $\wle{}$ as follows.
\begin{enumerate}

\item Let $U ({\fw})$ be the least set of ordinals such that
\begin{enumerate}
\item $\Omega\in U (\fw)$,

\item if $\fu \sqsubset U (\fw)$ and $\fu \wle{}\fw$ then $\cond (\fu)\in U (\fw)$.
\end{enumerate}

\item Then, set

\begin{enumerate}

\item $\cond (\fw)=o(\fw)$ if $\fw\sqsubset \Omega $,

\item  otherwise, set $\cond (\fw)$ to be the least ordinal $\mu$ such that $\mu\not\in U(\fw)$.

\end{enumerate}

\end{enumerate}

\end{definition}

We will write $\cond(\fw)$ or $\cond\fw$ indistinctly. Once again, we can write Definition \ref{DefU} in the terminology of Definition \ref{DefFClose} by setting
\[U(\fw)=\close{\cond\upharpoonright\{\fv : \fv\wle{}\fw\}}{\{\Omega\}}.\]
Thus Lemma \ref{LemmCardFClose} gives us the following:

\begin{lemma}\label{LemmUCountable}
For every worm $\fw$, $U(\fw)$ and $\cond\fw$ are countable.
\end{lemma}

Throughout this section we will assume that $\Omega=\omega_1$, so that from Lemma \ref{LemmUCountable} we obtain $\cond\fw<\Omega$ for all worms $\fw$. As was the case for defining $\uppsi$, with some extra technical work we can take $\Omega=\omega^{CK}_1$ instead.

Note that $U(\fw)$ itself is not worm-closed, as it does not contain, for example, the ordinal $\Omega+1 = o (0\Omega \top)$. However, its countable part is indeed worm-perfect. The next lemmas will establish this fact. First, we show that it is worm-closed.

\begin{lemma}\label{LemmOIsWorm}
For any worm $\fv$ with $o(\fv)\geq\Omega$, $U(\fv)\cap \Omega$ is worm-closed.
\end{lemma}

\proof
By Corollary \ref{CorBound}.\ref{CorBoundC}, if $\fw\sqsubset U(\fv)\cap \Omega$, then $\fw\wle{}\Omega\top \wle{}\fv$, so that $o(\fw)=\cond (\fw)\in U(\fv)$. But by Lemma \ref{LemmWormCard}, $o(\fw)<\Omega$, so $o(\fw) \in U(\fv)\cap \Omega$ as needed.
\endproof

Recall that Lemma \ref{LemmExpPsi} states that $\uppsi (\xi)= e^{\uppsi(\xi)}1$. Next, we show that $\cond$ enjoys a similar property.

\begin{lemma}\label{LemmCondFix}
If $o(\fw)\geq \Omega$, then $o(\langle \cond \fw \rangle\top)= \cond \fw$.
\end{lemma}

\proof
Suppose not. Then, by Lemma \ref{LemmLowerBound}.\ref{LemmLowerBoundItOne}, $\cond \fw<o ( \langle \cond\fw\rangle \top ) $, so that by Corollary \ref{CorSurj}, there is a worm $\fv$ such that $o(\fv)= \cond\fw$. Since $o(\fv)<o(\langle\cond\fw\rangle\top)$, by Lemma \ref{LemmLowerBound}.\ref{LemmLowerBoundItOne} once again, we must have that $\fv\sqsubset \cond\fw \subseteq U(\fw)$. But by Lemma \ref{LemmOIsWorm}, $\cond \fw=o(\fv)\in U(\fw)$, contradicting the definition of $\cond\fw$.
\endproof

\begin{lemma}\label{LemmMinusO}
For any worm $\fw$, $U(\fw)\cap\Omega $ is worm-perfect and
\[U(\fw)\cap \Omega = U(\fw)\setminus\{\Omega\}.\]
\end{lemma}

\proof
For the first claim, in view of Lemma \ref{LemmOIsWorm}, it remains to show that if $\xi\in U(\fw)\cap\Omega $, then $\xi=o(\fv)$ for some $\fv\sqsubset U(\fw)\cap\Omega$. By definition of $U(\fw)$, if $\xi\in U(\fw)\cap\Omega $, then $\xi=\cond\fu$ for some $\fu\sqsubset U(\fw)$. If $\fu\sqsubset \Omega$, then $\xi=\cond\fu=o\fu$. Otherwise, by Lemma \ref{LemmCondFix}, $\xi=\cond\fu=o(\langle\cond\fu\rangle\top)$.

The second claim is immediate from Lemma \ref{LemmUCountable} and the assumption that $\Omega=\omega_1$, since $\cond\fw<\Omega$ for every worm $\fw$.
\endproof

However, as we have mentioned, $U(\fw)$ itself is not worm-closed, and neither is $\iter o{U(\fw)}$. Nevertheless, the latter does satisfy a bounded form of hyperexponential closure:

\begin{lemma}\label{LemmUClosed}
Given any worm $\fw$ and ordinals $\alpha,\beta$, if $\alpha,\beta\in \ocl{U(\fw)}$ and $e^\alpha\beta< e^{\Omega+1}1$ then $e^ \alpha\beta\in \ocl{ U(\fw)}$.
\end{lemma}

\begin{proof}
If $\alpha,\beta\in \ocl{ U(\fw)}$ and $e^\alpha\beta< e^{\Omega+1}1$, we may assume without loss of generality that $\beta>0$ (since otherwise $e^\alpha\beta=0$), so by the assumption that $o(\fw)<e^{\Omega+1}1=e^\Omega\omega$, we see by monotonicity that either $\alpha<\Omega$, or $\alpha=\Omega$ and $\beta<\omega$.

First assume that $\alpha<\Omega$, and let
\[\fv=\lambda_1\hdots\lambda_n\top \sqsubset U (\fw)\]
be such that $\beta= o(\fv)$. In view of Lemma \ref{LemmMinusO}, for each $\lambda\in [1,n]$, either $\lambda_i = \Omega$, in which case $\alpha+\lambda_i=\lambda_i$, or $\lambda_i\in U(\fw)\cap\Omega$, which since $U(\fw)\cap\Omega$ is worm-perfect (Lemma \ref{LemmMinusO}) gives us $\alpha+\lambda_i\in U(\fw)\cap \Omega\subseteq U(\fw)$ (Lemma \ref{LemmSum}). Thus $\alpha+\lambda_i\in U(\fw)$ for each $i$, hence $\alpha\uparrow \fv\sqsubset U (\fw)$, and
\[o(\alpha\uparrow \fv) = e^\alpha o(\fv)=e^\alpha \beta.\]
Otherwise, $\alpha=\Omega$, so $\beta<\omega$ and we see that $o(\langle\Omega\rangle^\beta\top)=e^\alpha\beta$. In either case, it follows that $e^\alpha\beta\in \ocl{(U(\fw))}$.
\end{proof}

\begin{lemma}\label{LemmEOmBound}
Suppose that $\Omega=\omega_1$. Then, given any worm $\fw$,
\[e^{\Omega+1}1=\sup \Big \{o(\fv):\exists\fw \ \big ( \fv\sqsubset U (\fw) \big) \Big \}.\]
\end{lemma}

\proof
Let
\[\Lambda=\sup \Big \{o(\fv):\exists\fw \ \big ( \fv\sqsubset U (\fw) \big) \Big \}.\]
We have that
\[e^{\Omega+1}1=e^\Omega\omega=\lim_{n<\omega}e^\Omega n.\]
But, $e^\Omega n=o( \langle \Omega \rangle ^n\top)$, so $e^{\Omega+1}1\leq \Lambda$.

To see that $\Lambda\leq e^{\Omega+1}1$, proceed by induction on $\cnorm\fv$ to show that if $\fv\sqsubset U (\fw)$ for some $\fw$, then $o(\fv)<e^{\Omega+1}1$.

If $\fv=\top$ there is nothing to prove, and if $\min\fv=0$ then by the induction hypothesis, $oh(\fv),ob(\fv)<e^{\Omega+1}1$. Since the latter is additively indecomposable,
\[o(\fv)=ob(\fv)+1+oh(\fv)<e^{\Omega+1}1.\]
Finally, if $\mu=\min\fv>0$, then $o(\fv)=e^\mu o(\mu\downarrow \fv)$. Consider two cases. If $\mu<\Omega$, then since by the induction hypothesis $o(\mu\downarrow \fv)<e^{\Omega+1}1$, we obtain
\[e^\mu o(\mu\downarrow\fv)<e^\mu e^{\Omega+1}1=e^{\mu+\Omega+1}1=e^{\Omega+1}1.\]
Otherwise, $\mu=\Omega$, but this means that $\mu\downarrow \fv=0^n\top$ for some $n$, hence $o(\fv)=e^{\Omega}n<e^{\Omega+1}1$.
\endproof

The above results tell us that $\cond$ behaves a lot like a version of $\uppsi$ that is restricted to $e^{\Omega+1}1$. Let us see that this is, in fact, the case.

\begin{lemma}\label{LemmPsiCond}
For every worm $\fw$ with $o(\fw)\in [\Omega, e^{\Omega+1}1]$,

\begin{enumerate}

\item $C \big (-\Omega+ o(\fw) \, \big ) \cap e^{\Omega+1}1 = \ocl{ U(\fw)}$, and

\item $\cond (\fw)=\uppsi \big (-\Omega+o(\fw) \big )$.

\end{enumerate}
\end{lemma}

\proof
We prove both claims by induction on $o(\fw)$. Set $C=C(-\Omega+ o(\fw))$. First let us show that
\[C\cap e^{\Omega+1}1\subseteq \ocl{ U(\fw)}.\]
Note that by Lemma \ref{LemmBoundMinC}, $C \cap e^{\Omega+1}1$ is the least set containing $0,1,\Omega$, closed under $\alpha,\beta,\gamma\mapsto e^\alpha(\beta+\gamma)$ below $e^{\Omega+1}1$, and closed under $\uppsi\upharpoonright (-\Omega+ o(\fw))$. But by Lemma \ref{LemmSum}, $\ocl{ U(\fw)}$ is closed under addition and by Lemma \ref{LemmUClosed}, by hyperexponentiation below $e^{\Omega+1}1$, so we only need to check that it is closed under $\uppsi\upharpoonright \big ( -\Omega + o(\fw) \big ) $. 

If $\alpha\in \ocl{U(\fw)}$ and $\alpha<o(-\Omega+o(\fw))$, then by Lemma \ref{LemmSum} we have that $\Omega+\alpha=o(\fu)$ for some $\fu\sqsubset U(\fw)$. Then, by the induction hypothesis,
\[ \uppsi(\alpha) = \uppsi(-\Omega+o(\fu)) =  \cond(\fu) \in U(\fw),\] 
so that $\cond(\fu)\top\sqsubset U(\fw)$ and by Lemma \ref{LemmCondFix}, $\cond(\fu)=o(\cond(\fu)\top)$, as needed. Thus by the minimality of $C\cap e^{\Omega+1}1$, we conclude that $C \cap e^{\Omega+1} 1 \subset \ocl{ U(\fw)}.$

Next we check that
\[\ocl{ U(\fw)} \subseteq C\cap e^{\Omega+1}1.\]
By Lemma \ref{LemmEOmBound}, $\ocl{U(\fw)}\subseteq e^{\Omega+1}1$, so we only need to prove that $\ocl{ U(\fw)} \subseteq C$. But, in view of Lemmas \ref{LemmCClosedO} and Lemma \ref{LemmClosedSum}, $C$ is worm-perfect. Thus to show that $\ocl{U(\fw)} \subseteq C $, it suffices to prove that $U(\fw)\subseteq C$. As before, we show that $C$ satisfies the inductive definition of $U(\fw)$.

Let $\fv\sqsubset C$ be such that $\fv\wle{}\fw$. Once again by Lemma \ref{LemmCClosedO}, we have that $o(\fv) \in C$. Now, if $o(\fv)<\Omega$, then this gives us $\cond\fv=o(\fv)\in C$. Otherwise, $-\Omega+o(\fv) < -\Omega + o(\fw)$, and thus $\uppsi \big (-\Omega+o(\fv) \, \big ) \in C$. But, by the induction hypothesis, $\uppsi \big ({-}\Omega+o(\fv) \, \big )=\cond \fv$, so that $\cond\fv\in C$, as needed. By minimality of $U(\fw)$, we conclude that $ U(\fw)  \subseteq C$ and thus $\ocl{U(\fw)} \subseteq C$.

Since we have shown both inclusions, we conclude that
\[\ocl{U(\fw)} = C \cap e^{\Omega+1}1.\]
Moreover, $\uppsi \big (-\Omega+ o(\fw) \big )$ is defined as the least ordinal not in $C=C \big (-\Omega+ o(\fw) \big )$, and since $C$ is countable it is also the least ordinal not in $C\cap \Omega$. Similarly, $\cond\fw$ is the least ordinal not in ${ U(\fw) } \cap \Omega=\ocl{ U(\fw) } \cap \Omega$. Since these two sets are equal, it follows also that $\uppsi \big (-\Omega+ o(\fw) \big )=\cond\fw$.
\endproof

\begin{corollary}\label{CorPsiCond}
$\cond(\langle \Omega+1\rangle\top)=\uppsi(e^{\Omega+1}1)$.
\end{corollary}

\proof
Immediate from Lemma \ref{LemmPsiCond} using the fact that
\[e^{\Omega+1}1= e^{\Omega + 1} o(\langle 0\rangle \top) = o(\langle e^{\Omega +1 }\rangle \top ).\qedhere\]
\endproof

\subsection{Impredicative worm notations}

Now let us extend Beklemishev's autonomous worms with the new modality $\Omega$ and projections of uncountable worms. Aside from the addition of $\Omega$, the presentation is very similar to that of Section \ref{SubsecAutWorms}.

\begin{definition}
Define the set of impredicative autonomous worms to be the least set $\WAut_\Omega$ such that
\begin{enumerate}[label=(\roman*)]

\item $\top \in \WAut_\Omega$, and

\item if ${\aw}, {\av}\in \WAut_\Omega$, then

\begin{multicols}2
\begin{enumerate}

\item $({\aw}){\av}\in \WAut_\Omega$, and

\item  $ \Omega{\av}\in \WAut_\Omega$.

\end{enumerate}

\end{multicols}

\end{enumerate}

\end{definition}

As before, the intention is for impredicative autonomous worms to be interpreted as standard worms. We do this via the following translation:

\begin{definition}
We define a map $\cdot^\cond\colon \WAut_\Omega\to\Worms$ given by
\begin{enumerate}

\item $\top^\cond=\top$,

\item $\big ( \text{\tt (}{\aw}\text{\tt )}{\av} \big )^\cond=\langle \cond({\aw}^\cond)\rangle {\av}^\cond$, and

\item $(\Omega{\av})^\cond=\langle \Omega \rangle {\av}^\cond$.

\end{enumerate}
\end{definition}

Every ordinal in $U(\langle \Omega+1\rangle\top)\cap \Omega$ can be represented as an autonomous worm. Below, define ${\sf p}{\aw}=\cond({\aw}^\cond)$.

\begin{lemma}\label{LemmImpAut}
If $\Omega=\omega_1$, then for every ordinal $\xi \in U(\langle \Omega+1\rangle\top)\cap \Omega$ there is ${\aw}\in \WAut_\Omega$ such that $\xi={\sf p}{\aw}$.
\end{lemma}

\proof
Using the notation of Definition \ref{DefFClose}, we prove by induction on $n$ that if $\xi\in{\{\Omega\}}^\cond_n\cap \Omega$, then there is ${\aw}\in \WAut_\Omega$ such that $\xi={\sf p} {\aw}$.  If $n=0$ there is nothing to prove, so we may assume that $n=k+1$. Write $\xi=\cond({\fv})$ with $\fv \sqsubset \{ \Omega\}^\cond_k$. If $\fv=\top$, then $\xi=0={\sf p} \top $. Otherwise, we can write $\fv = \lambda\fu $ for some worm $\fu$. By a secondary induction on the length of $\fv$, we have that $\fu={\au}^\cond$ for some ${\au}\in \WAut_\Omega$; meanwhile, either $\lambda=\Omega$, and ${\av}=\Omega{\au}\in \WAut_\Omega$ satisfies
\[{\sf p} {\av} =\cond({\av}^\cond)=\cond(\langle\Omega\rangle {\au}^\cond)=\cond(\langle\Omega\rangle \fu )=\cond(\fv)=\xi,\]
or $\lambda<\Omega$, which means that $\lambda\in \{ \Omega\}^\cond_k$, so by the induction hypothesis, $\lambda={\sf p}{\aw}$ for some ${\aw}\in \WAut_\Omega$. It follows that $\xi={\sf p} \text{\tt (}{\aw}\text{\tt )}{\av}$, as desired.
\endproof

Just as autonomous worms gave us a notation system for $\Gamma_0$, impredicative autonomous worms give us a notation system for $\uppsi \big (e^{\Omega+1}1 \big )$. 

\begin{theorem}\label{TheoImpAut}
If $\Omega=\omega_1$, then for every $\xi<\uppsi \big (e^{\Omega+1}1 \big )$ there is ${\aw}\in \WAut_\Omega$ such that $\xi={\sf p} {\aw} $.
\end{theorem}

\proof
By Corollary \ref{CorPsiCond},
\[\uppsi \big (e^{\Omega+1}1 \big )=\cond (\langle \Omega+1\rangle\top),\]
and the latter is, by definition, the least ordinal not belonging to $U(\langle\Omega+1\rangle\top)$. Moreover, $\uppsi \big (e^{\Omega+1}1 \big )$ is countable by Lemma \ref{LemmPsiCountable}, so we have that $\xi<\Omega$. It follows that
\[\uppsi \big (e^{\Omega+1}1 \big ) \subseteq U(\langle\Omega+1\rangle\top) \cap \Omega;\]
thus we obtain the claim by Lemma \ref{LemmImpAut}.
\endproof

Impredicative autonomous worms may be suitable for a consistency proof in the spirit of Theorem \ref{TheoPACons} for theories with proof-theoretic strength the Bachmann-Howard ordinal (or even slightly more powerful theories). Examples of such theories are the theory ${\rm ID}_1$ of non-iterated inductive definitions, Kripke-Platek with infinity, and {\em parameter-free} $\pica$, where the $\Pi^1_1$ comprehension axiom is restricted to formulas without free set variables. However, the proof-theoretical ordinal of unrestricted $\pica$ is quite a bit larger, and obtained by collapsing all of the ordinals $\{\aleph_n: n<\omega\}$.

We remark that our notation system does not take the oracle in $[\infty|X]_T$ into account, and it is possible that autonomous worms with oracles would indeed give us a notation system for the proof-theoretical ordinal of $\pica$. However, we will not follow this route; instead, we will pass from worms to {\em spiders,} which will allow us to obtain notations for this, and much larger, ordinals.

\section{Spiders}\label{SecSpiders}

The problem with using iterated $\omega$-rules to interpret $[\lambda]_T\phi$ is that $\glp$ no longer applies when $\lambda\geq\omega_1$; since we have that $[\omega_1+1]_T\phi$ is equivalent to $[\omega_1]_T\phi$, we cannot expect the $\glp$ axiom $\langle\omega_1\rangle\phi\to [\omega_1+1]\langle\omega_1\rangle\phi$ to hold. So the question naturally arises: what kind of (sound) provability operator could derive all true instances of $\langle\omega_1\rangle\phi$?

Well, we know that $\langle\omega_1\rangle\phi$ is equivalent to $\forall \xi{<}\omega_1 \, \langle\xi\rangle\phi$, which gives us a strategy for proving that $\langle\infty\rangle_T\phi$ holds: prove that
\[\langle 0\rangle_T\phi,\langle 1 \rangle_T\phi,\langle 2\rangle_T\phi,\hdots,\langle \omega \rangle_T\phi,\hdots,\langle\Gamma_0\rangle_T\phi,\hdots \langle \uppsi(\varepsilon_{\Omega+1})\rangle_T\phi,\hdots\]
all hold, and more generally, that $\langle \xi\rangle_T\phi$ holds for all $\xi<\omega_1$. Let us sketch some ideas for formalizing this in the language of set-theory. We remark that this material is exploratory, and will be studied in detail in upcoming work.

\subsection{$\aleph_\xi$-rules}

We use $\alang$ to denote the language of first-order set theory whose only relation symbols are $\in$ and $=$. As we did in second-order arithmetic, we use $x\subseteq y$ as a shorthand for $\forall z (z\in x\rightarrow z\in y).$ We also use $\exists! x \phi(x)$ as the standard shorthand for ``there is a unique''. Then, recall that Zermenlo-Fraenkel set theory with choice, denoted $\rm ZFC$, is the extension of first-order logic axiomatized by the universal closures of the following:

\begin{description}

\item[Extensionality:]  $(x\subseteq y\wedge y\subseteq x) \rightarrow y=x $;

\item[Foundation:] $\exists x \,\phi(x)\rightarrow \exists x \, \big (\phi(x) \wedge \forall y\in x \, \neg\phi(y)\big )$, where $\phi(x)$ is an arbitrary formula in which $y$ does not occur free;

\item[Pair:] $\exists z \, (x\in z\wedge y\in z)$;

\item[Union:] $ \exists y \, \forall z {\in} x \, (z\subseteq  y)$;

\item[Powerset:] $ \exists y \, \forall z\, (z\subseteq  x \rightarrow z\in y)$;

\item[Separation:] $ \exists y\, \forall z\,  \big (z\in y\leftrightarrow z\in x\wedge \phi(z) \big )$, where $y$ does not occur free in $\phi(z)$,

\item[Collection:] $ \forall x{\in} w \, \exists y \, \phi(x,y)\rightarrow \exists z \, \forall x{\in} w \, \exists y{\in} z \, \phi(x,y) $, where $z$ does not occur free in $\phi(x)$,

\item[Infinity:] $\exists w \, \Big (\exists x \, \big (x\in w\wedge \forall y\, (y\not\in x) \big ) \wedge \, \forall x { \in } w\, \exists y {\in} w \, \forall z (z\in y \leftrightarrow z\in x\vee z=x ) \Big ),$ and

\item[Choice:] $
\forall x {\in} w \, \Big ( \exists y  \, (y \in x) \wedge  \forall y{\in} w \, \big (\exists z  (z\in x \wedge z\in y) \rightarrow x = y \big ) \Big )$

\hfill$\rightarrow \exists z \, \forall x{\in} w \, \exists ! y \, (y\in x \wedge y\in z).
$

\end{description}

As we have stated the union and powerset axioms we may obtain sets that are too big, but we can then obtain the desired sets using separation. Observe also that the Foundation scheme states that $\in$ is well-founded; this allows us to simply define an ordinal as a transitive set all of whose elements are transitive as well, obtaining well-foundedness for free.

%%%%%%%%%%%%%%%%%%%%%%%%%%%%%%%%%%%%%%%%%%%%%%%%%%%%%%%%%%

%%%%%%%%%%%%%%%%%%%%%%%%%%%%%%%%%%%%%%%%%%%%%%%%%%%%%%%%%%

This set-theoretic context will allow us to define an analogue of the $\omega$-rule which quantifies over all elements of $\omega_1$; more generally, for any cardinal $\kappa$ we can define the {\em $\kappa$-rule} by
\[\dfrac{\langle \phi(\xi)\rangle_{\xi<\kappa}}{\forall x< \kappa \, \phi(x)}.\]
Of course, in order to do this we need to have names for all elements of $\kappa$, as well as $\kappa$ itself. To this effect, let $\alang^\kappa$ be a (possibly uncountable) extension of $\alang$ which contains one constant $c_\xi$ for each $\xi<\kappa$; to simplify notation, we may assume that $c_\xi=\xi$ and simply write the latter. Then, the $\kappa$-rule is readily applicable in any language extending $\lan\in^{\kappa+1}$. Similarly, for a theory $T$ over $\alang$, let $T^\kappa$ be the extension of $T$ over $\lan\in^{\kappa}$ with the axioms $\xi\in \zeta$ whenever $\xi<\zeta\leq\kappa$, and $\xi\not \in \zeta$ whenever $\zeta\leq\xi\leq\kappa$.

If $T$ is an extension of ${\rm ZFC}^\kappa$, we may enrich $T$ by operators of the form $\iprovx\lambda\kappa T\phi$, meaning that $\phi$ is provable using $\kappa$-rules of depth at most $\alpha$. Recall that if $\xi$ is an ordinal, then $\aleph_\xi$ denotes the $\xi^{\rm th}$ infinite ordinal. Then, any  infinite cardinal $\kappa$ may be represented in the form $\aleph_\beta$ for some $\beta$, and we write $\iprovx{\xi}{\beta}T\phi$ to state that $\phi$ may be proven by iterating $\aleph_\beta$-rules along $\xi$.

If we want the $\aleph$ function to be well-defined, we must work within a cardinal that is closed under $\xi\mapsto\aleph_\xi$. Fortunately, $\xi\mapsto\aleph_\xi$ is a normal function, so we may hyperate it, and readily observe that $\aleph^\omega (0)$ is the first ordinal $\xi$ such that $\aleph_\xi=\xi$. Thus we may assume that $T$ is an extension of ${\rm ZFC}^{\aleph^\omega (0) }$.

\begin{definition}
Let $T$ be a theory over $\lan\in^{\aleph^\omega (0) }$, $\alpha,\beta$ be ordinals, and $\phi\in \lan\in^{\aleph^\omega (0) }$. Then, by recursion on $\beta$ with a secondary recursion on $\alpha$, we define $\iprovx \alpha\beta T\phi$ to hold if either
\begin{enumerate}
\item $\nc_T \phi$, or

\item there are a formula $\psi(x)$ and ordinals $\gamma,\eta$ such that $\eta\leq \beta$ and either $\eta<\beta$ or $\gamma<\alpha$, and such that

\begin{enumerate}
\item for each $\delta < \aleph_ {\eta}$, $\iprovx\gamma{\eta}T{\psi(\delta)}$, and

\item $\nc_T \big ((\forall x < \aleph_\eta \, \psi(x))\to\phi \big )$.
\end{enumerate}
\end{enumerate}
\end{definition}

As was the case with $\omega$-rules, we have that for any $\beta$, the $\aleph_\beta$-rule saturates by $\aleph_{\beta+1}$:

\begin{theorem}\label{TheoSatAleph}
If $\iprovx \lambda\eta T\phi$ for arbitrary $\lambda$, then there is $\lambda'<\aleph_{\eta+1}$ such that $\iprovx {\lambda'} \eta T\phi$.
\end{theorem}

\proof
By induction on $\eta$ with a secondary induction on $\lambda$. If $\nc_T \phi$ holds then clearly $\iprovx 0\eta T\phi$. Otherwise, there are a formula $\psi(x)$ and ordinals $\gamma$ and $\delta\leq \eta$ such that either $\delta<\eta$ or $\gamma<\lambda$, and for each $\xi < \aleph_ {\delta}$, $\iprovx\gamma{\delta}T{\psi(\xi)}$ and $\nc_T \big ((\forall x < \aleph_\delta \, \psi(x))\to\phi \big )$.

By the induction hypothesis, for each $\xi<\aleph_\delta$ there is
\[\lambda_\xi<\aleph_{\delta+1}\leq \aleph_{\eta+1}\]
such that $\iprovx{\lambda_\xi}{\delta}T{\psi(\xi)}$. By Lemma \ref{LemmCardinal}, we have that
\[\lambda=\sup_{\xi<\aleph_\delta}\lambda_\xi<\aleph_{\eta+1},\]
and therefore also $\lambda+1 <\aleph_{\eta+1}$. But then observe that $\iprovx{\lambda+1}\eta T\phi$, as desired.
\endproof

Thus we have a similar situation as we had when considering $\langle \omega_1+\xi\rangle_T\phi$; any expressions of the form $\iconsx{\aleph_{\beta+1}+\alpha}\beta T\phi$ is equivalent to $\iconsx{\aleph_{\beta+1}}\beta T\phi$. Moreover, observe that $\iconsx{\aleph_{\beta+1}}\beta T\phi$ is in turn equivalent to $\iconsx{0}{\beta+1} T\phi$; thus we should only be interested in expressions of the form $\iconsx{\alpha}\beta T\phi$ in cases when $\alpha<\aleph_{\beta+1}$. Otherwise, as we did for impredicative worms, we may collapse $\alpha$ to an ordinal $\uppsi_\beta(\alpha)<\aleph_{\beta+1}$.

In Section \ref{SubsecCollAleph} we will review a version of Buchholz's ordinal notation system which achieves exactly that, and in Section \ref{SubsecCollSpiders} we will see how these ideas may be applied to {\em spiders,} which are similar to worms but based on modalities ${\alpha\bangle\beta}$. However, before we continue, we remark that working with uncountable languages has some obvious drawbacks. Fortunately, this can be avoided by working with admissible ordinals rather than cardinals.

\subsection{Iterated admissibles}

If we work with an uncountable language then the usual proof of the validity of
\[\iconsx 00 T\phi\rightarrow \iprovx 10 T\iconsx 00 T\phi\]
will not go through, given that we cannot code all possible derivations as natural numbers. There is more than one way to get around this problem; one can allow only ordinals appearing in $\phi$ to be used in a derivation of $\phi$, for example. Alternately, we can work with admissible ordinals, (many of) which are countable, instead of cardinals.

In the set-theoretical context, a {\em $\Delta_0$ formula} is any formula $\phi$ of $\lan\in$ such that all quantifiers appearing in $\phi$ are either of the form $\forall x\in y$ or $\exists x\in y$. Then, {\em Kripke-Platek set theory} is the subtheory $\rm KP$ of $\rm ZFC$ in which the axioms of choice, powerset and infinity are removed, and separation and collection are restricted to $\phi\in\Delta_0$.

With this in mind, we say that an ordinal $\alpha$ is {\em admissible} if ${\mathbb L}_\alpha$ (in G\"odel's constructible hierarchy) is a model of $\KP$. Admissible sets are studied in great detail in \cite{Barwise}. Moreover, an analogue of Theorem \ref{TheoSatAleph} also holds if we define:
\begin{enumerate}[label=(\roman*)]

\item $\omega^{CK}_ 0=\omega$,

\item $\omega^{CK}_ {\xi+1}$ to be the least admissible $\alpha$ such that $ \omega^{CK}_\xi < \alpha$, and

\item $\omega^{CK}_\lambda=\displaystyle\lim_{\xi<\lambda}\omega^{CK}_\xi$ for $\lambda$ a limit ordinal.

\end{enumerate}
This allows us to interpret $\iprovx\alpha\beta T$ using a countable language by replacing the $\aleph_\beta$-rule by the {\em $\omega^{CK}_\beta$-rule,}
\[\dfrac{\langle \phi(\xi)\rangle_{\xi<\omega^{CK}_\beta}}{\forall x< \omega^{CK}_\beta \, \phi(x)}.\]
Working with admissibles rather than cardinals makes the properties of collapsing functions more difficult to prove, but this has been done by Rathjen in \cite{Rathjen1993}. For simplicity, in this text we will continue to work with the $\aleph$-function.

\subsection{Collapsing the Aleph function}\label{SubsecCollAleph}

In this section we will review a variant of Buchholz's notation system of ordinal notations based on collapsing the aleph function \cite{Buchholz}. The ordinals obtained appear, for example, in the proof-theoretical analysis of the theories ${\rm ID}_\nu$ of iterated inductive definitions \cite{BookInductiveDefinitions}. Below, define $\Om\xi=-\omega+\aleph_\xi$; we will continue with this convention throughout the rest of the text.

\begin{definition}\label{DefPsiEta}
Given ordinals $\eta,\xi$, we simultaneously define the sets $\Cset_\eta(\xi)$ and the ordinals $\newpsi_\eta(\xi)$ by induction on $\xi$ as follows:
\begin{enumerate}
\item $\Cset_\eta (\xi)$ is the least set such that
\begin{enumerate}
\item $2+\Om\eta\subseteq \Cset_\eta(\xi)$;\label{DefPsiEtaItOne}

\item if $\alpha,\beta,\gamma \in \Cset_\eta(\xi)$ then $e^\alpha(\beta+\gamma)\in  \Cset(\xi)$, and

\item if $\alpha,\beta\in  \Cset_\eta(\xi)$ and $\beta<\xi$, then $\uppsi_{\alpha}(\beta)\in  \Cset_\eta(\xi)$;
\end{enumerate}
\item $\newpsi_\eta(\xi)=\min \{\xi : \xi\not\in \Cset_\eta(\xi)\}$.
\end{enumerate}
\end{definition}

Observe that \eqref{DefPsiEtaItOne} could be simplified somewhat if we had defined $\Om 0=2$, but our presentation will in turn simplify some expressions later. As before, it is possible to define $\Cset_\eta(\xi)$ using the notation of Definition \ref{DefFClose} and thus we can apply our previous work to these sets. Aside from the first item, which is easy to check, the following lemma summarizes the analogues of Lemmas \ref{LemmExpPsi}, \ref{LemmCClosedO}, and \ref{LemmPsiCountable}. The proofs are essentially the same and we omit them.

\begin{lemma}\label{LemmCProp}
Given ordinals $\eta,\mu$,
\begin{enumerate}

\item $\uppsi_{1+\eta}(0)=\Om{1+\eta}$;
\label{LemmCPropZero}

\item $\uppsi_\eta(\mu)$ is additively indecomposable and satisfies $e^{\uppsi_\eta(\mu)}1=\uppsi_\eta(\mu)$;
\item $C_\eta(\mu)$ is hyperexponentially perfect,

\item $\card{C_\eta(\mu)}=\Om\eta$, and

\item $\uppsi_\eta (\mu)\in[\Om\eta,\Om{\eta+1})$.\label{LemmCPropFour}

\end{enumerate}

\end{lemma}

The first ordinal that we cannot write using indexed collapsing functions is $\uppsi_0(\Upomega^\omega 1)$:

\begin{lemma}\label{LemmUpomegaBound}
Given ordinals $\eta<\Upomega^\omega 1$ and an arbitrary ordinal $\mu$,
\[\sup C_\eta(\mu)=\Upomega^\omega 1.\]
\end{lemma}

\proof
To see that $\sup C_\eta(\mu)\leq\Upomega^\omega 1,$ we observe that $\Upomega^\omega 1$ is closed under all of the operations defining $C_\eta ( \mu ) $:

Since $\eta<\Upomega^\omega 1$, we have that $\Om\eta\subseteq \Upomega^\omega 1$. By Lemmas \ref{LemmCardOrd} and \ref{LemmExpCard}, we see that if $\alpha,\beta,\gamma <\Upomega^\omega 1$, then
\[\kappa := \card{e^\alpha(\beta+\gamma)}\leq \max\{\omega,\card\alpha,\card\beta,\card\gamma\}.\]
We then have that $\kappa<\Upomega^\omega 1$, so writing $\kappa=\Om\xi$ for some $\xi<\Upomega^\omega 1$, we observe that \[e^\alpha(\beta+\gamma)  < \Om{\xi+1} < \Upomega^\omega 1.\]
Finally we note that if $\nu,\xi<\Upomega^\omega 1$, then by Lemma \ref{LemmCProp}.\ref{LemmCPropFour}, $\uppsi_\nu(\xi)<\Om{\nu+1}<\Upomega^\omega 1$.

Now, to see that
\[\sup C_\eta(\mu)\geq \Upomega^\omega 1,\]
simply consider the sequence $(\pi_n)_{n<\omega}$ given by $\pi_0=0$ and $\pi_{n+1}=\uppsi_{\pi_n}(0)\in C_\eta(\mu)$. By Lemma \ref{LemmCProp}.\ref{LemmCPropZero} we have that $\pi_{n+1}=\Om{\pi_n}$ which by Lemma \ref{LemmExpRec} converges to $\Upomega^\omega 1$.
\endproof

The ordinal $\uppsi_0(\Upomega^\omega 1)$ is also computable, but we will not prove this here; see e.g.~\cite{Buchholz} for details.
In the next section, we will present a variant of the functions $\uppsi_\nu$ using worm-like notations obtained from iterated $\aleph_\xi$-rules.

\subsection{Iterated Alephs and spiders}\label{SubsecCollSpiders}

We have seen in Theorem \ref{TheoAutWorm} that Beklemishev's autonomous worms give a notation system for all ordinals below the Feferman-Sch\"utte ordinal $\Gamma_0$, and in Theorem \ref{TheoImpAut} that impredicative worms extend this to all ordinals below $\uppsi(e^{\Omega+1}1)$ (which becomes $\uppsi_0 (e^{\uppsi_0(0)+1}1)$ in our version of Buccholz's notation). Now let us introduce spiders, which may be used to give notations for much larger ordinals than we could with worms.

\begin{definition}\label{DefSpider}
Let $\Lambda$ be either an ordinal or the class of all ordinals, and $f\colon\Lambda\to\Lambda$ be a normal function. We define $\Lambda \bangle f$ to be the class of all pairs of ordinals ${\lambda\bangle\mu}$ such that $f(\mu)+\lambda<f(\mu+1)$, and write $\Spiders^\Lambda_f$ for the set of all expressions of the form 
\[{\bm \lambda}_1\hdots {\bm \lambda}_n\top,\]
with each ${\bm \lambda}_i\in {\Lambda\bangle f}$. We simply write $\Spiders$ instead of $\Spiders^\ord_\Upomega$. Elements of $\Spiders$ are called {\em spiders.}
\end{definition}

We will restrict our attention to the case where $f(\xi) = \Om \xi= -\omega+\aleph_\xi$, although we state Definition \ref{DefSpider} with some generality to stress that there are other possible choices for $f$. In a way, spiders are simply a different way to represent worms; to pass from one representation to the other, we introduce two auxiliary functions.

\begin{definition}
Let $\alpha$ be any ordinal. Then, define
\begin{enumerate}[label=(\roman*)]

\item $\cindex\alpha$ to be the greatest ordinal such that $\Om{\cindex\alpha} \leq \alpha,$ and

\item $\dot\alpha=- \Om{\cindex\alpha} +\alpha$.

\end{enumerate}

\end{definition}

This definition is sound because for any normal function $f$ with $f(0)=0$ and any ordinal $\mu$, there is always a greatest ordinal $\xi$ such that $f(\xi)\leq \mu$. The `translation' between worms and spiders is the following:

\begin{definition}
Define:

\begin{enumerate}

\item $\flat\colon{\ord\bangle\Upomega} \to \ord$ by $\flatten {\lambda\bangle\mu}=\Om\mu+\lambda,$ and set ${\lambda\bangle\mu}\leq {\eta\bangle\nu}$ if and only if $\flatten{\lambda\bangle\mu}\leq \flatten{\eta\bangle\nu}.$ If $\sx={\bm \lambda}_1\hdots {\bm \lambda}_n\top\in \Spiders$, set $\flatten \sx=\flatten {\bm \lambda}_1\hdots\flatten {\bm \lambda}_n\top$.

\item $\sharp\colon \ord\to {\Lambda\bangle\Upomega}$ by $\sharp \lambda={{\dot\lambda}\bangle{\cindex{\lambda}}}.$ If $\fw={ \mu}_1\hdots { \mu}_n\top\in \Worms$, set $\sharpen \fw=\sharpen{\mu}_1\hdots\sharpen {\mu}_n\top$.

\end{enumerate}

\end{definition}

The following is then immediately verified:

\begin{lemma}\label{LemmInvFS}
The class functions $\flatten{}$ and $\sharpen{}$ are bijective and inverses of each other.
\end{lemma}

With this, we can extend our worm notation to spiders.

\begin{definition}
If $\sx\in\Spiders $, define
\begin{enumerate}

\item $O(\sx)=o(\flatten\sx)$,

\item $H(\sx)=h(\flatten \sx)$ and $B(\sx)=b(\flatten \sx)$,

\item $\sx\wle{}\sy$ if and only if $\flatten \sx\wle{}\flatten \sy$, and

\item  if $\mu$ is any ordinal, $\mu\promote\sx = \sharpen{(\mu\promote{\flatten \sx})}.$

\end{enumerate}

\end{definition}

Alternately, we can define the head and body of a spider without first turning them into worms:

\begin{lemma}\label{LemmSpiderHead}
Given a spider $\sx$, $H(\sx)$ is the maximum initial segment \[H(\sx)={\lambda_1\bangle \eta_1 }\hdots {\lambda_m\bangle \eta_m }\top\in\Spiders\]
of $\sx$ such that for all $i\in [1,m]$, either $\lambda_i\not=0$ or $\eta_i\not=0$.

If $H(\sx)=\sx$ then $B(\sx)=\top$, otherwise $B(\sx)$ is the unique spider such that
\[\sx=H(\sx)\textstyle{0\bangle 0} B(\sx).\]
\end{lemma}

As was the case with worms, the cardinality of $O(\sx)$ is easy to extract from $\sx$:

\begin{lemma}\label{LemmBoundSpid}
If
\[\sx={\lambda_1\bangle \eta_1 }\hdots {\lambda_n\bangle \eta_n }\top\in\Spiders,\]
then
\begin{enumerate}

\item for every $i\in[1,n]$, $\lambda_i,\eta_i \leq O(\sx)$, and\label{ItemBoundSpidOne}

\item if $\card{ O(\sx)}>\omega$, then $\card{O(\sx)}=\Om{\max_{i\in [1,n]}\eta_i}$.\label{ItemBoundSpidTwo}

\end{enumerate}

\end{lemma}

\proof
Immediate by applying Lemma \ref{LemmWormCard} to $\flatten\sx$ and observing that if $\mu>0$, $\card{\Om\mu+\lambda}=\Om\mu$ given that $\lambda<\Om{\mu+1}$.
\endproof

We can also give an analogue of $\sqsubset$ for spiders:

\begin{definition}
If
\[\sw={\lambda_1\bangle \eta_1 }\hdots {\lambda_n\bangle \eta_n }\top\in\Spiders \]
and $\Theta$ is a set of ordinals, we define $\sw\spidersub\Theta$ if each $\lambda_i,\eta_i\in\Theta$.
\end{definition}

With this, we are ready to `project' spiders.

\begin{definition}
Given $\sx,\sy\in\Spiders $, we define $U_\sy (\sx)\subseteq \ord$ and an ordinal $\cond_\sy\sx$ by induction on $\sx$ along $\wle{}$ as follows.
\begin{enumerate}

\item Let $U_\sy({\sx})$ be the least set of ordinals such that if
\[\su,\sv\spidersub \Om{ O(\sy) }\cup U_\sy({\sx})\]
and $\sv\wle{}\sx$, then $\cond_\su\sv \in U_\sy({\sx})$.

\item For any $\sy\in\Spiders $,
\begin{enumerate}

\item If $\sx\wle{} {0 \bangle { O(\sy) +1}}\top$, set $\cond_\sy(\sx)= O(\sx)$;

\item otherwise,
\[\cond_\sy(\sx)=\min \{\xi : \xi\not\in  U_\sy(\sx) \}.\]

\end{enumerate}

\end{enumerate}

\end{definition}

In the remainder of this section, we will see that the functions $\cond_\sx$ behave very similarly to the functions $\uppsi_\nu$. We begin with a simple lemma.

\begin{lemma}\label{LemmSpiderZeroOne}
If $\sx,\sy$ are spiders with $O(\sx)>1$, then $0,1\in U_\sy(\sx)$.
\end{lemma}

\proof
Immediate from observing that $0= O( \top) =\cond_\top \top$ and $1=O({0\bangle 0}\top)=\cond_\top \left ( {0\bangle 0}\top \right ) $.
\endproof

With the next few lemmas, we show that the elements of $U_\sy(\sx) \cap O(\sx)$ can be characterized as the order-types of suitable spiders.
In the process, we obtain some useful properties of $\cond_\sy \sx$.

\begin{lemma}\label{LemmOIsSpider}
If $\sx\spidersub U_\sv(\sw)$ and $\sx\wle{}\sw$, then $ O(\sx) \in U_\sv(\sw)$.
\end{lemma}

\proof
Let $\sx\spidersub U_\sv(\sw)$ be such that $\sx \wle{} \sw$. Since $\Upomega$ is normal, for every $\xi$ we have that $\xi\leq \Om\xi$. In particular,
\[ O( \sx)< O(\sx)+1\leq \Om{ O( \sx) +1}.\]

It follows that $ O( \sx) =\cond_{\sx}\sx \in U_\sv(\sw)$.
\endproof

With this, we can show that $\cond_\sy \sx$ has cardinality $\Om{O ( \sy ) }$, provided $O(\sx) $ is large enough.

\begin{lemma}\label{LemmCondBound}
If $\sx,\sy$ are spiders with $O(\sx) \geq \Om{O(\sy) + 1}$, then
\[\cond_\sy \sx \in \big [\Om{ O(\sy) } , \Om{ O(\sy) + 1} \big ).\]
\end{lemma}

\proof
If $\xi<\Om{O(\sy)}$, then by Corollary \ref{CorSurj} we obtain $\fw\sqsubset\Om{ O( \sy) }$ such that $o(\fw)=\xi$ and observe that $\sharpen \fw \wle{} \sx$, so that by Lemma \ref{LemmOIsSpider}, $\xi = O( \sharpen\fw) \in U_\sy (\sx)$. It follows that $\cond_\sy \sx \geq \Om{ O( \sy) }$. Meanwhile, by Lemma \ref{LemmCardFClose}, $\card{U_\sy(\sx)} \leq \omega + \Om{ O( \sy ) }$, so $\cond_\sy \sx < \Om{ O(\sy) + 1} $.
\endproof

Moreover, $\cond_\sy\sx$ satisfies an analogue of Lemma \ref{LemmCondFix}:

\begin{lemma}\label{LemmCondsFixedAleph}
If $O(\sx) \geq \Om{ O(\sy) +1}$, then $O \left ( {{\cond_\sy\sx}\bangle {O(\sy)}}\top \right ) =\cond_\sy\sx$.
\end{lemma}

\proof
Analogous to the proof of Lemma \ref{LemmCondFix}, except that to reach a contradiction we use Lemma \ref{LemmBoundSpid}.\ref{ItemBoundSpidOne} to obtain a spider $\sv$ such that $O(\sv)= \cond_\sy\sx$ and all of whose entries are strictly bounded by $\cond_\sy\sx$.
\endproof

With this we can show that the elements of $U_\sy(\sx) \cap O(\sx)$ are the order-types of suitable spiders, as claimed.

\begin{lemma}\label{LemmSpiderIsO}
Let $\sx,\sy$ be spiders and $\xi$ an ordinal. Then, $\xi\in U_\sy(\sx) \cap O(\sx)$ if and only if there is $\sw\spidersub U_\sy(\sx) \cap O(\sx)$ such that $\xi = O(\sw)$.
\end{lemma}

\proof
One direction is Lemma \ref{LemmOIsSpider}. For the other, if $\xi\in U_\sy(\sx)\cap O(\sx)$, then there are $\su,\sv \spidersub U_\sy (\sx)$ such that $\su \wle{} \sx$ and $\xi=\cond_\sv\su$. If $\su\wle{}{0 \bangle {O(\sv)+1}}\top$, then we already have $\xi=O(\su)$. If not, by Lemma \ref{LemmCondBound}, \[O(\sv) \leq \Om{O(\sv)} \leq \xi < O(\sx),\]
so that by Lemma \ref{LemmOIsSpider}, $O(\sv) \in U_\sy (\sx)$, and hence ${{\xi}\bangle{O(\sv)}} \top \spidersub U_\sy (\sx)\cap O(\sx)$.

Meanwhile, by
Lemma \ref{LemmCondsFixedAleph},
\[O \left ( {{\cond_\sv\su}\bangle{O(\sv)}}\top \right ) =\cond_\sv\su=\xi , \]
as needed.
\endproof

Lemma \ref{LemmSpiderIsO} is useful in showing that $U_\sy(\sx)$ is well-behaved. For example, it satisfies a bounded version of additive reducibility.

\begin{lemma}\label{LemmCantorSpider}
Given spiders $\sx,\sy$ and an additively decomposable ordinal $\xi<O(\sx)$, we have that $\xi \in U_\sy(\sx)$
if and only if there are $\alpha,\beta \in U_\sy(\sx) \cap \xi$ such that $\xi = \alpha + \beta$.
\end{lemma}

\proof
Analogous to the proof of Lemma \ref{LemmSum}. To illustrate, let us check that if $\xi\in U_\sy(\sx) \cap O(\sx)$ is additively decomposable, then there are $\alpha,\beta\in U_\sy(\sx)\cap O(\sx)$ such that $\xi=\alpha+\beta$.
Using Lemma \ref{LemmSpiderIsO}, write $\xi=O(\sw)$ with $\sw\spidersub U_\sy(\sx)$. Then, by Theorem \ref{TheoTranOrder},
\[\xi=O(\sw)=o(\flatten\sw)=ob(\flatten\sw)+1+o h(\flatten \sw).\]
Set $\alpha= ob(\flatten\sw)$ and $\beta=1+o h(\flatten \sw)$. Observe that $\alpha<\xi$, while $\beta$ is additively indecomposable so $\beta\not=\xi$. Hence, $\alpha,\beta<\xi$.

Finally, observe that $H(\sw),B(\sw)\spidersub U_\sy(\sx)$,
\[\beta=1+oh(\flatten\sw)=o((\flatten H(\sw))0)=O \left (H(\sw)\textstyle{0\bangle 0} \right ),\]
and $H(\sw)\spidersub U_\sy(\sx)$; similarly, $\alpha=O B(\sw),$ so $\alpha,\beta\in U_\sy(\sx)$.
\endproof

Note that $U_\sy(\sx)$ is not necessarily additively reductive; howerer, this truncated form of additive reducibility is sufficient to obtain the conclusion of Lemma \ref{LemmSumCantor}:

\begin{lemma}\label{LemmSumCantorTrunc}
Let $\Theta$ be a set of ordinals such that $0\in\Theta$, and $\lambda$ be an ordinal such that, whenever $\xi < \lambda$ is additively reducible, then $\xi \in \Theta$ if and only if there are $\alpha,\beta < \xi$ such that $\alpha + \beta = \xi$. Then, for any ordinal $\xi$:

\begin{enumerate}

\item if $0\not=\xi\in\Theta \cap \lambda $, there are ordinals $\alpha,\beta$ such that $\alpha,\omega^\beta\in \Theta$ and $\xi=\alpha+\omega^\beta$;

\item if $\beta\in\Theta \cap \lambda$ and $\alpha<\beta$  (not necessarily a member of $\Theta$), then $-\alpha+\beta\in \Theta$.\label{LemmSumCantorMinusTrunc}

\end{enumerate}
\end{lemma}

The proof is identical to that of Lemma \ref{LemmSumCantor} and we omit it.
Next we see that the sets $U_\sy(\sx)$ are also closed under some operations related to cardinality.

\begin{lemma}\label{LemmClosedCard}
If $\sx,\sy$ are worms and $\xi\in U_\sy(\sx) \cap \Om{\sx}$, then:
\begin{enumerate}

\item
 $\cindex\xi \in U_\sy(\sx)$;
 
 \item 
 
 if moreover $\Om\xi < O(\sx)$, then $\Om\xi \in U_\sy(\sx)$.
 
\end{enumerate}
\end{lemma}

\proof
For the first claim, if $\xi$ is at most countable, $\cindex\xi=0\in U_\sy(\sx)$. If not, by Lemma \ref{LemmSpiderIsO}, $\xi=O(\sw)$ for some $\sw\spidersub U_\sy(\sx)$, and by \ref{LemmBoundSpid}.\ref{ItemBoundSpidTwo}, $\eta=\cindex\xi$ occurs in $\sw$, hence $\eta\in U_\sy(\sx)$.

For the second, we observe that ${0 \bangle {O(\sw) } } \top \wle{} {0 \bangle {O(\sw) + 1} } \top$, so that
\[\Om\xi = O\left ({0 \bangle {O(\sw) } } \top \right ) = \cond_\sw \left ( {0 \bangle {O(\sw) } } \top \right ).\]
If we moreover have $\Om\xi < O(\sx)$, this gives us $\Om\xi\in U_\sy(\sx)$.
\endproof

The following lemmas show that our work on worms can be used to study the sets $ U_\sy(\sx)$.

\begin{lemma}\label{LemmSpiderIffFlat}
Given spiders $\sw,\sx,\sy$,
\[\sw\spidersub U_\sy(\sx) \cap O(\sx)\]
if and only if
\[\flatten\sw\sqsubset U_\sy(\sx) \cap O(\sx).\]
\end{lemma}

\proof
Let
\[\sw={\lambda_1\bangle \eta_1 }\hdots {\lambda_n\bangle \eta_n }\top\in\Spiders.\]
If $\sw\spidersub U_\sy(\sx) \cap O(\sx)$, then each $\lambda_i,\eta_i\in U_\sy(\sx)  $. By Lemma \ref{LemmClosedCard}, $\Om{\eta_i}\in U_\sy(\sx) $, and by Lemma \ref{LemmCantorSpider}, $\Om{\eta_i}+\lambda_i\in U_\sy(\sx)$. Since
\[\Om{\eta_i}, \Om{\eta_i} + \lambda_i \leq O(\sw) < O(\sx),\]
it follows that $\flatten\sw\sqsubset U_\sy(\sx) \cap O(\sx) $.

Conversely, if $\flatten\sw\sqsubset U_\sy(\sx) \cap O(\sx) $, write $ \flatten\sw=\mu_1\hdots\mu_n\top$. By Lemma \ref{LemmClosedCard}, $\cindex{\mu_i}\in U_\sy(\sx)$, and by Lemma \ref{LemmCantorSpider} together with Lemma \ref{LemmSumCantorTrunc}.\ref{LemmSumCantorMinusTrunc}, $\dot \mu_i\in U_\sy(\sx)$. It follows from Lemma \ref{LemmInvFS} that
\[\sw=\sharpen{\flatten\sw}\spidersub U_\sy(\sx) \cap O(\sx) \]
by observing that $\cindex{\mu_i},\dot \mu_i \leq \mu_i < O(\sx)$.
\endproof

With this we see that the sets $U_\sy(\sx) \cap O(\sx)$ are almost worm-perfect.

\begin{theorem}\label{TheoSpiderPerfect}
Given spiders $\sx,\sy$ and an ordinal $\xi<O(\sx)$, $\xi \in U_\sy(\sx) \cap O(\sx)$ if and only if there is $\mathfrak z\sqsubset U_\sy(\sx) \cap O(\sx)$ with $\xi = o(\mathfrak z)$.
\end{theorem}

\proof
Given an ordinal $\xi$, by Lemma \ref{LemmSpiderIsO}, $\xi\in U_\sy(\sx) \cap O(\sx) $ if and only if there is $\sz\spidersub U_\sy(\sx) \cap O(\sx)$ with $\xi=O(\sz)$. But by Lemma \ref{LemmSpiderIffFlat}, by setting ${\mathfrak z}=\flatten\sz$ we see that this is equivalent to there existing ${\mathfrak z}\sqsubset U_\sy(\sx) \cap O(\sx) $ with $\xi=o({\mathfrak z})$. 
\endproof

As a consequence, we obtain that $U_\sy(\sz)$ is closed under bounded hyperexponentiation.

\begin{lemma}\label{LemmSpidersClosedHE}
If $\sx,\sy$ are worms and $\alpha,\beta \in U_\sy(\sx)$ are such that $e^\alpha \beta < O(\sx)$, then $e^\alpha \beta \in U_\sy(\sx)$.
\end{lemma}

\proof
We may assume that $0<\alpha,\beta < e^\alpha \beta$, so that if $e^\alpha \beta < O(\sx)$, then $\alpha,\beta<O(\sx)$. By Theorem \ref{TheoSpiderPerfect}, $\beta = o(\fv)$ for some $\fv = \lambda_1 \hdots \lambda_n \top \sqsubset U_\sy(\sx) \cap O(\sx)$. Since $e^\alpha \beta$ is additively indecomposable, for each $i \in [1,n]$, $\alpha + \lambda_i \leq \alpha + \beta < e^\alpha \beta$, hence by Lemma \ref{LemmCantorSpider}, $\alpha + \lambda_i \in U_\sy(\sx)$. Thus $\fw = \alpha\uparrow \fv \sqsubset U_\sy(\sx)$, and $o(\fw) = e^\alpha \beta$, which by Theorem \ref{TheoSpiderPerfect} implies that $e^\alpha \beta \in U_\sy(\sx)$.
\endproof

This tells us that, below $O(\sx)$, the sets $U_\sy(\sx)$ behave very similar to the sets $C_\eta(\lambda)$. Conversely, we can prove that the sets $C_\eta(\lambda)$ are `spider-perfect'.

\begin{lemma}\label{LemmCSpiderPerfect}
If $\eta,\lambda$ are ordinals and $\sw\spidersub C_\eta(\lambda)$, then $O(\sw)\in C_\eta(\lambda)$.
\end{lemma}

\proof
Suppose that $\sw\spidersub C_\eta(\lambda)$ and $\sw\wle{}\sx$. The set $C_\eta(\lambda)$ is closed under $\Om\cdot$ and addition, so from $\sw\spidersub C$ we obtain $\flatten\sw\sqsubset C$. But $C_\eta(\lambda)$ is hyperexponentially perfect, thus by Lemma \ref{LemmClosedSum} it is worm-perfect.  We conclude that $O(\sw)=o(\flatten\sw)\in C.$
\endproof

Thus the functions $\cond_\sy$ should closely mimic the functions $\uppsi_\eta$. However, a full translation between the two systems would go beyond the scope of the current work. Instead, we conclude with a conjecture. 

\begin{conj}\label{ConjSpiders}
$\uppsi_0{\Upomega^\omega 1} = \cond_ \top \left ({0\bangle{\Upomega^\omega 1}}\top \right )$.
\end{conj}

\subsection{Autonomous spiders and ordinal notations}

We can use autonomous spiders to produce an ordinal notation system, similar to Beklemishev's autonomous worms. We define them as follows:

\begin{figure}
\[\term{\term{}{\term{}{}}}{\term{\term{}{}}{}}\]
\caption{An autonomous spider.}
\end{figure}

\begin{definition}
We define the set of {\em autonomous spiders,} $\SpAut$, to be the least set such that:
\begin{enumerate}
\item $\top\in \SpAut$;
\item if ${\aX},{\aY},{\aZ}\in \SpAut$, then $\term {\aX} {\aY}{\aZ}\in \SpAut$.

\end{enumerate}
\end{definition}

As with autonomous worms, each autonomous spider can be interpreted as a `real' spider.

\begin{definition}
We define a function $\spiderfun{\cdot}\colon \SpAut \to \Spiders$ by \begin{enumerate}
\item $\spiderfun{\top}=\top$,
\item $ \spiderfun{\left (\term{\aX}{\aY}{\aZ} \right ) }={{\cond_{\spiderfun{\aY}}\spiderfun{\aX}}\bangle {{O(\spiderfun{\aY})}}}{\spiderfun{\aZ}}$.
\end{enumerate}
For ${\aX},{\aY}\in \SpAut$ we set ${\sf O}({\aX}) = O (\spiderfun{{\aX}})$ and ${\sf p}_{\aY}{\aX}=\cond_{\spiderfun{\aY}}\spiderfun{\aX}$.

\end{definition}

We will often omit writing $\top$, so that for example $\term{}{}$ denotes $\term \top\top\top$.
The proofs of the following two results are analogous to those of Lemma \ref{LemmImpAut} and Theorem \ref{TheoImpAut}, respectively, and we omit them.

\begin{theorem}\label{TheoLast}
For any $\xi\in U_\top \left ( {{0} \bangle {\Upomega^\omega(0)}}\top \right )$, there exists ${\aX}\in\SpAut$ such that $\xi={\sf O}({\aX})$.
\end{theorem}

Thus assuming Conjecture \ref{ConjSpiders}, the autonomous spiders indeed provide a notation system for all ordinals below $\uppsi_{0}{\Upomega}^\omega 1$, along with some uncountable ordinals.

\section{Concluding remarks}\label{SecConc}

We have developed notation systems for impredicative ordinals based on reflection calculi, thus providing a positive answer to Mints' and Pakhomov's question. These notation systems are obtained by considering strong provability operators extending a theory $T$. In the process, we have also given a general overview of existing notation systems based on worms. 

This work is still exploratory and further developments are required to fully flesh out our proposal. First, no decision procedure is given to determine whether ${\sf O}(\aw) < {\sf O}(\av)$ when $\aw$, $\av$ are impredicative autonomous worms or spiders. While such a decision procedure might be extractable from Theorem \ref{TheoWormOrder} together with procedures for more standard systems based on $\uppsi$, it would be preferable to provide deductive calculi in the style of $\sf RC$. Second, the set-theoretic interpretations sketched in Section \ref{SecSpiders} are only tentative and require a rigorous treatment. I'll leave both of these points for future work.

The ultimate goal of the efforts presented here are for the computation of $\Pi^0_1$ ordinals of strong theories of second-order arithmetic. There are many more hurdles to overcome before attaining such a goal, but hopefully the ideas presented here will help to lead the way forward.

\subsection*{Acknowledgements}

I would like to take this opportunity to express my gratitude to Professor Mints not only for suggesting the topic of this paper, but also for his inspiration and support as a doctoral advisor. His passing was a great personal loss and a great loss to logic. I would also like to thank Fedor Pakhomov for bringing up the same issue and for many enlightening discussions; Lev Beklemishev and Joost Joosten for introducing me to the world of worms, and for many useful comments regarding this manuscrupt; and Andr\'es Cord\'on-Franco, F\'elix Lara-Mart\'in, as well as my student Juan Pablo Aguilera, for their contributions to the results reviewed here, and Ana Borges for her sharp eye spotting errors in an earlier draft.

Finally, I would like to thank the John Templeton Foundation and the Kurt G\"odel Society for the support they have given myself and other logicians through their fellowship program. Their effort is a great boost to logic worldwide; let us hope that it continues to encourage many more generations of logicians.

\begin{figure}

\begin{center}
\scalebox{.9}{$
\begin{array}{rlrlrl}
\omega&
\term{\term{}{}}{}&
\varepsilon_0&
\term{\term{\term{}{}}{}}{}&
\Gamma_0&
\term{\term{}{\term{}{}}}{}
\end{array}
$
}
\end{center}

\caption{Some familiar ordinals represented as autonomous spiders.}
\end{figure}

%\bibliographystyle{plain}
%\bibliography{References}

\end{document}